\documentclass[a4paper]{article}

\usepackage[bottom=2in]{geometry} 

\usepackage{amsmath}
\usepackage{amsthm}
\usepackage{amssymb}
\usepackage{color,todonotes}
\usepackage{cancel}
\usepackage[normalem]{ulem}
\usepackage{comment}
\usepackage[utf8]{inputenc}
\usepackage{hyperref}
\hypersetup{
	unicode,
	pdfauthor={Author One, Author Two, Author Three},
	pdftitle={A simple article template},
	pdfsubject={A simple article template},
	pdfkeywords={article, template, simple},
	pdfproducer={LaTeX},
	pdfcreator={pdflatex}
}

\usepackage[sort&compress,numbers,square]{natbib}
\theoremstyle{plain}
\newtheorem{theorem}{Theorem}[section] 
\newtheorem{corollary}[theorem]{Corollary}
\newtheorem{lemma}[theorem]{Lemma}
\newtheorem{proposition}[theorem]{Proposition}
\newtheorem{assumption}[theorem]{Assumption}
\newtheorem{remark}[theorem]{Remark}

\theoremstyle{definition}
\newtheorem{definition}[theorem]{Definition}
\newtheorem{example}[theorem]{Example}

\usepackage[normalem]{ulem}

\usepackage{graphicx, color}
\graphicspath{{fig/}}

\usepackage{algorithm, algpseudocode} 
\usepackage{mathrsfs} 
\usepackage{bm}
\def\R{{\mathbb R}}
\def\Ab{{\mathsf A}}
\def\Bb{{\mathsf B}}
\def\Cb{{\mathsf C}}

\def\Eb{{\mathsf E}}

\def\Gb{{\mathsf G}}

\def\Ib{{\mathsf I}}

\def\Mb{{\mathsf M}}

\def\Qb{{\mathsf Q}}
\def\Rb{{\mathsf R}}
\def\Sb{{\mathsf S}}

\def\Ub{{\mathsf U}}
\def\Vb{{\mathsf V}}
\def\Wb{{\mathsf W}}

\def\Yb{{\mathsf Y}}
\def\Zb{{\mathsf Z}}
\def\T{^{\top}}

\def\aa{{\mathbf a}}

\def\cc{{\mathbf c}}

\def\e{{\mathbf e}}

\def\h{{\mathbf h}}

\def\s{{\mathbf s}}

\def\u{{\mathbf u}}
\def\v{{\mathbf v}}
\def\w{{\mathbf w}}
\def\x{{\mathbf x}}
\def\y{{\mathbf y}}
\def\z{{\mathbf z}}

\def\mmu{{\bm  \mu}}
\def\sSigma{{\bm \Sigma}}
\def\svec{\operatorname{svec}}
\def\PP{{\mathbb P}}
\def\EE{{\mathbb E}}
\def\Var{{\mathbb V}{\rm ar}}
\def\Cov{{\mathbb C}{\rm ov}}

\def\DDD{{\mathcal D}}

\def\T{^{\top}}

\newcommand{\norm}[1]{\left\lVert#1\right\rVert}

\usepackage{float}

\usepackage{enumitem}

\usepackage{bbm}

\usepackage{subcaption} 
\usepackage{graphicx} 

\usepackage{booktabs} 
\usepackage{makecell} 
\usepackage{array} 
\usepackage{multirow} 

\iftrue
\newcommand{\comments}[1]{\footnote{\textcolor{blue}{\textit{#1}}}}

\newcommand{\HZ}[1]{\textcolor{violet}{#1}}
\newcommand{\MB}[1]{\textcolor{red}{#1}}
\newcommand{\DV}[1]{\textcolor{blue}{#1}}
\else
\newcommand{\comments}[1]{}

\newcommand{\HZ}[1]{#1}
\newcommand{\MB}[1]{#1}
\newcommand{\DV}[1]{#1}
\fi

\title{Decision-dependent distributionally robust standard quadratic optimization with Wasserstein ambiguity}
\author{Immanuel M. Bomze\thanks{Research Network Data Science and Faculty of Mathematics, University of Vienna;\\
\hbox{}\qquad L2S, Paris Saclay University, CNRS, CentraleSupélec; and 
CCOR/CIAS, Corvinus University Budapest;\\
\hbox{}\qquad postal address: Oskar-Morgenstern-Platz 1, A-1090 Vienna, Austria. E-mail: immanuel.bomze@univie.ac.at} \and Daniel de Vicente\thanks{University of Vienna, Oskar-Morgenstern-Platz 1, A-1090 Vienna, Austria;\\
\hbox{}\qquad E-mail: 
a12032143@unet.univie.ac.at}
\and Abdel Lisser\thanks{Université Paris Saclay, CNRS, CentraleSupélec, Laboratoire des Signaux et des Systèmes, \\
\hbox{}\qquad Bat Breguet, 3 Rue Joliot Curie, F-91190 Gif-sur-Yvette, France. E-mail: abdel.lisser@l2s.centralesupelec.fr}
\and Heng Zhang\thanks{Université Paris Saclay, CNRS, CentraleSupélec, Laboratoire des Signaux et des Systèmes, \\
\hbox{}\qquad Bat Breguet, 3 Rue Joliot Curie, F-91190 Gif-sur-Yvette, France. E-mail:
heng.zhang@centralesupelec.fr}
}

\date{\today}

\begin{document}
	\maketitle
	\begin{abstract}
The standard quadratic optimization problem (StQP) consists of minimizing a quadratic form over the standard simplex. Without assuming convexity or concavity of the quadratic form, the StQP is NP-hard. This problem has many interesting applications ranging from portfolio optimization to machine learning.
Sometimes, the data matrix is uncertain but some information about its distribution can be inferred, e.g. a distance to a reference distribution (typically, the empirical distribution after sampling). In distributionally robust optimization, the goal is to hedge against the worst case of all possible distributions in an ambiguity set, defined by  above mentioned distance. In this paper we will focus on distributionally robust StQPs under Wasserstein distance, and show equivalence to an accordingly modified deterministic instance of an StQP. This blends well into recent findings for other approaches of StQPs under uncertainty. We will also address out-of-sample performance guarantees. Carefully designed experiments shall complement and illustrate the approach.
	\end{abstract}
\textbf{Keywords:} 
transportation-information inequality, Gaussian Orthogonal Ensemble, Wishart Ensemble, maximum weighted clique problem 

\section{Introduction}\label{sec:1.intro}

\subsection{Motivation and historical remarks}
The standard quadratic optimization problem (StQP) consists of minimizing a quadratic form over the standard simplex
\begin{equation}\label{StQP}\tag{StQP}
\min_{\x \in \Delta} \,\x^{\top}\Qb\x
\end{equation}
where $\Delta:= \{\x \in \mathbb{R}^n_+: \e^{\top}\x = 1\}$ is the standard simplex in $\mathbb{R}^n$ and $\Qb$ is a symmetric $n\times n$ matrix (denoted by $\Qb\in \mathcal S^n$ in the sequel). Here $\e\in \mathbb{R}^n$ is the vector of all ones and ${}^{\top}$ denotes transposition. Random parameters are denoted with a tilde sign. The objective function is already in general form since any general quadratic objective function $q(\x):=\x^{\top}\Ab\x + 2\cc^{\top}\x$ can be written in homogeneous form  $q(\x):= \x^{\top}\Qb\x$ by defining  $\Qb:= \Ab + \cc\e^{\top} + \e\cc^{\top}\in \mathcal S^n$.

\noindent Even though the StQP seems rather easy —minimization of a quadratic function under linear constraints — it is NP-hard without assumptions on the definiteness of the matrix $\Qb$. 
{In \cite{motzkin-straus-1965}, the authors} showed that the maximum clique problem, a well-known NP-hard problem, can be reduced to an StQP. Hence, the StQP is often regarded as the simplest of hard problems \cite{bomze2018complexity} since it contains the simplest non-convex objective function which is a quadratic form, and the simplest polytope as feasible set. The StQP is a very flexible optimization class that allows for modelling of diverse problems such as portfolio optimization problems  \cite{markowitz-1952}, pairwise clustering \cite{pavan2003dominant} and replicator dynamics \cite{bomze1998standard}.

\noindent The only data required to fully characterize an StQP is the data matrix $\Qb$. However, in many applications the matrix $\Qb$ is uncertain. StQPs with uncertain data have been explored in the literature.

\noindent One of the most natural ways to deal with uncertain objective functions is the worst-case approach, where there is no information on the distribution of the data matrix $\Qb$ except for its support. 

\noindent {The concept of a robust standard quadratic optimization problem was introduced in \cite{bomze2021trust}, where it was formulated as a maximin problem. In that work, various uncertainty sets for the uncertain matrix were investigated, and it was shown that the copositive relaxation gap coincides with the minimax gap. The authors also demonstrated that the robust StQP reduces to a deterministic StQP.}
\noindent {A setting in which the uncertain data matrix $\widetilde \Qb$ 
 contains a deterministic principal submatrix while the remaining entries are randomly distributed according to a known probability distribution was investigated in \cite{bomze2022two}. In that model, the decision-maker makes a first decision before the uncertainty is revealed and a second decision afterwards. The decision variable is therefore divided into first-stage and second-stage components, and the resulting StQP is formulated as a two-stage stochastic optimization problem, which does not admit a deterministic StQP counterpart.}
\noindent The concept of a chance-constrained StQP was introduced in \cite{bomze2025uncertain}, where it was formulated as the minimization of the Value-at-Risk of the random variable $\x^{\top}\widetilde\Qb\x$. Under suitable distributional assumptions, the authors showed that this model is equivalent to a deterministic StQP.

\noindent By contrast, in this paper we will tackle the StQP with uncertain data matrix $\widetilde \Qb$ via distributionally robust optimization (DRO) with Wasserstein ambiguity. DRO is a modeling framework for optimization problems under uncertainty where the underlying probability distribution is not fully known and only assumed to lie within an ambiguity set of probability distributions. A DRO problem is thus a minimax game where the decision player plays against the ``nature" which governs the underlying probability distribution, see e.g. 
\cite{lin2022distributionally, rahimian2022frameworks}. 
Distributionally robust optimization arguably started with 
\cite{scarf} who proposed a minimal problem for selecting the optimal quantity of an item, where the full distribution of the demand was unknown and only the mean and standard deviation were known. This type of ambiguity would later be called moment-based ambiguity and was also used in many other works, e.g. 
\cite{vzavckova1966minimax} and 
\cite{delage2010distributionally}. 
Historically for the first time, \cite{pflug2007ambiguity} incorporated  the 1-Wasserstein distance to a distributionally robust optimization problem. This was continued in 
\cite{wozabal2012framework,pflug2014multistage}. Later,  
\cite{mohajerin2018data,zhao2018data} 
showed that for distributionally robust optimization problems with 1-Wasserstein ball centered at the empirical distribution the worst-case distribution can be explicitly constructed. Those results were generalized in 
\cite{gao2023distributionally,blanchet2019quantifying} 
These papers establish strong duality results for distance-based cost functions~\cite{gao2023distributionally} 
on $p$-Wasserstein balls with $p\geq 1$ or 
similar strong duality results for general lower semi-continuous cost functions.

\noindent Although the StQP features a non-convex objective, most tractability results in distributionally robust optimization hinge on convex or linear structure. This is reflected in the unified frameworks of 
\cite{delage2010distributionally} and 
\cite{wiesemann2014distributionally}, later extended to two-stage linear models by 
\cite{hanasusanto2015k}. Addressing the central issue of out-of-sample guarantee, 
\cite{mohajerin2018data} leveraged measure-concentration results of 
\cite{fournier2015rate} to guide the choice of the ambiguity radius. To reduce the conservatism implied by these bounds, later work proposed data-driven calibration via resampling, including the bootstrap-based methods of 
\cite{bertsimas2022bootstrap} and empirical likelihood techniques developed by 
\cite{lam2016robust}. More recently, 
\cite{blanchet2019quantifying} and 
\cite{gao2023finite} strengthened these guarantees by exploiting optimal-transport geometry and additional problem structure to alleviate the curse of dimensionality.

\subsection{Notations, contributions  and organization of the paper}

\subsubsection*{Norms and sets.}
An arbitrary norm on $\mathbb{R}^m$ is denoted by $\|\cdot\|$, while $\|\cdot\|_q$ denotes the specific $\ell_q$-norm for $1 \le q \le \infty$. 
The unit sphere in $\mathbb{R}^n$ is
\[
\mathbb{S}^{n-1} := \{ \y \in \mathbb{R}^n : \|\y\|_2 = 1 \}.
\]
For a set $\mathcal C \subseteq \mathbb{R}^m$, we denote its boundary, closure, conic hull, and normal cone at $\z \in \mathbb{R}^m$ by $\partial C$, $\mathrm{cl}(\mathcal C)$, 
$$\mathrm{cone}(\mathcal C):= \{\y = \lambda \v: \lambda \geq 0, \v \in \mathcal C\}$$
and 
\begin{equation}\label{normal cone of M set}
\mathcal N_{\mathcal C}(\z):=\{\v \in \mathbb R^m: \v^{\top}\z \geq \v^{\top}\y\,,\text{ for all } \y \in \mathcal C\}\,,
\end{equation}
respectively.
\subsubsection*{Probability measures and divergences.}
Let $\mathcal{P}(\mathbb{R}^m)$ be the set of all Borel probability measures on $\mathbb{R}^m$. 
For $p \ge 1$, define
\[
\mathcal{P}_p(\mathbb{R}^m)
:=
\left\{
\mathbb P \in \mathcal{P}(\mathbb{R}^m)
:
\mathbb{E}_{\mathbb P}[\|\boldsymbol{\tilde{\xi}}\|^p] < \infty
\right\}.
\]
For $\boldsymbol{\xi} \in \mathbb{R}^m$, the Dirac measure located at $\boldsymbol{\xi}$ is denoted by $\delta_{\boldsymbol{\xi}}$.
For $\{\mathbb P, \mathbb P^{\prime} \}\subset \mathcal{P}(\mathbb{R}^m)$, the Wasserstein distance and Kullback--Leibler divergence are denoted by $W_p(\mathbb P,\mathbb P')$ and $D(\mathbb P\|\mathbb P')$, respectively. Absolute continuity is denoted by $\mathbb P \ll \mathbb P'$. 

\subsubsection*{Data and empirical quantities.}
Let $N \in \mathbb{N}$ be the sample size and $[N] := \{1,\dots,N\}$. 
Given a sample $\{\boldsymbol{\widehat{\xi}}_i\}_{i=1}^N \subset \mathbb{R}^m$, the empirical distribution and sample mean are
\[
\widehat{\mathbb P}_N := \frac{1}{N}\sum_{i=1}^N \delta_{\boldsymbol{\widehat{\xi}}_i},
\qquad
\overline{\boldsymbol{\xi}} := \frac{1}{N}\sum_{i=1}^N \boldsymbol{\widehat{\xi}}_i.
\]
\subsubsection*{Probability distributions and random matrices.}
We write $\mathrm{Exp}(\lambda)$ for the exponential distribution with parameter $\lambda$, 
$\chi^2(k)$ for the chi-squared distribution with $k$ degrees of freedom, 
$\mathcal{N}(\mu,\sigma^2)$ for the univariate normal distribution with mean $\mu$ and variance $\sigma^2$, and 
$\mathcal{N}_m(\boldsymbol{\mu},\boldsymbol{\Sigma})$ for the $m$-dimensional normal distribution with mean $\boldsymbol{\mu}$ and covariance matrix $\boldsymbol{\Sigma}$. 
The Gaussian Orthogonal Ensemble of dimension $n$ is denoted by $\mathrm{GOE}(n)$, and the Wishart Ensemble of dimension $n$ with covariance matrix $\boldsymbol{\Sigma}$ and $k$ degrees of freedom by $\mathcal W_n(\boldsymbol{\Sigma},k)$.
\subsubsection*{Matrix notation.}
The trace of a matrix $\Ab \in \mathbb{R}^{n \times n}$ is denoted by $\mathrm{tr}(\Ab)$. Let $\mathcal{S}^n$ be the space of symmetric matrices. For $\Ab \in \mathcal{S}^n$, we denote the smallest and largest eigenvalues by $\lambda_{\min}(\Ab)$ and $\lambda_{\max}(\Ab)$, and the symmetric vectorization by $\mathrm{svec}(\Ab)$. 
For $\Ab,\Bb \in \mathcal{S}^n$, the Frobenius inner product and norm are
\[
\langle \Ab,\Bb\rangle_F := \mathrm{tr}(\Ab^\top \Bb),
\qquad
\|\Ab\|_F := \sqrt{\langle \Ab,\Ab\rangle_F}.
\]
The identity matrix is denoted by $\Ib$, and the all-ones matrix by $\Eb := \e\e^\top$. 
For $\Ab,\Bb \in \mathcal{S}^n$, we write $\Ab \succeq \Bb$ if $\Ab-\Bb$ is positive semidefinite.
\subsubsection*{Main contributions and organisation of this paper}

\noindent This paper
\begin{itemize}
\item characterizes the set of first moments of all distributions in a Wasserstein ambiguity ball and proves that it coincides with a closed ball centered at the nominal mean, thereby enabling a simplified treatment of moment uncertainty in distributionally robust optimization;
\item develops reformulations of the inner worst-case (supremum) problem using push-forward measures and dual norms, providing the technical foundation for tractable Wasserstein-DRO reformulations;
\item introduces the distributionally robust standard quadratic optimization problem (DRStQP) and shows that, although the problem is nonconvex in the decision variable, its linear dependence on the uncertain matrix allows the Wasserstein-DRO formulation to be reduced to a robust optimization problem;
\item derives tractable deterministic reformulations of DRStQP for both fixed ambiguity radii and decision-dependent ambiguity radii that are specified as part of the optimization model;
\item establishes a strict minimax inequality example for smooth norms in the cost function, thereby justifying the use of (distributionally) robust optimization instead of solving the corresponding wait-and-see (maximin) formulation;
\item unifies three uncertainty models for the standard quadratic optimization problem—robust StQP, chance-constrained StQP, and distributionally robust StQP—and proves their equivalence under specific distributional assumptions;
\item develops out-of-sample performance guarantees for DRStQP by calibrating the Wasserstein radius from data such that the ambiguity set centered at the empirical distribution contains the unknown true distribution with high probability, ensuring that the DRO optimal value upper-bounds the true out-of-sample performance with high probability;
\item proves a finite-sample guarantee based on measure concentration results and highlights the resulting curse of dimensionality as the problem dimension increases;
\item introduces additional structural assumptions to mitigate dimensionality effects and derives improved, potentially dimension-insensitive, convergence rates under these stronger assumptions;
\item validates the proposed framework through computational experiments on the maximum weighted clique problem, identifying structural transitions in solutions and observing runtime peaks when nominal objective and regularization effects are balanced;
\item investigates the decision-dependent ambiguity setting experimentally, showing how noise level and ambiguity radius influence clique structure, with spectral properties governing sparsity versus saturation of selected cliques.
\item Demonstrates scalability across varying problem dimensions and shows that larger ambiguity levels can enhance structural robustness to noise while preserving computational tractability; and
\item provides appendices containing additional results for alternative norms, complementary proofs, and a strengthened finite-sample guarantee obtained under stronger assumptions using more advanced probabilistic arguments.
\end{itemize}

\noindent
The paper is organized as follows. Section 2 reviews the distributionally robust optimization (DRO) framework and introduces the Wasserstein ambiguity set. We show that its first-moment characterization allows a simplified reformulation of the inner problem. 
\newline
Section 3 defines the distributionally robust standard quadratic optimization problem (DRStQP). Although this problem is nonconvex in the decision variable, it depends linearly on the uncertain parameters, which enables a tractable robust reformulation. We derive deterministic equivalents for both fixed and decision-dependent ambiguity radii. We also present an example illustrating a strict minimax inequality, motivating the use of DRO instead of a wait-and-see approach. Furthermore, we unify three models of uncertain quadratic optimization—robust, chance-constrained, and distributionally robust—and establish their equivalence under suitable assumptions.
\newline
Section 4 develops out-of-sample performance guarantees and proposes a data-driven calibration of the ambiguity radius to ensure high-probability coverage of the true distribution. We derive finite-sample guarantees based on measure concentration results and discuss the curse of dimensionality. To address this limitation, we introduce additional structural assumptions that lead to improved performance bounds.
\newline
Section 5 evaluates the framework through experiments on the maximum weighted clique problem. We analyze both decision-independent and decision-dependent settings, highlighting structural transitions, runtime behavior, and robustness to noise. Scalability tests confirm stable performance across problem sizes. Additional results and proofs are provided in the appendices.

\section{Basics in distributionally robust optimization}\label{sec:2 basics in DRO}

\subsection{The setting}

Suppose that we have an optimization problem where the objective function
$f:\mathcal X \times \mathbb R^m \to \mathbb R$ is a random function that depends on the the decision variable $\x\in \mathcal X \subseteq \mathbb R^n$ and on a random vector $\boldsymbol{\tilde \xi} \in \mathbb R^m$ whose realization is only revealed after the decision $\x$ has been fixed. If the ``true" distribution $\mathbb P_{\rm true}$ of $\boldsymbol{\tilde \xi}$ is known, one can choose to minimize the expected value of $f$, that is to solve the stochastic optimization problem
$$
\inf_{\x \in \mathcal X} \, \mathbb E_{\mathbb P_{\rm true}}[f(\x,\boldsymbol{\tilde \xi})]\,.
$$

\noindent In many applications, the true distribution $\mathbb P_{\rm true}$ is unknown.  Nevertheless, there might be some kind of information available about $\mathbb P_{\rm true}$. Let $\DDD$ denote the ambiguity set of probability distributions that are compatible with the known information about $\mathbb P_{\rm true}$. Naturally, the ambiguity set $\DDD$ will be a subset of $\mathcal P(\mathbb R^m)$
. The distributionally robust optimization problem consists of finding a decision $\x$ that is robust against all probability distributions $\mathbb P$ in the ambiguity set $\mathcal D$
\begin{equation}\label{DRO}\tag{DRO}
\inf_{\x \in \mathcal X} \,\sup_{\mathbb P\in \mathcal D} \, \mathbb E_{\mathbb P}[f(\x,\boldsymbol{\tilde \xi})]\,.
\end{equation}

\noindent Beforehand note that there is a general minimax inequality regardless where we put the expectation operator. It is easy to see that we always have, for a general ambiguity set $\mathcal D$,
$$\sup_{\PP\in\DDD} \, \EE_{\PP}  \left [ \inf_{\x \in \mathcal X} f(\x,\boldsymbol{\tilde \xi})\right ]
\leq  \sup_{\PP\in\DDD} \, \inf_{\x \in \mathcal X}  \EE_{\PP}  [f(\x,\boldsymbol{\tilde \xi})]
\leq
\inf_{\x \in \mathcal X} \,\sup_{\PP\in \DDD} \,  \EE_{\PP}[f(\x,\boldsymbol{\tilde \xi})]\,.
$$

\noindent In some applications, the decision maker has access to the true expectation
$$
\mmu_{\rm true} := \mathbb E_{\mathbb P_{\rm true}}[\boldsymbol{\tilde \xi}]
$$
and true covariance matrix
$$
\sSigma_{\rm true} = \Cov_{\mathbb P_{\rm true}}[\boldsymbol{\tilde \xi}] = \mathbb E_{\mathbb P_{\rm true}}[(\boldsymbol{\tilde \xi} - \mmu_{\rm true})(\boldsymbol{\tilde \xi} - \mmu_{\rm true})^{\top}]
$$
of the random vector $\boldsymbol{\tilde \xi}$. In such cases, the natural ambiguity set $\mathcal D$ consists of all probability distributions with fixed first two moments, see, e.g. 
\cite{delage2010distributionally}:
$$
\mathcal D:=\{\mathbb P \in \mathcal P(\mathbb R^m): \mathbb E_{\mathbb P}[\boldsymbol{\tilde \xi}] = \mmu_{\rm true},\, \Cov_{\mathbb P}[\boldsymbol{\tilde \xi}] = \sSigma_{\rm true}\}\,.
$$

\begin{remark}\label{rem1}  If $f(\x,\boldsymbol{\tilde\xi})= \h(\x)\T\boldsymbol{\tilde\xi}$ is linear in the uncertain parameter $\boldsymbol{\tilde\xi}$ for some mapping $\h : \mathcal X \to \R^m$, we have
\[
\mathbb E_{\PP} [f(\x,\boldsymbol{\tilde\xi})] = \h(\x)\T \EE_\PP [\boldsymbol{\tilde\xi} ] \,. 
\]
Under above moment-based ambiguity model we have
\begin{equation}\label{fixfirstmom}
\mathcal D \subseteq \{\mathbb P \in \mathcal P(\mathbb R^m) : \mathbb E_{\PP} [\boldsymbol{\tilde \xi}] = \mmu_{\rm true}\}\, .
\end{equation}
Then by the assumed linearity,~\eqref{DRO} reduces to a deterministic equivalent 
\[
 \underset{\x \in \mathcal X}{\inf} \overline f(\x) \quad\mbox{with}
 \quad \overline f(\x):= \h(\x)\T \mmu_{\rm true} =  f(\x,\mmu_{\rm true})\, .
\]
\end{remark}

\subsection{Wasserstein ambiguity}
However, in this paper we will focus on another type of distributionally robust optimization problems centered around a reference distribution.

\noindent Suppose that the decision maker has access to a sample $\{\boldsymbol{\widehat{\xi}}_i\}_{i=1}^N$ from $\mathbb P_{\rm true}$ and denote the empirical distribution by $\widehat\PP_N$.
For $p\in [1,\infty)$, the $p$-Wasserstein distance between two probability distributions $\mathbb P, \mathbb P^{\prime} \in \mathcal P(\mathbb R^m)$ is defined by
$$
W_p(\PP, \PP^{\prime}):= \left(\underset{\pi \in \Pi(\PP, \PP^{\prime})}{\inf}\displaystyle\int_{\mathbb R^m\times \mathbb R^m} \,  \norm{\boldsymbol{\xi}-\boldsymbol{\xi^{\prime}}}^p\, \mathrm{d} \pi (\boldsymbol{\xi},\boldsymbol{\xi^{\prime}})\right)^{1/p}
$$
where $\norm{\cdot}$ is a norm on $\mathbb R^m$ and $\Pi(\PP, \PP^{\prime})$ denotes the set of all joint probability distributions on $\mathbb R^m \times \mathbb R^m$ with marginals $\PP$ and $\PP^{\prime}$. In the fully deterministic setting
$\PP=\delta_{\boldsymbol{\xi}}$ and $\PP^{\prime}=\delta_{\boldsymbol{\xi^{\prime}}}$ we obviously have
$\Pi(\PP, \PP^{\prime}) = \{\delta_{(\boldsymbol{\xi}, \boldsymbol{\xi^{\prime}})}\}$, so there is neither a minimization exercise nor an integration to determine $W_p(\PP, \PP^{\prime})$, and for all $p\ge 1$ we get
$W_p(\PP, \PP^{\prime}) =  \norm{\boldsymbol{\xi}-\boldsymbol{\xi^{\prime}}}\, .$

\noindent The Wasserstein distance $W_p$ is finite on $\mathcal P_p(\mathbb R^m)$. Indeed, if $\{\mathbb P, \mathbb P^{\prime}\}\subset \mathcal P_p(\mathbb R^m)$, then this forces $W_p(\mathbb P, \mathbb P^{\prime}) < \infty$. Moreover, if $W_p(\mathbb P, \mathbb P^{\prime}) < \infty$ and $\mathbb P^{\prime} \in \mathcal P_p(\mathbb R^m)$ then it also must hold $\mathbb P \in \mathcal P_p(\mathbb R^m)$. Both implications follow from Young's inequality
$$
{\norm{\boldsymbol{\xi}}}^p\leq 2^{p-1}\left(\norm{\boldsymbol{\xi}-\boldsymbol{\xi^{\prime}}}^p + {\norm{ \boldsymbol{\xi^{\prime}}}}^p\right)\,.
$$
Based on the $p$-Wasserstein distance, we can construct the Wasserstein ambiguity set (also called Wasserstein ball). The Wasserstein ambiguity set centered at a reference measure with finite $p$-th moment $\mathbb P^{\prime} \in \mathcal P_p(\mathbb R^m)$ with radius $\theta>0$ consists of all probability measures with $p$-Wasserstein distance of at most $\theta$ to $\widehat\PP_N$
\begin{equation*}
 \mathbb B_{\theta,p} (\mathbb P^{\prime}):= \{\PP\in \mathcal P(\mathbb R^m) : W_p(\PP, \mathbb P^{\prime})\leq \theta \} \, \subseteq \mathcal P_p (\mathbb R^m)\,.
\end{equation*}
A popular choice for the reference measure is the empirical measure $\mathbb P^{\prime} = \widehat\PP_N$, since, in particular, it has moments of all orders. Under the choice $\mathcal D = \mathbb B_{\theta, p}(\widehat\PP_N)$ where~\eqref{fixfirstmom} is violated (see Theorem~\ref{expectation ball and wasserstein ball} below), we obtain the distributionally robust optimization problem under Wasserstein ambiguity
\begin{equation}\label{WDRO}\tag{WDRO}
   \underset{\x \in \mathcal X}{\vphantom{\sup}\inf} \, \underset{\PP \in \mathbb B_{\theta,p} (\widehat\PP_N)}{\sup} \,\mathbb E_{\PP}[f(\x,\boldsymbol{\tilde\xi})]\, .
\end{equation}

\noindent Observe that this setting includes the deterministic case via $\widehat\PP_1 = \delta_{\boldsymbol{\widehat{\xi}}_1}$ and $N=1$. 
Replacing the Wasserstein ambiguity set centered at the empirical distribution $\mathbb B_{\theta,p}(\widehat{\mathbb P}_N)$ by the Wasserstein ambiguity set centered at the Dirac distribution at the sample mean $\mathbb B_{\theta,p}(\delta_{\overline{\boldsymbol{\xi}}})$ does not help, since we lose information: the variance from the empirical distribution gets lost and the true distribution may lie outside the ball $\mathbb B_{\theta,p}(\delta_{\overline{\boldsymbol{\xi}}})$ but inside the ball $\mathbb B_{\theta,p}(\widehat{\mathbb P}_N)$, as the following lemma demonstrates.

\begin{lemma}
    Let $\norm{\cdot} = \norm{\cdot}_2$ be the Euclidean norm and $p=2$. Let $\z \in \mathbb R^m$ be an estimator of the true mean $\mmu_{\rm true}$ and set
$$
\theta:=\operatorname{Bias}(\z) = \|\z - \mmu_{\rm true}\|_2 = W_2(\delta_{\z},\delta_{\mmu_{\rm true}})
$$
$$
\sigma^2_{\rm true} := \Var_{\mathbb P_{\rm true}}[\boldsymbol{\tilde \xi}] = \operatorname{tr} \Cov_{\mathbb P_{\rm true}}[\boldsymbol{\tilde \xi}]
$$
$$
\operatorname{MSE}(\z) := \mathbb E_{\widehat{\mathbb P}_N}[\|\z-\boldsymbol{\tilde \xi}\|_2^2] = W^2_2(\delta_{\z},\widehat{\mathbb P}_N)\,.
$$
For the 2-Wasserstein distance between the Dirac distribution $\delta_{\z}$ and the empirical distribution $\widehat{\mathbb P}_N$ we have
$$
W_2^2(\delta_{\z},\widehat{\mathbb P}_N) \to W_2^2(\delta_{\z},\delta_{\mmu_{\rm true}}) + \sigma^2_{\rm true}\,, \quad \text{ as } N \to \infty
$$
by the familiar formula $\operatorname{MSE} = \operatorname{Bias}^2+\Var$. The variance term only disappears in the deterministic case $\mathbb P_{\rm true} = \delta_{\boldsymbol{\widehat \xi}}$.
\end{lemma}
\begin{proof}
In this case we have
$$
W_2^2(\delta_{\z}, \widehat{\mathbb P}_N) = \frac{1}{N}\sum_{i=1}^N\|\z - \boldsymbol{\widehat{\xi}}_i\|_2^2 \to \mathbb E_{\mathbb P_{\rm true}}[\|\z - \boldsymbol{\tilde \xi}\|_2^2]\quad \text{ as } N \to \infty\,,
$$
where the last part follows from the law of large numbers. Moreover,
$$
\begin{aligned}
   \mathbb E_{\mathbb P_{\rm true}}&[\|\z - \boldsymbol{\tilde \xi}\|_2^2] = \mathbb E_{\mathbb P_{\rm true}}[\|\z - \mmu_{\rm true} + \mmu_{\rm true} - \boldsymbol{\tilde \xi}\|_2^2]\\
   & = \mathbb E_{\mathbb P_{\rm true}}[\|\z - \mmu_{\rm true}\|_2^2  + 2(\z - \mmu_{\rm true})^{\top}(\mmu_{\rm true} - \boldsymbol{\tilde \xi}) + \|\mmu_{\rm true} - \boldsymbol{\tilde \xi}\|_2^2]\\
   &= \|\z -\mmu_{\rm true}\|_2^2 + \mathbb E_{\mathbb P_{\rm true}}[\|\mmu_{\rm true} - \boldsymbol{\tilde \xi}\|_2^2] = \theta^2 + \mathbb E_{\mathbb P_{\rm true}}[\|\mmu_{\rm true} - \boldsymbol{\tilde \xi}\|_2^2]\,.
\end{aligned}
$$
Finally, using the well known identity $\aa^{\top}\aa = \operatorname{tr}(\aa\aa^{\top})$ and the linearity of the trace and expectation operators we obtain
$$
\begin{aligned}
\theta^2 &+ \mathbb E_{\mathbb P_{\rm true}}[\|\mmu_{\rm true} - \boldsymbol{\tilde \xi}\|_2^2] = \theta^2 + \mathbb E_{\mathbb P_{\rm true}}[\operatorname{tr}((\mmu_{\rm true} - \boldsymbol{\tilde \xi})(\mmu_{\rm true} - \boldsymbol{\tilde \xi})^{\top})]\\
&= \theta^2 + \operatorname{tr}\mathbb E_{\mathbb P_{\rm true}}[(\mmu_{\rm true} - \boldsymbol{\tilde \xi})(\mmu_{\rm true} - \boldsymbol{\tilde \xi})^{\top}] =  \theta^2 + \sigma^2_{\rm true}\,.
\end{aligned}
$$
\end{proof}

\noindent The following proposition gives a closed form solution for the Wasserstein distance between two Gaussian distributions, see 
\cite{Olki82} or 
\cite{takatsu2011}. As mentioned above, the Wasserstein ball can also encapsulate information about higher moments.

\begin{proposition}\label{prop:gaussian-measures}
    If $\norm{\cdot} = \norm{\cdot}_2$ is the Euclidean distance and $p=2$, the 2-Wasserstein distance between Gaussian measures $\mathbb P = \mathcal N_m(\mmu, \boldsymbol{\Sigma})$ and $\mathbb P^{\prime} = \mathcal N_m(\mmu^{\prime}, \boldsymbol{\Sigma}^{\prime})$ is given by
$$
W_2^2(\mathbb P,\mathbb P^{\prime}) = \|\mmu - \mmu^{\prime}\|_2^2 + \operatorname{tr}\boldsymbol{\Sigma} + \operatorname{tr}\boldsymbol{\Sigma}^{\prime} - 2\operatorname{tr}\sqrt{\boldsymbol{\Sigma}^{1/2}\boldsymbol{\Sigma}^{\prime}\boldsymbol{\Sigma}^{1/2}}\,.
$$
\end{proposition}

\noindent For more examples and a deeper discussion of the Wasserstein distance, we refer the reader to 
\cite{panaretos2019statistical} and 
\cite{peyre2019computational}.

\noindent In the special case where the function $f$ is linear in $\boldsymbol{\tilde\xi}$,  by linearity of expectation, we may write the inner maximization problem as

\begin{equation*}\sup_{\PP\in \mathbb B_{\theta,p} (\widehat\PP_N)} f(\x,\mathbb E_{\PP} [\boldsymbol{\tilde\xi}]) \,,
\end{equation*}
but it remains unclear how to resolve this in closed form. 


\subsection{Role of first-order moments under Wasserstein  ambiguity}
 We will now show that the set of first moments of all distributions in a Wasserstein ball $\mathbb B_{\theta,p}(\mathbb P^{\prime})$ around a given nominal distribution $\mathbb P^{\prime}\in \mathcal P_p(\mathbb R^m)$ with radius $\theta >0$ coincides precisely with the closed ball 
\begin{equation*}
    B_{\theta}(\w):= \lbrace \y \in \mathbb R^m: \norm{\y - \w}\leq \theta\rbrace\,
\end{equation*}
of same radius $\theta$, centered at $\w= \mathbb{E}_{\mathbb P^{\prime}}[\boldsymbol{\tilde{\xi}^{\prime}}]$. Denote by
    $$
 \mathcal{M}_{\theta, p}(\mathbb P^{\prime}):= \lbrace \y \in \mathbb R^m: \y =\mathbb{E}_{\mathbb{P}}[\tilde{\bm\xi}], \text{ for some } \mathbb P \in \mathbb B_{\theta,p}(\mathbb P^{\prime})\rbrace\, .
  $$

\begin{theorem}\label{expectation ball and wasserstein ball} Let $p \geq 1$, $\theta > 0$, $\norm{\cdot}$ be any norm on $\mathbb R^m$ and $\mathbb P^{\prime} \in \mathcal{P}_{p}(\mathbb{R}^{m})$, then we have $\mathcal{M}_{\theta, p}(\mathbb P^{\prime}) =  B_{\theta}(\mathbb{E}_{\mathbb P^{\prime}}[\boldsymbol{\tilde{\xi}^{\prime}}])$, independently of $p$.
\end{theorem}
\begin{proof}
In order to show that $\mathcal{M}_{\theta, p}(\mathbb P^{\prime}) \subseteq  B_{\theta}(\mathbb{E}_{\mathbb P^{\prime}}[\boldsymbol{\tilde{\xi}^{\prime}}])$, take $\y \in \mathcal{M}_{\theta, p}(\mathbb P^{\prime})$ arbitrary. This means that there exists a $\mathbb{P} \in \mathcal{P}(\mathbb{R}^{m})$ satisfying $W_{p}(\mathbb{P}, \mathbb P^{\prime}) \leq \theta$ and such that $\y = \mathbb{E}_{\mathbb P}[\boldsymbol{\tilde{\xi}}]$. Consider a random pair $(\boldsymbol{\tilde{\xi}},\boldsymbol{\tilde{\xi}^{\prime}})$ having joint distribution $\pi \in \Pi(\mathbb{P}, \mathbb{P}^{\prime})$ and set $\w = \mathbb{E}_{\mathbb P^{\prime}}[\boldsymbol{\tilde{\xi}^{\prime}}]$. Using Jensen's inequality for the convex function $t\mapsto t^p$, we obtain
\begin{equation}\label{expectation ball and wasserstein ball equ 1}
        \| \y - \w \|^p = \| \mathbb{E}_{\pi}[\boldsymbol{\tilde{\xi}} - \boldsymbol{\tilde{\xi}^{\prime}}] \|^p
    \leq \mathbb{E}_{\pi}[\| \boldsymbol{\tilde{\xi}} - \boldsymbol{\tilde{\xi}^{\prime}}\|]^p
        \leq   \mathbb{E}_{\pi}[\| \boldsymbol{\tilde{\xi}} - \boldsymbol{\tilde{\xi}^{\prime}} \|^{p}]  \, .
    \end{equation}
Taking the infimum on both sides over $\pi\in\Pi(\mathbb{P}, \mathbb P^{\prime})$ of~\eqref{expectation ball and wasserstein ball equ 1}, we can see that
    \begin{equation*}
        \| \y - \w \| \leq W_{p}(\mathbb{P}, \mathbb P^{\prime}) \leq \theta\, ,
    \end{equation*}
which means that $\mathcal{M}_{\theta, p}(\mathbb P^{\prime}) \subseteq  B_{\theta}(\mathbb{E}_{\mathbb P^{\prime}}[\boldsymbol{\tilde{\xi}^{\prime}}])$. To see the reverse inclusion, let $\w = \mathbb{E}_{\mathbb P^{\prime}}[\boldsymbol{\tilde{\xi}^{\prime}}]$ and take any $\y \in \mathbb R^m$ satisfying $\norm{\y - \w} \leq \theta$. Put $\bm\kappa := \y - \w$ and define a translation distribution of $\mathbb P^{\prime}$ as a push-forward measure as follows:
\begin{equation}\label{expectation ball and wasserstein ball equ 2}
        \mathbb{P} := T_{\#} \mathbb P^{\prime}, \text{ with } T(\z) := \z + \bm\kappa \text{ for }  \z \in \mathbb{R}^{m}.
    \end{equation}
    Using \eqref{expectation ball and wasserstein ball equ 2}, we can see that
    \begin{equation*}
        \mathbb{E}_{\mathbb{P}}[\boldsymbol{\tilde{\xi}}] = \mathbb{E}_{\mathbb P^{\prime}}[\boldsymbol{\tilde{\xi}^{\prime}}] + \bm\kappa = \y.
    \end{equation*}
    Thus, we can obtain for the particular joint law $\pi \in \Pi(\mathbb{P}, \mathbb P^{\prime})$ governing  $(\boldsymbol{\tilde{\xi}} , \boldsymbol{\tilde{\xi}^{\prime}})$ under $\mathbb P^{\prime}$ with $\boldsymbol{\tilde{\xi}} = T(\boldsymbol{\tilde{\xi}^{\prime}})$,
    \begin{equation*}
        W^{p}_{p}(\mathbb{P}, \mathbb P^{\prime}) \leq \mathbb{E}_{\pi}[\| \boldsymbol{\tilde{\xi}}- \boldsymbol{\tilde{\xi}^{\prime}} \|^{p}]= \| \bm\kappa \|^{p} = \norm{\y - \w} ^{p} \leq \theta^{p}\, ,
    \end{equation*}
    which implies $\y \in \mathcal{M}_{\theta, p}(\mathbb P^{\prime})$.
\end{proof}

\begin{remark}
    In Theorem~\ref{expectation ball and wasserstein ball} we observe that the choice of $p\geq 1$ is irrelevant, as long as the $p$-th moment exists for the expected cost function under $\mathbb{P}^{\prime}$. The only thing that matters is that the norm used in the set of first moments and in the closed ball is the same.
\end{remark}

\noindent If $f$ is linear in $\boldsymbol{\tilde \xi}$, Theorem~\ref{expectation ball and wasserstein ball} becomes particularly relevant in determining the worst-case distribution, i.e., the solution to the inner ``$\sup$" problem. To this end, recall that for a norm $\norm{\cdot}$ on $\mathbb R^m$ the dual norm, denoted by $\norm{\cdot}^*$ is given as
$$
\norm{\z}^*:=\sup\{\z^{\top}\y: \norm{\y} \leq 1\}\,.
$$
The following theorem can be found in the literature, see 
 e.g., {~\cite{gao2023distributionally}}. We will give a simple proof based upon first-moment observations.

\begin{theorem}\label{dual norm regularization} Let $p \geq 1$, $\theta > 0$, and let $\norm{\cdot}$ be any norm on $\mathbb R^m$ as well as $\mathbb P^{\prime} \in \mathcal{P}_{p}(\mathbb{R}^{m})$.
Suppose $f(\x,\boldsymbol{\tilde \xi}) = \h(\x)\T\boldsymbol{\tilde \xi}$ is linear in $\bm \xi$. Consider any fixed $\x$ and denote by $\cc:=\h(\x)\in \R^m$. Then, for the inner problem we have
$$
\sup_{\mathbb P \in \mathbb B_{\theta,p}(\mathbb P^{\prime})} \mathbb E_{\mathbb P}[\cc^{\top}\boldsymbol{\tilde{\xi}}] = \cc^{\top}\mathbb E_{\mathbb P^{\prime}}[\boldsymbol{\tilde{\xi}^{\prime}}] + \theta \|\cc\|^*\,.
$$
\end{theorem}
\begin{proof}
    Theorem \ref{expectation ball and wasserstein ball} gives
    \begin{equation*}
           \sup_{\mathbb P \in \mathbb B_{\theta,p}(\mathbb P^{\prime})} \mathbb E_{\mathbb P}[\cc^{\top}\boldsymbol{\tilde{\xi}}] = \sup_{\y \in B_{\theta}(\mathbb E_{\mathbb P^{\prime}}[\boldsymbol{\tilde{\xi}^{\prime}}])} \cc^{\top}\y \, .
\end{equation*}
Putting $\w=(\y - \mathbb E_{\mathbb P^{\prime}}[\boldsymbol{\tilde{\xi}^{\prime}}])/\theta$, we
see that the latter expression equals
$$
\sup \{\cc^{\top}(\theta\w +\mathbb E_{\mathbb P^{\prime}}[\boldsymbol{\tilde{\xi}^{\prime}}]) : \|\w\|\leq 1, \w \in \mathbb R^m\}
= \cc^{\top}\mathbb E_{\mathbb P^{\prime}}[\boldsymbol{\tilde{\xi}^{\prime}}] + \theta \|\cc\|^* 
$$
by definition of the dual norm.
\end{proof}

\begin{remark} 
As a consequence, under linearity of $f$ in $\boldsymbol{\tilde \xi}$, the {\em a priori} infinite-dimensional minimax problem given by (DRO) and its ambiguity set $\mathbb B_{\theta,p}(\mathbb P^{\prime})$ can be reduced to a finite-dimensional robust minimax problem with uncertainty set $B_{\theta}(\mathbb{E}_{\mathbb P^{\prime}}[\boldsymbol{\tilde{\xi}^{\prime}}])$. 
It is further worth noting that Theorem \ref{dual norm regularization} depicts that an infinite dimensional linear DRO problem can be reformulated as a deterministic linear problem with a dual norm regularization term. There are many common pairs of dual norms in the Euclidean space $\mathbb R^m$. For the $\ell_p$-norms $\norm{\cdot}_p$ we have dual pairs $(\| \cdot \|_{1}, \| \cdot \|_{\infty})$ and $(\| \cdot \|_{p}, \| \cdot \|_{q})$ where $1<p,q < \infty$ are H\"older conjugate numbers satisfying $1/p + 1/q = 1$.
\end{remark}

\noindent Before we proceed, let us address the question when a minimax theorem can hold. 

\begin{theorem}\label{thm:maximal-element} 
Suppose that $f(\x,\boldsymbol{\tilde \xi}) = \h(\x)\T\bm\xi$ is linear in $\boldsymbol{\tilde \xi}$.
Under the assumptions of Theorem~\ref{expectation ball and wasserstein ball}, if there exists a $\z \in \mathcal{M}_{\theta,p}(\mathbb P^{\prime})$ such that
\begin{equation}\label{relation between range of h and normal cone of M}
    \lbrace \h(\x): \x \in \mathcal{X} \rbrace \subset \mathcal{N}_{\mathcal{M}_{\theta,p}(\mathbb P^{\prime})}(\z)\,,
\end{equation}
then $\z$ is a $\mathcal X$-maximal element in the sense that
\begin{equation}\label{maxel}
\h(\x)^{\top}\z \geq \h(\x)^{\top}\y\quad \mbox{for all }
    (\x,\y)\in \mathcal X \times\mathcal{M}_{\theta,p}(\mathbb P^{\prime})\,,
    \end{equation}
and \eqref{WDRO} is reduced to a deterministic instance of it:
\[
\inf_{\x \in \mathcal X} \sup_{\mathbb P \in \mathbb B_{\theta,p}(\mathbb P^{\prime})} \mathbb E_{\mathbb P}[\h(\x)^{\top}\boldsymbol{\tilde \xi}] =  \inf_{\x \in \mathcal X} \h(\x)^{\top}\z\,.
\]
In particular, $\z = \mathbb E_{\mathbb P^*}[\boldsymbol{\tilde \xi}]$ for some reference distribution $\PP^*$ residing in the ambiguity set $\mathbb B_{\theta,p}(\mathbb P^{\prime})$.
\end{theorem}
\begin{proof} From the definition of the normal cone, we can see that conditions \eqref{relation between range of h and normal cone of M} and \eqref{maxel} are equivalent. As shown in ~\cite[Lemma~1]{bomze2021trust} for the robust StQP, the minimax theorem holds and again, problem ~\eqref{WDRO} is reduced to the given deterministic instance.
\end{proof}

\noindent Using the smoothness of $\partial \mathcal{M}_{\theta, p}(\mathbb P^{\prime})$, we can answer the question whether property~\eqref{maxel} can hold. Recall that different norms lead to different smooth or non-smooth boundary of $\mathcal{M}_{\theta, p}(\mathbb P^{\prime})$.
For instance, the ball $B_{\theta}(\mathbb{E}_{\mathbb P^{\prime}}[\boldsymbol{\tilde{\xi}^{\prime}}])$ has a smooth boundary, if the chosen norm is an $\ell_p$-norm $\norm{\cdot}_p$ with $1<p<\infty$. This smooth boundary would determine $\h(\x)$ in~\eqref{maxel} uniquely up to a positive scaling. Therefore, in this case there is little chance to satisfy~\eqref{maxel} unless $\h(\mathcal X)$ is one-dimensional and normal to the tangent hyperplane of $B_{\theta}(\mathbb{E}_{\mathbb P^{\prime}}[\boldsymbol{\tilde{\xi}^{\prime}}])$ at $\z$. As a consequence, we obtain the following result. 

\begin{corollary}\label{no h satisfy the inclusion}
    Let $\mathcal{X} \subset \mathbb{R}^{n}$, for $n > 1$, and assume that $\| \cdot \|$ is a smooth norm in $\mathbb{R}^{m}$. Let the assumptions of Theorem \ref{expectation ball and wasserstein ball} hold, then the only mappings $\h : \mathcal{X} \to \mathbb{R}^{m}$ satisfying \eqref{relation between range of h and normal cone of M} must be of the form $\h(\x) = \varphi(\x)\cc$ for some constant $\cc\in \R^m$ and some (deterministic) scalar function $\varphi : \mathcal X \to \R$. 
\end{corollary}
\begin{proof}
    The claim follows immediately from Theorem~\ref{expectation ball and wasserstein ball} since the normal cone at each point in $\partial \mathcal{M}_{\theta, p}(\mathbb P^{\prime})$ is one-dimensional under a smooth norm.
\end{proof}

\noindent We see that this sufficient condition for a minimax theorem is unlikely for many $\h$, and indeed, we will show below that the minimax inequality is strict for the WDRO variant of the uncertain StQP where $\h$ is not of the form discussed in Corollary~\ref{no h satisfy the inclusion}. In Appendix \ref{sec:AppendixA2-l-1-norm}, the same will be shown true for this problem class under the non-smooth $\ell_1$-norm. However, for some non-smooth norms like the $\ell_{\infty}$-norm the situation is different, see Appendix \ref{sec:AppendixA1-l-infty-norm}.

\noindent Returning to the general discussion, we will show that for the special case of the Euclidean norm and linear objective function, the worst-case distribution exists and is given by a constant shift of every element of the reference distribution. This result generalizes~\cite[Proposition 4.3]{lanzetti2025first}, where the authors obtained the worst-case formulation for a linear functional under $p = 2$ and absolutely continuous reference measures. In contrast, we extend the results to cases $p \geq 1$ and under any reference measure $\mathbb P^{\prime} \in \mathcal{P}_{p}(\mathbb{R}^{m})$.

\begin{theorem}\label{thm:linear-loss}
    Let $p\geq 1$, $\theta > 0$, $\norm{\cdot} = \norm{\cdot}_2$ be the Euclidean norm on $\mathbb R^m$, $\mathbb P^{\prime} \in \mathcal P_p(\mathbb R^m)$ and $\boldsymbol{0}\neq \cc \in \mathbb R^m$, then
\begin{equation*}
\sup_{\mathbb P \in \mathbb B_{\theta,p} (\mathbb P^{\prime})} \mathbb E_{\mathbb P}[\cc^{\top}\boldsymbol{\tilde \xi }] = \cc^{\top}\mathbb E_{\mathbb P^{\prime}}[\boldsymbol{\tilde{\xi}^{\prime}}] + \theta \norm{\cc}_2\,,
\end{equation*}
and the worst-case distribution is given by
$$
\mathbb P^* = T_{\#}\mathbb P^{\prime} \quad \text{ with } \quad
T(\z) = \z + \frac{\theta}{\norm{\cc}_2}\,\cc\,.
$$
\end{theorem}
\begin{proof}
    The first part follows from Theorem \ref{dual norm regularization} since the Euclidean norm satisfies $\norm{\cdot}_2^* = \norm{\cdot}_2$. It is easy to see that the worst-case distribution $\mathbb P^*$ is feasible, since
$$
\begin{aligned}
W_p^p&(\mathbb P^*, \mathbb P^{\prime}) = \inf_{\pi \in \Pi(\mathbb P^*, \mathbb P^{\prime})} \int_{\mathbb R^m \times \mathbb R^m}\|\boldsymbol{\xi}- \boldsymbol{\xi^{\prime}}\|^p \, \mathrm{d} \pi (\boldsymbol{\xi},\boldsymbol{\xi^{\prime}})\\
&= \inf_{\pi \in \Pi(\mathbb P^*, \mathbb P^{\prime})} \int_{\mathbb R^m \times \mathbb R^m}\|\boldsymbol{\xi^{\prime}} + \tfrac{\theta}{\norm{\cc}_2}\cc- \boldsymbol{\xi^{\prime}}\|^p \, \mathrm{d} \pi (\boldsymbol{\xi},\boldsymbol{\xi^{\prime}}) = \theta^p\,.
\end{aligned}
$$
Finally, we will show that $\mathbb P^*$ is optimal
$$
\mathbb E_{\mathbb P^*}[\cc^{\top}\boldsymbol{\tilde \xi }] = \cc^{\top}\mathbb E_{\mathbb P^*} [\boldsymbol{\tilde{\xi}}] = \cc^{\top}\mathbb E_{\mathbb P^{\prime}} [\boldsymbol{\tilde{\xi}^{\prime}} + \tfrac{\theta}{\norm{\cc}_2}\cc] = \cc^{\top}\mathbb E_{\mathbb P^{\prime}}[\boldsymbol{\tilde{\xi}^{\prime}}] + \theta \norm{\cc}_2\,.
$$
\end{proof}

\section{Distributionally robust StQP}\label{sec:3 DRStQP}

\noindent In this section we will discuss the \eqref{WDRO} version of an uncertain StQP.

\begin{definition}
  {\em Let $\widetilde \Qb \in {\mathcal S}^{n}$ be a random symmetric matrix that follows an unknown distribution $\mathbb P_{\rm true}$ and let $\{\widehat{\Qb}_i\}_{i=1}^N$ be a sequence of independent and identically distributed samples drawn according to~$\mathbb P_{\rm true}$. Consider the reference (empirical) distribution
    $\widehat\PP_N = \frac{1}{N}\sum_{i=1}^N \delta_{\widehat{\Qb}_i}$
 generated by this sample, which in turn defines the ambiguity set $\mathbb B_{\theta,p} (\widehat\PP_N)$, the $p$-Wasserstein ball around $\widehat\PP_N$ with radius $\theta$ under the matrix norm $\norm{\cdot}$. Then the Distributionally Robust Standard Quadratic Optimization Problem (DRStQP) is defined by the problem}

\begin{equation}\label{DRStQP}\tag{DRStQP}
  \underset{\x \in \Delta}{\vphantom{\sup}\inf} \, \underset{\PP \in \mathbb B_{\theta,p} (\widehat\PP_N)}{\sup} \,\mathbb E_{\PP}[\x^{\top}\widetilde \Qb\x]\,.
\end{equation}
\end{definition}

\noindent Observe that even though we defined the Wasserstein distance for random vectors, it can be adapted to incorporate random matrices instead of random vectors, in the following way. 
The well-known symmetric vectorization operator $\operatorname{svec}:\mathcal S^n \to \mathbb R^m$ collects all entries on and above the diagonal of symmetric matrices in $\mathcal S^n$ in vectors in $\mathbb R^m$ with $m = \tfrac{n(n+1)}{2}$. For the ease of presentation, we will assume that for every symmetric matrix $\Ab = (a_{ij}) \in \mathcal S^{n}$, the entries of its symmetric vectorization $\operatorname{svec}(\Ab)$ are arranged in the following way
$$
\operatorname{svec}(\Ab) = (a_{11},a_{22},\dots,a_{nn},\sqrt{2}a_{12},\sqrt{2}a_{13},\dots,\sqrt{2}a_{n-1,n})^{\top}\,,
$$
i.e. the first $n$ entries of $\operatorname{svec}(\Ab)$ correspond to the diagonal entries of $\Ab$ and the last \sloppy $\frac{n(n+1)}{2} - n = \frac{n(n-1)}{2}$ entries of $\operatorname{svec}(\Ab)$ correspond to the off-diagonal entries of $\Ab$ multiplied with the factor $\sqrt{2}$. Since off-diagonal entries are multiplied by $\sqrt 2$, this mapping is an isometry in the sense that 
$$
\langle \Ab , \Bb \rangle_F= 
\svec(\Ab)\T\svec(\Bb)\,,
$$
the rightmost expression being the standard inner product on $\mathbb R^m$. In particular this mapping is a linear bijection. This way all the above settings can be easily transferred from $\R^m$ to ${\mathcal S}^n$; in particular, the ambient space of all distributions $\mathbb P$ should be ${\mathcal S}^n$.

\noindent With these observations and using the objective 
$$f(\x,\widetilde \Qb)= \h(\x)\T\boldsymbol{\tilde \xi}\quad \mbox{with} \quad\h(\x)=\svec (\x\x\T) \quad\mbox{and}\quad \boldsymbol{\tilde \xi}=\svec (\widetilde \Qb)$$
we express $\x\T \widetilde \Qb\x = \langle \x\x\T, \widetilde \Qb\rangle_F  = f(\x,\widetilde \Qb)$ in the form discussed above. 

\subsection{Deterministic reformulations of uncertain StQPs; \texorpdfstring{\\}{ }
positive minimax gap}\label{subsec:minimax gap}

As a straightforward consequence of our general considerations, we obtain the following deterministic reformulation of~\eqref{DRStQP}:

\begin{theorem}\label{thm:DRStQP} Let $p\geq 1$, $\theta >0$, $\norm{\cdot}= \norm{\cdot}_F$ be the Frobenius norm and $\mathbb P^{\prime} = \widehat{\mathbb P}_N$ be the empirical measure, then the \eqref{DRStQP} is equivalent to a deterministic StQP, independent of the choice of $p$. In particular:
\begin{equation}\label{equivalent formulation for decision independent DRO}
  \underset{\x \in \Delta}{\vphantom{\sup}\inf} \, \underset{\PP \in \mathbb B_{\theta,p} (\widehat\PP_N)}{\sup} \,\mathbb E_{\PP}[\x^{\top}\widetilde\Qb\x] = \min_{\x \in \Delta} \, \x^{\top}(\overline \Qb + \theta \, \Ib)\x \,,
\end{equation}
where $\overline{\Qb}:=\frac{1}{N}\sum_{i=1}^N\widehat{\Qb}_i$ denotes the sample mean, and the worst-case distribution is given by $\mathbb P^*_\x := (T^\x)_{\#}\widehat{\mathbb P}_N$, with $T^\x(\Zb) = \Zb + \tfrac{\theta}{\x^{\top}\x}\x\x^{\top}$.
\end{theorem}
\begin{proof} The claim follows from Theorem \ref{thm:linear-loss} by taking $\cc = \operatorname{svec}(\x\x^{\top})$, $\boldsymbol{\tilde \xi} = \operatorname{svec}(\widetilde\Qb)$ and observing that $\|\operatorname{svec}(\x\x^{\top})\|_2 = \x^{\top}\x$.
\end{proof}

\begin{remark} So for the Euclidean norm, the worst-case distribution does depend on $\x$. This is not the case for the maximum norm, where a minimax theorem holds and the worst-case distribution does not depend on $\x$; see Theorem~\ref{thm:DRStQP-maximum-norm} below for the resulting deterministic reformulation. In contrast, in context of Theorem~\ref{thm:DRStQP} the minimax theorem does not hold. Indeed, it is violated already for the robust StQP. As an example, take $\Qb_{\rm nom} = - \Ib$ and $0< \theta < 1$. Then the Frobenius ball
$\mathcal Q:=B_\theta(-\Ib)$ is consisting entirely of negative-definite matrices, so that
$$\min_{\x\in \Delta} \x\T\Qb\x = \min_{i \in [n]} q_{ii}
\quad \mbox{for all }\Qb = (q_{ij})\in \mathcal Q\, .$$
Since $\min_{i \in [n]} v_{ii} \le \frac\theta{\sqrt n}$ for all $\Vb = (v_{ij})$ with ${\norm \Vb}_F\le \theta$, it follows for the maximin value
$$\max_{\Qb\in\mathcal Q} \min_{\x\in \Delta} \x\T\Qb\x =
\max_{\Qb\in \mathcal Q} \min_{i \in [n]} q_{ii}   = \frac \theta{\sqrt n}-1\, ,$$
but the minimax value equals, as discussed above,
$$\min_{\x\in \Delta}\max_{\Qb\in \mathcal Q} \x\T\Qb\x= \min_{\x\in \Delta}
\x\T(\theta-1)\Ib \x = (\theta-1)\max_{\x\in \Delta}\x\T\x = \theta-1\, ,$$
a strictly larger value.
\end{remark}

\noindent The following corollary shows that several uncertain StQP models have the same equivalent deterministic StQP counterpart. Recall a popular way to model random symmetric matrices:

\begin{definition}\label{def:GOE}
    A random symmetric matrix $\widetilde\Gb = (\widetilde g_{ij}) \in \mathcal S^n$ is said to be distributed according to the Gaussian Orthogonal Ensemble (in symbols: $\widetilde\Gb \sim \operatorname{GOE}(n))$, if the diagonal entries $\widetilde g_{ii}$ are i.i.d. with $\widetilde g_{ii} \sim \mathcal N(0,2)$ for $1\leq i \leq n$ and the off-diagonal entries $\widetilde g_{ij}$ are i.i.d. with $\widetilde g_{ij} \sim \mathcal N(0,1)$ for $1\leq i < j\leq n$ and independent of the $\widetilde g_{ii}$.
\end{definition}

\noindent Another popular way of modeling random symmetric matrices is: 

\begin{definition}\label{def:Wishart}
    A random symmetric matrix $\widetilde\Wb \in \mathcal S^n$ is said to be distributed according to the Wishart Ensemble with covariance matrix $\boldsymbol{\Sigma}$ and $k$ degrees of freedom (in symbols: $\widetilde \Wb \sim \mathcal W_n(\boldsymbol{\Sigma},k))$, if 
    $\widetilde\Wb  = \widetilde \Yb \widetilde \Yb^{\top}$ where $\widetilde \Yb = (\widetilde \y_1, \dots, \widetilde \y_k)$ and $\widetilde \y_1, \dots, \widetilde \y_k \sim \mathcal N_n(\boldsymbol{0},\boldsymbol{\Sigma})$ are i.i.d. random vectors from the $n$-dimensional normal distribution with mean $\boldsymbol{0}$ and positive definite covariance matrix $\boldsymbol{\Sigma}$.
\end{definition}

\begin{corollary}\label{cor:unifying RStQP, CCEStQP and DRStQP}
Consider the following uncertain StQP problems:
\begin{enumerate}[label=(\roman*)]
\item Robust StQP with ellipsoidal uncertainty set from \cite{bomze2021trust}: The distribution is unknown, it is only known that $\Qb$ is an element of an ellipsoidal uncertainty set around a nominal matrix $\Qb_{\rm nom}$ with radius $\theta>0$
\begin{equation*}
\min_{\x \in \Delta} \max_{\Qb \in \mathcal Q} \,\x^{\top}\Qb\x\, ,
\end{equation*}
with Frobenius ball
$$
\mathcal Q= B_{\theta}(\Qb_{\rm nom}) = \{\Qb \in \mathcal S^n: \|\Qb - \Qb_{\rm nom}\|_F \leq \theta\}\,.
$$
\item Chance-constrained StQP from~\cite{bomze2025uncertain}: The distribution $\mathbb P_{\rm true}$ is fully known, the chance-constrained StQP is defined as
\begin{equation*}
\min_{\x \in \Delta} \operatorname{VaR}_{\alpha}[\x^{\top}\widetilde \Qb\x]
\end{equation*}
where the Value-at-Risk of a random variable $\tilde \zeta$ with distribution $\mathbb P_{\rm true}$ at level $\alpha \in (0.5,1)$ is given by
$$
\operatorname{VaR}_{\alpha}[\tilde \zeta] = \inf\{t \in \mathbb R: \mathbb P_{\rm true}[\tilde \zeta \leq t]\geq \alpha\}\,.
$$
\begin{itemize}
    \item For the  Gaussian Orthogonal Ensemble perturbation $\widetilde \Gb \sim \operatorname{GOE}(n)$ with 
    $$
\widetilde \Qb = \Qb_{\rm nom} + \beta \widetilde \Gb\, , \quad \beta > 0\, ,
$$
put $\theta := \sqrt{2}\beta \Phi^{-1}(\alpha)>0$, where $ \Phi^{-1}$ is the inverse cumulative distribution function of the standard normal distribution.
\item For the  Wishart Ensemble perturbation $\widetilde \Wb \sim \mathcal W_n(\Ib,k)$ with 
$$
\widetilde \Qb = \Qb_{\rm nom} + \beta \widetilde \Wb\, , \quad \beta > 0\, ,
$$
put $\theta := 2\beta P^{-1}(\tfrac{k}{2},\alpha)$, where $ P^{-1}$ is the inverse of the regularized lower incomplete gamma function $P$ with respect to its second argument.
\end{itemize}
\item Distributionally robust StQP \eqref{DRStQP} under Euclidean norm and Wasserstein ambiguity with any $p \geq 1$, $\theta > 0$.
\end{enumerate}
Then for $\Qb_{\rm nom} = \overline\Qb$, all three uncertain StQP models can be equivalently rewritten as the deterministic StQP
$$
\min_{\x \in \Delta} \x^{\top}(\overline{\Qb} + \theta \, \Ib)\x\,.
$$
\end{corollary}
\begin{proof} The equivalent reformulations of $(i)$ can be found in~\cite{bomze2021trust}, and the reformulation of $(iii)$ is Theorem~\ref{thm:DRStQP} above. The equivalent reformulation of $(ii)$ for the case of GOE perturbation can be found in~\cite{bomze2025uncertain}. It remains to address the Wishart perturbation. For any general $\widetilde \Wb \sim \mathcal W_n(\boldsymbol{\Sigma},k)$, we have that $\x^{\top}\widetilde \Wb\x$ is a gamma distributed random variable with shape $a = \tfrac{k}{2}$ and scale $b = 2\x^{\top}\boldsymbol{\Sigma}\x$. In the special case $\boldsymbol{\Sigma} = \Ib$, this reduces to $b = 2\x^{\top}\x$. The inverse cumulative distribution function $F^{-1}$ of a gamma distributed random variable is given by
\begin{equation*}
  F^{-1}(\alpha) =
    \begin{cases}
       P^{-1}(a,\alpha)\,b& \text{ for } \alpha \geq 0\\
      -\infty &  \text{ for } \alpha < 0\,.
    \end{cases}       
 \end{equation*}
Then, the claim follows from \cite[Theorem~7]{bomze2025uncertain} by setting $\Mb := \Qb_{\rm nom}$ and $\Sb:=\Eb = \e\e^{\top}$.  \end{proof}


\begin{remark}
    Consider a general distributionally robust quadratic problem of the following form
\begin{equation}\label{general DRQP}
\inf_{\x \in \mathcal X} \sup_{\mathbb P \in \mathbb B_{\theta, p}(\widehat{\mathbb P}_N)} \mathbb E_{\mathbb P}[\x^{\top}\widetilde \Qb \x + \widetilde \cc^{\top}\x + \widetilde \omega]\,,
\end{equation}
where $\mathcal X \subseteq \mathbb R^n$ is an arbitrary deterministic subset of $\mathbb R^n$, all data is are random and we only have access to the sample $(\widehat \Qb_i, \widehat \cc_i, \widehat \omega_i)_{i=1}^N$. Observe that now the number of uncertain parameters is $m = \frac{n(n+1)}{2} + n + 1 = \frac{(n+1)(n+2)}{2}$. Denote the sample means by $\overline{\Qb}$, $\overline{\cc}$ and $\overline{\omega}$, respectively and set
$$
\aa=  \left(
        \begin{array}{c}
        \operatorname{svec}(\x\x^{\top})\\
        \x\\
        1
        \end{array}
\right)\,, \quad \boldsymbol{\tilde \xi} =  \left(
        \begin{array}{c}
        \operatorname{svec}(\widetilde \Qb)\\
        \widetilde \cc\\
        \widetilde \omega
        \end{array}
\right)\,.
$$
By Theorem \ref{thm:DRStQP}, for any $p\geq 1$ and Euclidean norm \eqref{general DRQP} is equivalent to the (nonlinear) deterministic problem
$$
\begin{aligned}
\inf_{\x \in \mathcal X}&\biggl\{ \x^{\top}\overline{\Qb}\x + \overline{\cc}^{\top}\x + \overline{\omega} + \theta \sqrt{\norm{\operatorname{svec}(\x\x^{\top})}_2^2 + \norm{\x}_2^2 + \norm{1}_2^2} \biggr\}\\
&= \overline{\omega}  + \inf_{\x \in \mathcal X}\biggl\{ \x^{\top}\overline{\Qb}\x + \overline{\cc}^{\top}\x + \theta \sqrt{(\x^{\top}\x)^2 + \x^{\top}\x + 1} \biggr\}\,.
\end{aligned}
$$
\end{remark}

\subsection{DRStQP with decision-dependent radius}\label{d3}

In this subsection we will consider decision-dependent ambiguity sets. Distributionally robust optimization with decision-dependent ambiguity set is a new framework where the radius of the ambiguity set is a function of the decision variable $\x$, see 
\cite{luo2020distributionally} or 
\cite{noyan2022distributionally}. Suppose that we are interested in Wasserstein ambiguity sets centered at the empirical distribution where the radius is a function $\theta=\theta(\x)$
is a function the decision variable $\x$, i.e., $\theta:\mathcal X\to\mathbb R_+$. This motivates the definition of the decision-dependent Wasserstein ball
\begin{equation*}
 \mathbb B_{\theta(\x),p}(\widehat\PP_N):= \{\PP\in \mathcal P(\mathbb R^m) : W_p(\PP, \widehat\PP_N)\leq \theta(\x) \} \,.
\end{equation*}
The decision-dependent distributionally robust optimization problem ($\operatorname{D^3RO}$) with Wasserstein ambiguity set is then defined as
\begin{equation}\label{D^3RO}\tag{$\operatorname{D^3RO}$}
   \underset{\x \in \mathcal X}{\vphantom{\sup}\inf} \, \underset{\PP \in \mathbb B_{\theta(\x),p}(\widehat\PP_N)}{\sup} \,\mathbb E_{\PP}[f(\x,\boldsymbol{\tilde\xi})]\, .
\end{equation}

\noindent Analogously, the decision-dependent distributionally robust standard quadratic optimization problem ($\operatorname{D^3RStQP}$) is defined as
\begin{equation}\label{D^3RStQP}\tag{$\operatorname{D^3RStQP}$}
   \underset{\x \in \Delta}{\vphantom{\sup}\inf} \, \underset{\PP \in \mathbb B_{\theta(\x),p}(\widehat\PP_N)}{\sup} \,\mathbb E_{\PP}[\x^{\top}\widetilde \Qb\x]\, .
\end{equation}

\begin{corollary}\label{Cor rewrite of D3RO under r(x)}
    Consider the decision-dependent Wasserstein ambiguity set with $p\geq 1$, Euclidean distance (generated by the Frobenius inner product) and any radius function $\theta:\Delta \to (0,\infty)$. Then we have the following deterministic reformulation for the \eqref{D^3RStQP}:
\begin{equation}\label{deterministic form for D^3RStQP}
  \underset{\x \in \Delta}{\vphantom{\sup}\inf} \, \underset{\PP \in \mathbb B_{\theta(\x),p} (\widehat\PP_N)}{\sup} \,\mathbb E_{\PP}[\x^{\top}\widetilde\Qb\x] = \inf_{\x \in \Delta} \, \x^{\top}[\overline \Qb + \theta(\x) \, \Ib]\x  =  \inf_{\x \in \Delta} \left [\x^\top \overline\Qb \x + \theta(\x) \x^\top \x\right ] \,.
\end{equation}
\end{corollary}
\begin{proof}
    For any fixed $\x \in \Delta$, we have that $\theta(\x)>0$ is a constant. Therefore, we can repeat the same steps as in the proof of Theorem~\ref{thm:DRStQP} with $\theta(\x)$ instead of $\theta$ to yield an equivalent reformulation of the inner maximization problem.
\end{proof}

\noindent There are at least three possible choices of $\theta(\x)$ rendering~\eqref{deterministic form for D^3RStQP} tractable:
\begin{enumerate}[label=(\roman*)]
    \item $\theta(\x) = \theta$ yields the \eqref{DRStQP}.
    \item $\theta(\x) = \frac{\gamma}{\x^{\top}\x}$ for some $\gamma \in \mathbb R$ makes the following interesting reduction:
$$
\inf_{\x \in \Delta} \, \x^{\top}[\overline \Qb + \theta(\x) \, \Ib]\x = \gamma + \min_{\x \in \Delta} \, \x^{\top}\overline \Qb\x \,. $$
Therefore we are back to a nominal StQP with data $\overline \Qb$ plus a constant shift.
%
%
\item 
For a given strictly copositive\footnote{i.e., if $\min_{\x\in \Delta}\x\T\Rb\x >0$} matrix $\Rb$, the choice $\theta(\x) = \frac{1}{\x^{\top}{\Rb}\x}$ yields
\begin{equation*}
\inf_{\x \in \Delta} \, \x^{\top}\left [\overline \Qb + \theta(\x) \, \Ib\right]\x = \min_{\x \in \Delta} \left [ \x^{\top}\overline \Qb\x  + \frac{\x^{\top}\x}{\x^{\top} \Rb\x}\right]\, .
\end{equation*}
Observe that we can have indefinite instances of $\Rb$ which are strictly copositive; even for the choice $\Rb=\tfrac{1}{\gamma}\overline \Qb$ which can be motivated as follows: if the decision $\x$ is ``too good''  for the average case ``to be true'', in the sense that $\x\T\overline\Qb\x$ is small, we may put less trust in the result and therefore choose a larger radius $\theta(\x)$ for the ambiguity set; while for moderate quality of $\x\T\overline \Qb\x$ (larger value), we may be satisfied with smaller ambiguity sets. For further discussion and experiments, see Section~\ref{exper} below.
\end{enumerate}

\begin{remark}
Corollary~\ref{Cor rewrite of D3RO under r(x)} can be extended in a straightforward way to more general cost functions than the Euclidean norm and as well to general reference distributions replacing $\widehat{\mathbb P}_N$.
\end{remark}

\section{Out-of-sample performance guarantees for DRStQP}\label{sec:4 out-of-sample performance guarantees}

\noindent In Section \ref{sec:3 DRStQP}, we derived tractable deterministic reformulations of \eqref{DRStQP} under the assumption that the ambiguity radius —either a fixed scalar $\theta$ or a decision-dependent function ${\theta}(\x)$— is specified a priori as part of the model. In this section, we provide the out-of-sample performance guarantees for~\eqref{DRStQP}; for detailed motivation of the following approach see, e.g. 
\cite{mohajerin2018data}. It consists of adapting the radius of the Wasserstein ambiguity set, namely the ball around the empirical distribution
$\widehat{\mathbb P}_N=\sum_{i=1}^N\delta_{\widehat{\boldsymbol{\xi}}_i}$, now seen as a randomly varying distribution under the law $\mathbb P_{\rm true}^N$ governing the values of $(\widehat{\boldsymbol{\xi}}_1,\dots,\widehat{\boldsymbol{\xi}}_N$). 

\noindent The following results are, as in~\cite{mohajerin2018data}, based on a finite sample result on the high-probability (some $1-\beta$ with $\beta>0$ small) coverage of the true (unknown to us) distribution $\mathbb P_{\rm true}$
by balls centered around $\widehat{\mathbb P}_N$ with appropriate radius {\footnote{{note that in contrast to the previous Section~\ref{d3}, the radius does not depend on the decision $\x$, but rather on the confidence level $\beta$; we stick with a similar notation to keep it simple.}}}  $\theta_N(\beta)>0$, which now is no longer a user-defined constant reflecting their confidence in the estimate $\widehat{\mathbb P}_N$, but rather determined by this coverage condition, i.e., by the property
$$
\mathbb P_{\rm true}\in {\mathbb B}_{\theta_N(\beta), p}(\widehat{\mathbb P}_N)\, .
$$

\subsection{Guarantees under exponential decay assumption}

First, we consider the distributions with tails that decay at an exponential rate.
\begin{assumption}\label{assumption of exponential decay tails}
    There exists an exponent $a > 1$ such that the true distribution $\mathbb{P}_{\rm true}$ satisfies
    \begin{equation*}
\mathbb{E}_{\mathbb{P}_{\rm true}}[\exp(\| \boldsymbol{\tilde \xi} \|_2^{a})] = \int_{\mathbb R^m} \exp(\| \bm\xi \|_2^{a}) \, \mathrm{d}\mathbb{P}_{\rm true}(\bm\xi) < \infty.
    \end{equation*}
\end{assumption}


\noindent The GOE model of random symmetric matrices satisfies Assumption~\ref{assumption of exponential decay tails}:

\begin{proposition}\label{lem:GOE-satisfies-exponential-decay} Let $\mathbb P_{\rm true}$ denote the true distribution of $\boldsymbol{\tilde \xi}$, where $\boldsymbol{\tilde \xi}:=\operatorname{svec}(\widetilde \Gb)$, and $\widetilde \Gb \sim \operatorname{GOE}(n)$. Then $\boldsymbol{\tilde \xi} \sim \mathcal N_m(\boldsymbol{0},2\Ib)$ and $\mathbb P_{\rm true}$ satisfies Assumption \ref{assumption of exponential decay tails} for any $a \in (1,2)$.
\end{proposition}
\begin{proof}By construction of the symmetric vectorization operator, since off-diagonal entries are multiplied by a factor of $\sqrt{2}$, we have $\boldsymbol{\tilde \xi} \sim \mathcal N_m(\boldsymbol{0},2\Ib)$. For any nondegenerate $m$-dimensional Gaussian random vector with mean $\boldsymbol{\mu}$ and covariance matrix $\boldsymbol{\Sigma}$ its probability density function is given by
$$
\rho(\boldsymbol{\xi}) = \frac{1}{\sqrt{(2\pi)^m\operatorname{det}\boldsymbol{\Sigma}}}\exp\left(-\tfrac{1}{2}(\boldsymbol{\xi} - \boldsymbol{\mu})^{\top}\boldsymbol{\Sigma}^{-1}(\boldsymbol{\xi} - \boldsymbol{\mu})\right)\, \quad \boldsymbol{\xi} \in \mathbb R^m\,.
$$
In our case for $\boldsymbol{\tilde \xi} \sim \mathcal N_m(\boldsymbol{0},2\Ib)$ this means that
$$
\mathbb E_{\mathbb P_{\rm true}}[\exp(\|\boldsymbol{\tilde \xi}\|_2^a)] = \int_{\mathbb R^m}\exp(\|\boldsymbol{\xi}\|_2^a) \, \mathrm{d}\mathbb P_{\rm true}(\boldsymbol{\xi}) = \frac{1}{2^m\pi^{m/2}}\int_{\mathbb R^m}\exp(\|\boldsymbol{\xi}\|_2^a-\tfrac{1}{4}\|\boldsymbol{\xi}\|_2^2)\,\mathrm{d}\boldsymbol{\xi}\,.
$$
Here we can see that if $a \geq 2$ the integral becomes infinite and for any $a \in (1,2)$ it is finite. 
\end{proof}
\noindent The Wishart model of random symmetric matrices does not satisfy Assumption~\ref{assumption of exponential decay tails}:
\begin{proposition}\label{lem:Wishart-does-not-satisfy-exponential-decay} 
Let $\mathbb P_{\rm true}$ denote the true distribution of $\boldsymbol{\tilde \xi}$, where $\boldsymbol{\tilde \xi}:=\operatorname{svec}(\widetilde \Wb)$, and $\widetilde \Wb \sim \mathcal W_n(\Ib,n)$. Then $\mathbb P_{\rm true}$ does not satisfy Assumption~\ref{assumption of exponential decay tails}.
\end{proposition}
\begin{proof} It is easy to see that for any $\z \in \mathbb R^m$ we have $\|\z\|_2 \geq |z_i|$ for every $i \in [m]$. This implies
$$
\mathbb E_{\mathbb P_{\rm true}}[\exp(\|\boldsymbol{\tilde \xi}\|_2^a)] = \mathbb E_{\mathbb P_{\rm true}}[\exp(\|\operatorname{svec}(\widetilde \Wb)\|_2^a)] \geq \mathbb E_{\mathbb P_{\rm true}}[\exp(|\widetilde w_{11}|^a)]\,.
$$
Now we will use the fact that $\widetilde \Wb = \widetilde \Yb \widetilde \Yb^{\top}$ and by definition, since $\widetilde y_{ij} \sim \mathcal N(0,1)$ are i.i.d. standard normal random variables,
$$
\widetilde w_{11} = \sum_{k=1}^n \widetilde y_{1k}^2 \sim \chi^2(n)\,.
$$
The probability density function of a chi-squared distributed random variable with $n$ degrees of freedom $\tilde \zeta \sim \chi^2(n)$ is given by
$$
\rho(\zeta) =
        \left\{
        \begin{array}{ll}
       \displaystyle\frac{1}{2^{n/2}\Gamma(n/2)}\zeta^{n/2 -1}\exp(-\zeta/2)\,, & \text{ if } \  \zeta \in (0,\infty)\,\\
       0\,, & \text{ otherwise}\,,
        \end{array}
        \right.
$$
where $\Gamma(\cdot)$ denotes the Gamma function. This implies 
$$
 \mathbb E_{\mathbb P_{\rm true}}[\exp(|\widetilde w_{11}|^a)] = \int_0^{\infty}\exp(|\zeta|^a) \displaystyle\frac{1}{2^{n/2}\Gamma(n/2)}\zeta^{n/2 -1}\exp(-\zeta/2) \, \mathrm{d}\zeta
$$
Here we can see that if $a >1$ the integral becomes infinite. 
\end{proof}

\noindent We apply \cite[Theorem~3.5]{mohajerin2018data}, to provide the following finite sample guarantee, aiming at the DRStQP. 

\begin{theorem}\label{finite sample guarantee with tails decay at an exponential rate}
    Let $p\geq 1$ be arbitrary, $\norm{\cdot} = \norm{\cdot}_F$, be the Frobenius norm, $\beta \in (0, 1)$ and suppose that the true measure $\mathbb P_{\rm true}$ satisfies Assumption~\ref{assumption of exponential decay tails}. Define 
 \begin{equation*}
        \theta_N(\beta) :=
        \left\{
        \begin{array}{ll}
        \left( \frac{\log(c_{1} \beta^{-1})}{c_{2} N} \right)^{1/\max\lbrace m, 2\rbrace}, & \text{if} \ \ N \geq \frac{\log(c_{1} \beta^{-1})}{c_{2}},\\
        \left( \frac{\log(c_{1} \beta^{-1})}{c_{2} N} \right)^{1/a}, & \text{if} \ \ N < \frac{\log(c_{1} \beta^{-1})}{c_{2}},
        \end{array}
        \right.
    \end{equation*}
    with positive constants $c_{1}$, $c_{2}$ only depending on $n$, $a$, and $\mathbb{E}_{\mathbb P_{\rm true}}[\exp(\| \widetilde\Qb \|_F^{a})]$. Then, for any $\y \in \mathbb{S}^{n-1}$, we obtain
    \begin{equation*}
        \mathbb P_{\rm true}^N\left[ \y^{\top}(\mathbb{E}_{\mathbb P_{\rm true}}[\widetilde\Qb] - \mathbb{E}_{\widehat{\mathbb P}_N}[\widetilde\Qb] )\y \leq \theta_N(\beta)  \right] \geq 1 - \beta\,.
    \end{equation*}
\end{theorem}
\begin{proof} Using Theorem~\ref{thm:linear-loss}, we know that for all $ \mathbb P \in \mathbb B_{\theta,p} (\widehat{\mathbb P}_N)$,
    \begin{equation*}
        \mathbb E_{\mathbb P}[\x^{\top}\widetilde\Qb\x] \leq \underset{\mathbb P \in \mathbb B_{\theta,p} (\widehat{\mathbb P}_N)}{\sup} \,\mathbb E_{\mathbb P}[\x^{\top}\widetilde\Qb\x] = \x^{\top}(\mathbb{E}_{\widehat{\mathbb P}_N}[\widetilde\Qb] + \theta \, \Ib)\x \quad \text{for all } \x\in \mathbb{R}^{n}\,,
    \end{equation*}
    which means that
$$
        \y^{\top} (\mathbb E_{\mathbb P}[\widetilde\Qb] - \mathbb{E}_{\widehat{\mathbb P}_N}[\widetilde\Qb] ) \y \leq \theta\quad \text{for all } \y \in \mathbb{S}^{n-1}\text{ and all }\mathbb P \in \mathbb B_{\theta_N(\beta),p} (\widehat{\mathbb P}_N)\, .
   $$
    Then, putting $\mathbb P=\mathbb P_{\rm true}$, the claim follows immediately from \cite[Theorems~3.4 and 3.5]{mohajerin2018data}.
    \end{proof}

\begin{remark}
Naturally, for any $\x \in \Delta$ we can find a $\y \in \mathbb S^{n - 1}$ via $\y = \tfrac{\x}{\|\x\|_2}$. Hence for the ``data-driven solution" or ``in-sample solution"
$$
\widehat \x := \underset{\x \in \Delta}{\arg\min} \biggl\{\underset{\mathbb P \in \mathbb B_{\theta_N(\beta), 1}(\widehat{\mathbb P}_N)}{\sup}\mathbb E_{\mathbb P}[\x^{\top}\widetilde \Qb \x]\biggr\}  = \underset{\x \in \Delta}{\arg\min} \biggl\{\x^{\top}(\mathbb E_{\widehat{\mathbb P}_N}[\widetilde \Qb ]+\theta_N(\beta)\Ib)\x\biggr\}
$$
with optimal value 
$$
\widehat J :=  \widehat \x^{\top}(\mathbb E_{\widehat{\mathbb P}_N}[\widetilde \Qb ]+\theta_N(\beta)\Ib) \widehat \x
$$
we can find an upper bound on the ``out-of-sample performance"
$$
\widehat\x^{\top}\mathbb{E}_{\mathbb P_{\rm true}}[\widetilde\Qb]\widehat\x
$$
with high probability:
$$
\mathbb P_{\rm true}^N\left[ \widehat \x^{\top}\mathbb E_{\mathbb P_{\rm true}}[\widetilde \Qb ]\widehat\x \leq \widehat J\right]\geq 1 - \beta\,.
$$
\end{remark}

\noindent The measure concentration provided in \cite[Theorem 2]{fournier2015rate} is the key method to yield the finite sample guarantee in Theorem \ref{finite sample guarantee with tails decay at an exponential rate}. We can notice that the curse of dimensionality occurs in Theorem \ref{finite sample guarantee with tails decay at an exponential rate} due to the radius $\theta_N(\beta)$ chosen of order $N^{-1/\max\lbrace 2, m \rbrace}$. Next, we provide finite sample guarantees with radius possessing order $N^{-2}$. Observing the negative result of Proposition~\ref{lem:Wishart-does-not-satisfy-exponential-decay}, we now provide  in Theorem~\ref{thm:sub-exponential-random-vectors-pointwise} a result which also applies to the Wishart Ensemble perturbation model, cf.\ Proposition~\ref{prop:Wishart are sub-exponential but not sub-Gaussian}.

\subsection{Concepts to reduce the curse of dimensionality}

We now turn to the data-driven calibration of the radius, denoted $\theta_N(\beta)$, with the goal of selecting it so that the Wasserstein ball centered at the empirical measure 
$\widehat{\mathbb P}_N=\sum_{i=1}^N\delta_{\widehat{\boldsymbol{\xi}}_i}$ 
contains the unknown true distribution $\mathbb P_{\rm true}$ with confidence level $1-\beta$. This coverage property ensures that the optimal value of the distributionally robust formulation provides, with high probability, a valid upper bound on the true out-of-sample performance. Our finite-sample guarantee Theorem~\ref{finite sample guarantee with tails decay at an exponential rate} relies on measure concentration results —specifically \cite[Theorem 2]{fournier2015rate}— but the resulting radius $\theta_N(\beta)$ scales as $\mathcal O(N^{-1/{\max\{2,m\}}})$, revealing a pronounced curse of dimensionality as the dimension $m$ grows. This limitation motivates the introduction in this section of additional structural assumptions for Theorem \ref{out-of-sample performance guarantee under p in [1, 2]}, to recover more favorable, potentially dimension-insensitive rates $\mathcal O(N^{-1/2})$, in an approach 
inspired by \cite[Theorem 1]{gao2023finite}.

\noindent In what follows, $\mathbb P_{\rm true}$ denotes the distribution of a random variable $\tilde \xi \in \mathbb R$ or a random vector $\boldsymbol{\tilde \xi} \in \mathbb R^m$, and we assume existence of all first moments. We will need two key concepts.

\begin{definition}[Sub-exponential random vectors]\label{definition of sub-exponential}
 A random variable $\tilde \xi \in \mathbb{R}$ is said to be sub-exponential if
    \begin{equation*}
        \| \tilde \xi \|_{\psi_{1}} := \inf \lbrace t > 0: \mathbb{E}_{\mathbb P_{\rm true}}[\exp(|\tilde \xi|/t)]  \leq 2\rbrace < \infty\,.
    \end{equation*}
    A random vector $\boldsymbol{\tilde \xi} \in \mathbb{R}^{m}$ is called sub-exponential  if
    \begin{equation*}
        \| \boldsymbol{\tilde \xi} \|_{\psi_{1}} := \sup_{\y \in \mathbb{S}^{m-1}} \| \y^{\top}\boldsymbol{\tilde \xi} \|_{\psi_{1}} < \infty\, .
    \end{equation*}
\end{definition}

\begin{definition}[Sub-Gaussian random vectors]\label{definition of sub-gaussian}
    A random variable $\tilde \xi \in \mathbb{R}$ is said to be sub-Gaussian if
    \begin{equation*}
        \| \tilde \xi \|_{\psi_{2}} := \inf \lbrace t > 0: \mathbb{E}_{\mathbb P_{\rm true}}[\exp(|\tilde \xi|^{2}/t^{2})]  \leq 2\rbrace < \infty\, .
     \end{equation*}
    A random vector $\boldsymbol{\tilde \xi} \in \mathbb{R}^{m}$ is called sub-Gaussian if
    \begin{equation*}
        \| \boldsymbol{\tilde \xi} \|_{\psi_{2}} := \sup_{\y \in \mathbb{S}^{m-1}} \| \y^   	{\top}\boldsymbol{\tilde \xi} \|_{\psi_{2}} < \infty\, .
    \end{equation*}
\end{definition}

\noindent The sub-exponential and sub-Gaussian norms 
${\norm .}_{\psi_q}$  are special cases of an 
Orlicz norm \cite{orlicz1932} of a real-valued random variable $\tilde \xi \in \mathbb R$, defined as\footnote{in slight abuse of notation, we also use subscripts to $\norm{\cdot}$ for the Orlicz norms, not only for $q$-norms. The distinction will be clear from the context.} 
$$
\| \tilde \xi \|_{\psi} := \inf \lbrace t > 0: \mathbb{E}_{\mathbb P_{\rm true}}[\psi(|\tilde \xi|/t)]  \leq 1\rbrace\,.
$$
\noindent The generator $\psi(z) = z^p$ yields the $p$-th moment, and the generator $\psi_i(z) = \exp(z^i) - 1$ yields the sub-exponential distributions (for $i=1$) and the sub-Gaussian distributions (for $i=2$).

\noindent From Definition~\ref{definition of sub-gaussian} we can see that all sub-Gaussian random vectors are also sub-exponential. In particular, all Gaussian random vectors satisfy these properties, and therefore the results apply especially to the GOE model as well. 
\begin{proposition}\label{lem:centered Gaussian is also sub-Gaussian} Any centered Gaussian random vector $\boldsymbol{\tilde \xi}\sim \mathcal N_m(\boldsymbol{0},\boldsymbol{\Sigma})$ is also sub-Gaussian.
\end{proposition}
\begin{proof} By \cite[Example~2.5.8]{vershynin2018high}, any centered Gaussian random variable $\tilde \xi \sim \mathcal N(0,\sigma^2)$ is sub-Gaussian. For any arbitrary $\y \in \mathbb S^{m-1}$, we have $\y^{\top}\boldsymbol{\tilde \xi} \sim \mathcal N(0,\y^{\top}\boldsymbol{\Sigma}\y)$. By the previous step, $\y^{\top}\boldsymbol{\tilde \xi}$ is a sub-Gaussian random variable, which yields the claim.
\end{proof}

\noindent A similar result holds for the Wishart distribution used in Subsection~\ref{exper}:
\begin{proposition}\label{prop:Wishart are sub-exponential but not sub-Gaussian}
Let $\widetilde \Wb \sim \mathcal W_n(\Ib, n)$, then $\boldsymbol{\tilde \xi}:=\operatorname{svec}(\widetilde \Wb)$ is sub-exponential but not sub-Gaussian.   
\end{proposition}
\begin{proof}First, we will show that $\boldsymbol{\tilde \xi}$ is sub-exponential. We have that 
$$
\widetilde w_{ii} = \sum_{\ell=1}^n \widetilde y_{i\ell}^2 \sim \chi^2(n)\,,
$$
and
$$
\widetilde w_{ij} = \sum_{k=1}^n \widetilde y_{i\ell}\widetilde y_{j\ell}\,, \quad \text{ for } i\neq j\,.
$$
Observe that $\widetilde y_{ij} \sim \mathcal N(0,1)$ are in particular sub-Gaussian. By \cite[Lemma~2.7.7]{vershynin2018high}, the product of two sub-Gaussian random variables is sub-exponential. Thus, for $\ell\in[n]$, the random variables $\widetilde y_{i\ell}^2$ and $\widetilde y_{i\ell}\widetilde y_{j\ell}$ are sub-exponential. Finite linear combinations of sub-exponential random variables are also sub-exponential, thus $\widetilde w_{ij}$ are sub-exponential for $1\leq i \leq j \leq n$. For any $\y \in \mathbb S^{m-1}$ we have that $\y^{\top}\operatorname{svec}(\widetilde \Wb)$ is a finite linear combination of sub-exponential random variables. Therefore, $\boldsymbol{\tilde \xi}$ is sub-exponential. Now we will show that $\boldsymbol{\tilde \xi}$ is not sub-Gaussian. By definition, a random variable $\boldsymbol{\tilde \xi}\in \mathbb R^m$ is sub-Gaussian, if $\y^{\top}\boldsymbol{\tilde \xi}$ is sub-Gaussian for every $\y \in \mathbb S^{m-1}$. Take $\y = \e_1$, then 
$$
\y^{\top}\boldsymbol{\tilde \xi} = \y^{\top}\operatorname{svec}(\widetilde\Wb) = \widetilde w_{11} \sim \chi^2(n)\,.
$$
We have 
$$
\begin{aligned}
||\widetilde w_{11}||_{\psi_2} &=
\inf \lbrace t > 0: \mathbb{E}_{\mathbb P_{\rm true}}[\exp(\widetilde w_{11}^2/t^2)]  \leq 2\rbrace \\
&=\inf \biggl\{ t > 0: \int_0^{\infty}\exp(w^2/t^2) \displaystyle\frac{1}{2^{n/2}\Gamma(n/2)}w^{n/2 -1}\exp(-w/2) \, \mathrm{d}w \leq 2\biggr\}\,.
\end{aligned}
$$
For every fixed $t>0$, above integral is infinite, so $\norm{\widetilde w_{11}}_{\psi_2} = \infty$ and $\boldsymbol{\tilde \xi}$ is not sub-Gaussian.   
\end{proof}

\subsection{Guarantees; role of the transportation-information inequality}

\noindent {Under the assumptions discussed in the previous subsection, we are able to formulate some out-of-sample performance guarantees, see
Theorem~\ref{thm:sub-exponential-random-vectors-pointwise}. But before let us proceed to 
consider distributions satisfying the transportation-information inequality.}

\noindent Recall that by the Radon-Nikodym theorem, for any two probability measures $\mathbb P, \mathbb P^{\prime} \in \mathcal P(\mathbb R^m)$ such that 
$\mathbb P \ll \mathbb P^{\prime}$, the Radon-Nikodym derivative $\tfrac{\mathrm{d}\mathbb P}{\mathrm{d}\mathbb P^{\prime}}$ exists. In this case, the Kullback-Leibler divergence (also known as relative entropy) of $\mathbb P$ from $\mathbb P^{\prime}$ is defined as
$$
D(\mathbb P || \mathbb P^{\prime}) := \int_{\mathbb R^m} \log\left(\frac{\mathrm{d} \mathbb P}{\mathrm{d}\mathbb P^{\prime}}\right) \, \mathrm{d}\mathbb P\,.
$$
In case that $\mathbb P$ is not absolutely continuous with respect to $\mathbb P^{\prime}$, we set $D(\mathbb P || \mathbb P^{\prime}) = \infty$.

\begin{assumption}\label{ass:transportation-information-inequality} There exists a $c > 0$ such that the true distribution $\mathbb P_{\rm true}\in \mathcal{P}_{p}(\mathbb{R}^{m})$ satisfies
    \begin{equation*}
        W_{p}(\mathbb P, \mathbb P_{\rm true}) \leq \sqrt{2cD(\mathbb P || \mathbb P_{\rm true})}\,, \quad \text{for all } \mathbb P \in \mathcal{P}_{p}(\mathbb{R}^{m})\,.
    \end{equation*}
\end{assumption}


\noindent We observe that the GOE model satisfies Assumption~\ref{ass:transportation-information-inequality} with $c=2$, and any $p \in [1,2]$:
\begin{proposition}\label{lem:GOE-satisfies-transportation-information inequality}
Let $\mathbb P_{\rm true}$ denote the true distribution of $\boldsymbol{\tilde \xi}$, where $\boldsymbol{\tilde \xi}:=\operatorname{svec}(\widetilde \Gb)$, and $\widetilde \Gb \sim \operatorname{GOE}(n)$. Then $\mathbb P_{\rm true}$ satisfies Assumption \ref{ass:transportation-information-inequality} for $c=2$ and any $p\in [1,2]$.
\end{proposition}
\begin{proof}
By Talagrand \cite[Theorem 1.1]{talagrand1996transportation}, for $\mathbb P^{\prime} = \mathcal N_m(\boldsymbol{0},\Ib)$ and $p=2$ we have 
$$
W_2(\mathbb P, \mathbb P^{\prime}) \leq \sqrt{2D(\mathbb P||\mathbb P^{\prime})} \quad \text{for all } \mathbb P \ll \mathbb P^{\prime}\,,
$$
\noindent
with optimal constant $c=1$. Observe that in this case, if $D(\mathbb P||\mathbb P^{\prime}) < \infty$, we automatically obtain that $\mathbb P \in \mathcal P_2(\mathbb R^m)$ since $W_2(\mathbb P, \mathbb P^{\prime})$ is 
finite and for Gaussian measures all moments exist, in analogy of the
inclusion ${\mathbb B}_{\theta,p}(\widehat{\mathbb P}_N) \subseteq \mathcal P_p(\R^m)$ discussed before introducing~\eqref{WDRO}. By direct computation it can be verified that if $T(\z) = a\z$, with $a>0$ then for any $p\geq 1$ we have
$$
W_p(T_{\#}\mathbb P,T_{\#}\mathbb P^{\prime}) = a  W_p(\mathbb P,\mathbb P^{\prime})\,, \quad \text{for all } \mathbb P, \mathbb P^{\prime} \in \mathcal P(\mathbb R^m)\,.
$$
This implies that for $\mathbb P^{\prime} = \mathcal N_m(\boldsymbol{0},\Ib)$ we have 
$$
W_2(T_{\#}\mathbb P,T_{\#} \mathbb P^{\prime}) \leq a\sqrt{2D(\mathbb P||\mathbb P^{\prime})} \quad \text{for all } \mathbb P \in \mathcal P(\mathbb R^m)\,.
$$
We know that the Kullback-Leibler divergence is invariant under affine transformations $T(\z) = a\z + b$, i.e.,
$$
D(\mathbb P||\mathbb P^{\prime}) = D(T_{\#}\mathbb P||T_{\#}\mathbb P^{\prime})\,, \quad \text{for all } \mathbb P, \mathbb P^{\prime} \in \mathcal P(\mathbb R^m)\,,
$$
and, moreover, $\mathbb P \in \mathcal P(\mathbb R^m)$ if and only if 
$T_{\#}\mathbb P \in \mathcal P(\mathbb R^m)$, thus we can rewrite the previous inequality as
$$
W_2(T_{\#}\mathbb P,T_{\#} \mathbb P^{\prime}) \leq a\sqrt{2D(T_{\#}\mathbb P||T_{\#}\mathbb P^{\prime})} \quad \text{for all } T_{\#}\mathbb P \in \mathcal P(\mathbb R^m)\,.
$$
It is well-known that if $\boldsymbol{\tilde \xi^{\prime}}\sim \mathcal N_m(\boldsymbol{0},\Ib)$ then $\boldsymbol{\tilde \xi}:=\sqrt{2}\boldsymbol{\tilde \xi^{\prime}}\sim \mathcal N_m(\boldsymbol{0},2\Ib)$. Set $\mathbb P_{\rm true} = \mathcal N_m(\boldsymbol{0},2\Ib)$, then for $a = \sqrt{2}$ we have $T_{\#}\mathbb P^{\prime} = \mathbb P_{\rm true}$ and this implies 
$$
W_2(T_{\#}\mathbb P,\mathbb P_{\rm true}) \leq \sqrt{4D(T_{\#}\mathbb P||\mathbb P_{\rm true})} \quad \text{for all } T_{\#}\mathbb P \in \mathcal P(\mathbb R^m)\,.
$$
In other words, $\boldsymbol{\tilde \xi}$ satisfies Assumption \ref{ass:transportation-information-inequality} for $c=2$ and $p=2$. Finally, we will show that for any $p\in [1,2]$ and $\mathbb P, \mathbb P^{\prime} \in \mathcal P(\mathbb R^m)$ we have 
$$
W_p(\mathbb P, \mathbb P^{\prime}) \leq W_2(\mathbb P, \mathbb P^{\prime})\,,
$$
which will yield the claim. Let $p \in [1,2]$ be arbitrary. By definition, 
$$
W_p(\PP, \PP^{\prime}):= \left(\underset{\pi \in \Pi(\PP, \PP^{\prime})}{\inf}\displaystyle\int_{\mathbb R^m\times \mathbb R^m} \,  \norm{\boldsymbol{\xi}-\boldsymbol{\xi^{\prime}}}^p\, \mathrm{d} \pi (\boldsymbol{\xi},\boldsymbol{\xi^{\prime}})\right)^{1/p}\leq  \left(\displaystyle\int_{\mathbb R^m\times \mathbb R^m} \,  \norm{\boldsymbol{\xi}-\boldsymbol{\xi^{\prime}}}^p\, \mathrm{d} \pi_2^* (\boldsymbol{\xi},\boldsymbol{\xi^{\prime}})\right)^{1/p}\,,
$$
where $\pi_2^*$ is the coupling which minimizes $W_2(\mathbb P, \mathbb P^{\prime})$ and exists by the lower semicontinuity of the norm, see 
\cite{villani2008optimal}. Recall the monotonicity of the $L^p$-norm, i.e. if $p\leq q$ then for any non-negative random vector $\boldsymbol{\tilde \zeta}$ with distribution $\pi$ we have 
$\|\boldsymbol{\tilde \zeta}\|_{L^p(\pi)} \leq \|\boldsymbol{\tilde \zeta}\|_{L^q(\pi)}$. This implies 
$$
\left(\displaystyle\int_{\mathbb R^m\times \mathbb R^m} \,  \norm{\boldsymbol{\xi}-\boldsymbol{\xi^{\prime}}}^p\, \mathrm{d} \pi_2^* (\boldsymbol{\xi},\boldsymbol{\xi^{\prime}})\right)^{1/p}\leq  \left(\displaystyle\int_{\mathbb R^m\times \mathbb R^m} \,  \norm{\boldsymbol{\xi}-\boldsymbol{\xi^{\prime}}}^2\, \mathrm{d} \pi_2^* (\boldsymbol{\xi},\boldsymbol{\xi^{\prime}})\right)^{1/2}=:W_2(\mathbb P, \mathbb P^{\prime})\,.
$$
The claim follows.\end{proof}

\noindent Now we will show that the Wishart Ensemble model does not satisfy Assumption \ref{ass:transportation-information-inequality}. Recall that the Lipschitz modulus of a function $g:\mathbb R^m\to\mathbb R$ is defined as
$$
\|g\|_{\rm Lip} := \sup_{\boldsymbol{\xi}\neq \boldsymbol{\xi^{\prime}}}\frac{|g(\boldsymbol{\xi})-g(\boldsymbol{\xi^{\prime}})|}{\|\boldsymbol{\xi}-\boldsymbol{\xi^{\prime}\|}_2}\,.
$$
\begin{proposition}\label{lem:Wishart-does-not-satisfy-transportation-information inequality}
Let $\mathbb P_{\rm true}$ denote the true distribution of $\boldsymbol{\tilde \xi}$, where $\boldsymbol{\tilde \xi}:=\operatorname{svec}(\widetilde \Wb)$, and $\widetilde \Wb \sim \mathcal W_n(\Ib,n)$. Then, $\mathbb P_{\rm true}$ does not satisfy Assumption~\ref{ass:transportation-information-inequality}.
\end{proposition}
\begin{proof}
    By \cite[Theorem 3.1]{bobkov1999exponential}, $\mathbb P_{\rm true}$ satisfies Assumption~\ref{ass:transportation-information-inequality} for $p=1$ and some $c>0$ if and only if for every function $g:\mathbb R^m \to \mathbb R$ with $||g||_{\rm Lip}\leq 1$ and $\mathbb E_{\mathbb P_{\rm true}}[g(\boldsymbol{\tilde \xi})] = 0$ we have
$$
\mathbb E_{\mathbb P_{\rm true}}\exp([tg(\boldsymbol{\tilde \xi}))] \leq \exp\left(\frac{ct^2}{2}\right)\, \quad \text{for every } t \in \mathbb R\,.
$$    
The function $g(\boldsymbol{\tilde \xi}):= \e_1^{\top}\boldsymbol{\tilde \xi} - n = \widetilde w_{11} - n$ satisfies the above assumption, since $||g||_{\rm Lip} = \|\e_1\|_2 = 1$ and 
$$
\mathbb E_{\mathbb P_{\rm true}}[g(\boldsymbol{\tilde \xi})] = \mathbb E_{\mathbb P_{\rm true}}[\widetilde w_{11}] - n = 0\,.
$$
By \cite[Proposition 2.5.2]{vershynin2018high}, this is precisely the characterization of a sub-Gaussian random variable via moment generating function. As we have seen in Proposition \ref{prop:Wishart are sub-exponential but not sub-Gaussian}, the random variable $\widetilde w_{11}-n$ is not sub-Gaussian, thus $\mathbb P_{\rm true}$ cannot fulfill Assumption \ref{ass:transportation-information-inequality} for $p=1$. This implies that $\mathbb P_{\rm true}$ cannot fulfill Assumption \ref{ass:transportation-information-inequality} for any higher $p\geq 1$.
\end{proof}

\noindent We will need the following auxiliary result.
\begin{lemma}\label{growth condition for quadratic form}
    Let $p \in [1, 2]$ and $\Qb \in \mathcal S^{n}$. Then,  there exist $M, L > 0$ such that
    \begin{equation*}
        |\y^{\top} \Qb \y| \leq M + L \| \operatorname{svec}(\Qb)\|^{p}_2\,,\quad\text{for all } (\y,\Qb) \in \mathbb{S}^{n-1} \times \mathcal S^{n}\, .
    \end{equation*}
\end{lemma}
\begin{proof}
    The situation under $p = 1$ is trivial. Assume $p \in (1, 2]$ with Hölder conjugate number $q := \tfrac{p}{p-1}$. Using the Cauchy-Schwarz inequality and Young's inequality, respectively, we obtain
    \begin{equation*}
        |\y^{\top} \Qb \y| = |\operatorname{svec}(\y \y^{\top})^{\top} \operatorname{svec}(\Qb) | \leq \|\operatorname{svec}(\y \y^{\top})||_{2} \| \operatorname{svec}(\Qb) \|_{2} \leq \frac{\|\operatorname{svec}(\y \y^{\top})||_{2}^{q}}{q} + \frac{\| \operatorname{svec}(\Qb) \|^{p}_{2}}{p} 
    \end{equation*}
for all $(\y,\Qb)\in \mathbb S^{n-1}\times {\mathcal S}^n$. Hence the result.\end{proof}

\begin{theorem}\label{out-of-sample performance guarantee under p in [1, 2]}
    Let $p \in [1, 2]$ and suppose that $\mathbb P_{\rm true}$ 
    satisfies Assumption \ref{ass:transportation-information-inequality}. 
For $0<\beta < 1$, put
$\theta_{c,N}(\beta) = \sqrt{\frac{2c\log (1/\beta)}{N}}$, then for any $\y \in \mathbb{S}^{n-1}$, 
    \begin{equation*}
        \mathbb P_{\rm true}^N\left[ \y^{\top} \left( \mathbb{E}_{\widehat{\mathbb{P}}_{N}}[\widetilde\Qb] - \mathbb{E}_{\mathbb P_{\rm true}}[\widetilde\Qb] \right) \y \leq \theta_{c,N}(\beta)\right] \geq 1 - \beta\,.
    \end{equation*}
\end{theorem}
\begin{proof}
    Fix $\theta > 0$. Using Theorem \ref{thm:DRStQP} we know that
    \begin{equation*}
        \underset{\mathbb P \in \mathbb B_{\theta,p} (\mathbb P_{\rm true})}{\sup} \,\mathbb E_{\mathbb P}[\x^{\top}\widetilde\Qb\x] - \x^{\top}\mathbb{E}_{\mathbb P_{\rm true}}[\widetilde\Qb]\x = \theta \|\x\|^{2}_2\,,\quad \text{for all } \x\in \mathbb{R}^{n}\,,
    \end{equation*}
    which implies that
    \begin{equation*}
        \underset{\mathbb P \in \mathbb B_{\theta,p} (\mathbb P_{\rm true})}{\sup} \,\mathbb E_{\mathbb P}[\y^{\top}\widetilde\Qb\y] - \y^{\top}\mathbb{E}_{\mathbb P_{\rm true}}[\widetilde\Qb]\y = \theta\,, \quad \text{for all } \y\in \mathbb{S}^{n-1}\, .
    \end{equation*}
    Using Lemma \ref{growth condition for quadratic form}, we can see that there exist $M, L > 0$ such that
    \begin{equation*}
        |\y^{\top} \Qb \y| = |\operatorname{svec}(\y \y^{\top})^{\top} \operatorname{svec}(\Qb)| \leq M + L \| \operatorname{svec}(\Qb)\|^{p}_2\,, \quad \text{for all } (\y,\Qb) \in \mathbb{S}^{n-1}\times \mathcal S^{n}\, .
    \end{equation*}
    Thus, the claim follows immediately from \cite[Theorem 1]{gao2023finite}.
\end{proof}

\noindent 
In Theorem \ref{out-of-sample performance guarantee under p in [1, 2]}, we focus on $p \in [1, 2]$ due to the limitation of applying the tensorization proposition (see \cite[Proposition 22.5]{villani2008optimal} and \cite[Lemma 4]{gao2023finite}), which provides the Talagrand’s transportation-cost inequalities for product distributions. In Appendix~\ref{sec:AppendixB about proof of section out-of-sample performance guarantees} we show that the range of $p$ can be extended to any $p \geq 1$ for the StQP under sub-Gaussian and sub-exponential assumptions. For ease of reading, we will specify an implication already at the end of this section.

\begin{theorem}[Sub-exponential random instances]\label{thm:sub-exponential-random-vectors-pointwise}
Let $ \widehat{\boldsymbol{\xi}}_i = \operatorname{svec}(\widehat \Qb_i)$ be i.i.d. sub-exponential random vectors
generated by observations of random $\widehat \Qb_i\in {\mathcal S}^n$, $i\in [N]$. Then, for any $\beta \in (0, 1)$ and any ${\y \in \mathbb{S}^{n-1}}$,
    \begin{equation*}
        \mathbb P_{\rm true}^N\left[\y^{\top}\left( \mathbb{E}_{\widehat{\mathbb{P}}_{N}}[\widetilde\Qb] - \mathbb{E}_{\mathbb P_{\rm true}}[\widetilde\Qb] \right)\y  \leq C K \left( \sqrt{\frac{m + \log(2/\beta)}{N}} + \frac{m + \log(2/\beta)}{N}  \right) \right] \geq 1 - \beta\, ,
    \end{equation*}
    where $C > 0$ is an absolute constant, $K := \max_{i} \| \widehat{\boldsymbol{\xi}}_i \|_{\psi_{1}}$, and $\| \cdot \|_{\psi_{1}}$ represents the sub-exponential norm defined in Definition~\ref{definition of sub-exponential}. Furthermore, assume that there exists an $R > 0$ such that
    \begin{equation*}
        \| \|\widehat{\boldsymbol{\xi}}_i\|_{2} \|_{\psi_{1}} \leq R\, \quad \text{for all } i \in [N]\,.
    \end{equation*}
    Then,
    \begin{equation*}
        \mathbb P_{\rm true}^N\left[ \y^{\top}\left( \mathbb{E}_{\widehat{\mathbb{P}}_{N}}[\widetilde\Qb] - \mathbb{E}_{\mathbb P_{\rm true}}[\widetilde\Qb] \right)\y \leq 2\sqrt{2}R \sqrt{\frac{2 \log(2/\beta)}{N}} + \frac{R\log(2/\beta)}{N} \right] \geq 1 - \beta\, .
    \end{equation*}
    Moreover, if $\{\widehat{\boldsymbol{\xi}}_i\}_{i=1}^N\subset \mathbb R^m$ are i.i.d. sub-Gaussian random vectors, 
    we have for all $\y \in \mathbb{S}^{n-1}$  
    \begin{equation*}
        \mathbb P_{\rm true}^N\left[  \y^{\top}\left( \mathbb{E}_{\widehat{\mathbb{P}}_{N}}[\widetilde\Qb] - \mathbb{E}_{\mathbb P_{\rm true}}[\widetilde\Qb] \right)\y  \leq C K \left( \sqrt{\frac{m}{N}} + \sqrt{\frac{\log(2 / \beta)}{N}} \right) \right] \geq 1 - \beta\, ,
    \end{equation*}
    where $C > 0$ is an absolute constant, $K := \max_{i} \| \widehat{\boldsymbol{\xi}}_i \|_{\psi_{2}}$, and $\| \cdot \|_{\psi_{2}}$ represents the sub-Gaussian norm defined in Definition~\ref{definition of sub-gaussian}.
\end{theorem}

\begin{proof} Follows in straightforward manner from~Theorem~\ref{thm:sub-exponential-random-vectors-uniform} where the same probability bounds are established for the suprema of the random quadratic forms across all $\y\in \mathbb S^{n-1}$.\end{proof}

\noindent The connection between the assumptions can be explained as follows. Assumption \ref{ass:transportation-information-inequality} is commonly referred to in the literature as a transportation-information inequality $T_p(c)$. It is known that $T_2(c)$ implies $T_1(c)$. Moreover, $T_1(c)$ is equivalent to the exponential decay Assumption \ref{assumption of exponential decay tails}, see \cite[Theorem 22.10]{villani2008optimal}, and $T_1(c)$ implies sub-Gaussianity. In turn, sub-Gaussianity implies sub-exponentiality. The GOE model satisfies $T_2(c)$, thus satisfies them all. The Wishart model only satisfies sub-exponentiality, the weakest of them all.

\section{Numerical {experiments} 
}\label{sec:5 numerical simulations}
In this section, we will apply the above theory to the maximum weighted clique problem with uncertain weights. Let $G=(V,E)$ be a simple undirected graph with vertex set $V$ and edge set $E \subseteq V \times V$, and denote the number of vertices by $n = |V| <  \infty$. Recall that a graph is called complete if $E = V \times V$, i.e. for every pair of distinct vertices there is a unique edge connecting them. A subset $S\subseteq V$ of the set of vertices is called a clique if it induces a complete subgraph of $G$. A clique $S$ is maximal if there is no superset of $S$ that is also a clique. A maximum clique is a maximal clique with largest cardinality over all cliques. The maximum clique problem consists of finding the maximum clique in a graph.

\noindent We will consider the following variation of the maximum clique problem. Let $\w = (w_i)_{i \in V} \in \mathbb R_+^n$ be a vector of vertex weights. The maximum weighted clique problem consists of finding the clique $S\subseteq V$ with largest total weight sum $W(S) = \sum_{i\in S} w_i$. Following the method given in~\cite[Section 4.2]{bomze2021trust}, we derive a weighted adjacency matrix $\Ab = (a_{ij})$ as follows:
\begin{equation}\label{weighted adjacency matrix}
a_{ij} := \left\{
        \begin{array}{lll}
        &1 - \frac{1}{2 w_{i}}, &\text{ if } i = j \in V,\\
        &1, \ \ \ \ \ \ \ \  &\text{ if }\lbrace i, j \rbrace \in E, \text{ and }\\
        &1 - \frac{1}{2 w_{i}} - \frac{1}{2 w_{j}}, &\text{ otherwise }.
        \end{array}
\right.
\end{equation}

\noindent The maximum weighted clique problem is equivalent to the following continuous optimization problem
\begin{equation}\label{prob:MWCP}
\underset{\x \in \Delta}{\max}\,\x^{\top}\Ab\x\,.
\end{equation}
By~\cite[Theorem 7]{bomze1998standard} a global solution $\x^*$ of the StQP \eqref{prob:MWCP} yields a maximum weight clique $S^*=\{i \in V: x_i^* > 0\}$ and its total weight can be also computed as $W(S^*) = [2(1-\x^{*\top}\Ab\x^*)]^{-1}$.

\noindent Let $\Eb := \e\e^{\top}$ denote the matrix of all ones and put $\Qb_{\rm nom} := \Eb - \Ab$. Since $\x^{\top}\Eb\x = \x^{\top}\e\e^{\top}\x = 1$ for $\x \in \Delta$, we can rewrite  \eqref{prob:MWCP} as an StQP as follows:
\begin{equation}\label{prob:MWCP-re}
\underset{\x \in \Delta}{\min}\,\x^{\top}\Qb_{\rm nom}\x\,.
\end{equation}

\subsection{Wasserstein balls with decision-independent radius}\label{section numerical simulation with decision-independent radius}
In this subsection, we solve the decision-independent DRStQP given in \eqref{equivalent formulation for decision independent DRO}. First, we randomly generate an undirected graph $G=(V,E)$ and a weights vector $\w$. To avoid negative entries of the weighted adjacency matrix $\Ab$, the weights are generated according to the exponential distribution as 
\begin{equation}\label{weight setting}
    \widetilde{w}_i = 1 + \widetilde{z}_i, \text{ with } \widetilde{z}_i \sim \text{Exp}(1.5).
\end{equation}
Then, we compute the weighted adjacency matrix $\Ab$ and form the matrix $\Qb_{\rm nom} := \Eb - \Ab$. After that, we sample $\{\widehat \Gb_i\}_{i=1}^N \sim \operatorname{GOE}(n)$, recall Definition~\ref{def:GOE}, and put 
\begin{equation*}
\widehat\Qb_{\beta,i}:= \Qb_{\rm nom} + \beta \,\widehat \Gb_i\,,
\end{equation*}
for a $\beta > 0$. The sample mean of $\{\widehat\Qb_{\beta,i}\}_{i=1}^N$ is denoted by
\begin{equation*}
\overline\Qb_{\beta}:=\frac{1}{N}\sum_{i=1}^N\widehat\Qb_{\beta,i}\, \,.
\end{equation*}

\begin{example}\label{solutions to decision-independent case under different beta and theta}
Figure \ref{graph of solutions to decision-independent case under different beta and beta} illustrates the impact of the ambiguity set radius $\theta$ and the noise level $\beta$ on the solution topology of the maximum clique problem \eqref{equivalent formulation for decision independent DRO}. The red edges represent the connectivity of the optimal solution clique. A fundamental structural transition is observed as the radius $\theta$ increases, highlighting the interplay between the structural objective and the regularization penalty. When the radius is small ($\theta = 0.01$, first column), the optimization is dominated by the underlying graph structure term $\x^{\top} \overline{\Qb}_{\beta} \x$. In low-noise regimes ($\beta \le 0.1$), the solver strictly adheres to the topological constraints, yielding solutions that are fully connected subgraphs (cliques). However, this strict adherence comes at the cost of robustness. The solutions are sensitive to the noise level $\beta$: as $\beta$ increases to $0.8$, the solution topology deteriorates from a dense clique into a sparser subgraph. Correspondingly, the weight of the identified subgraph drops significantly ($W(S)$ decreases from $9.48$ to $6.68$), indicating that without sufficient regularization, the solver fits to the noise in $\overline{\Qb}_{\beta} $. As $\theta$ increases to $1.5$ (third column), the regularization term $\theta \|\x\|^2_{2}$ exerts a strong smoothing effect, leading the solver to spread the probability mass over a larger support set. Consequently, the solution topology shifts away from a strict clique structure into a denser, more distributed subgraph. Despite this change in topology, the solutions exhibit remarkable stability. The visualized solution structure remains visually dense and highly connected across all noise levels $\beta \in \{0.01, 0.1, 0.8\}$ (third column). Notably, the objective weight does not degrade under high noise in this regime; instead, it exhibits an increasing trend, rising from $W(S)=18.87$ to $W(S)=20.37$. This demonstrates that a large $\theta$ effectively immunizes the solution against sample noise, allowing the solver to maintain, and even enhance, the solution quality even under significant uncertainty. The case of $\theta = 0.6$ (middle column) represents a transition phase, where the solution maintains clique-like properties under low noise but experiences a moderate decline in weight ($W(S)$ drops to $6.57$) as the noise $\beta$ becomes dominant.

\end{example}

\begin{figure}[htbp]
\centering
\includegraphics[scale = 0.23]{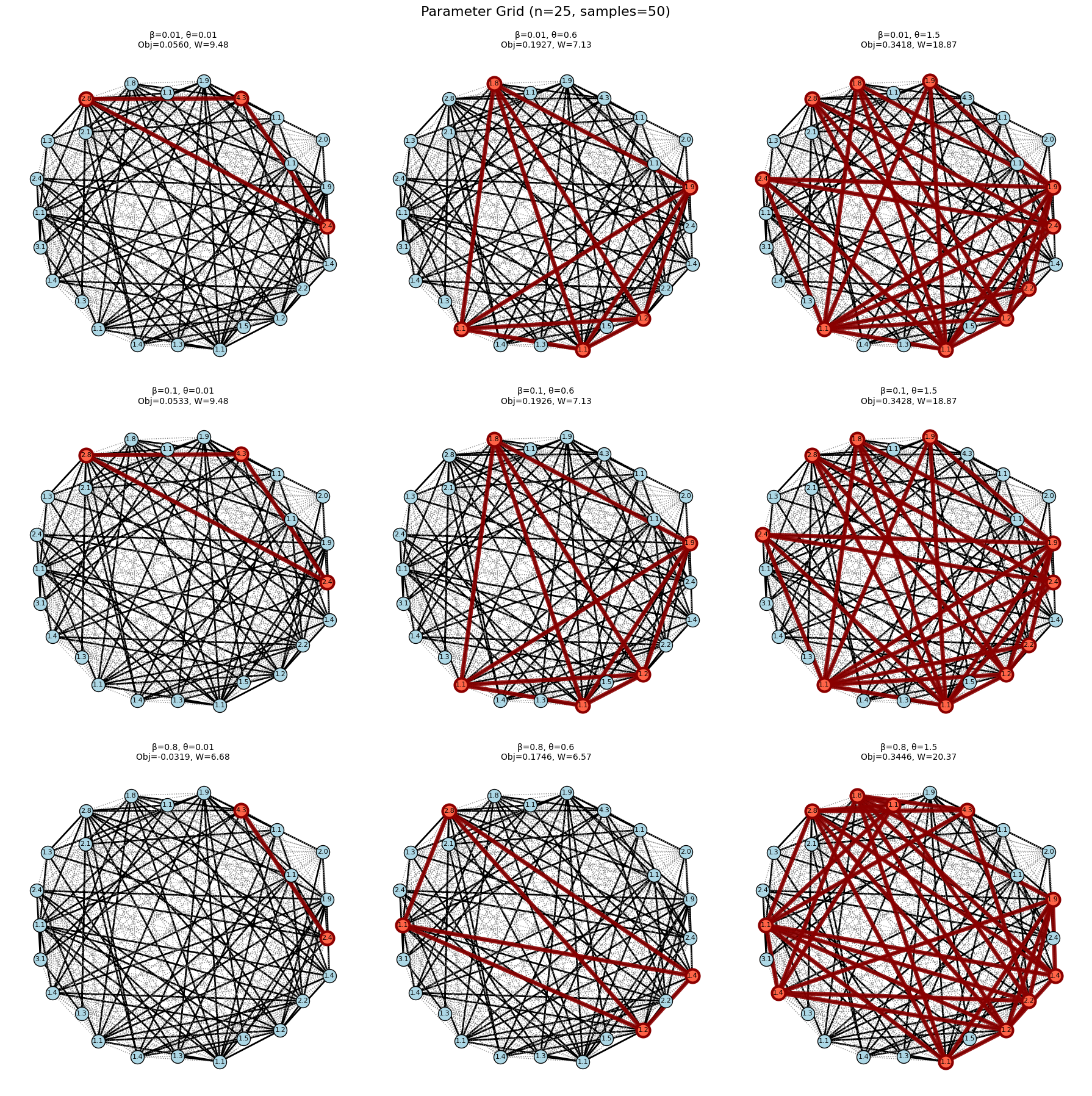}
\caption{Maximum clique solutions under varying noise levels $\beta$ and robustness radii $\theta$. The columns correspond to $\theta \in \{0.01, 0.6, 1.5\}$ and the rows correspond to $\beta \in \{0.01, 0.1, 0.8\}$. $W$ represents the weighted size of the solution clique. (Example \ref{solutions to decision-independent case under different beta and theta}).}
\label{graph of solutions to decision-independent case under different beta and beta}
\end{figure}

\begin{example}\label{Out-of-sample performance to decision-independent case under differential gamma}
To further investigate the properties of the decision-independent case, Figure \ref{graph of solutions to decision-independent case for different node scales and sample sizes-obj and max weights} presents four key performance metrics as a function of $\theta$ over the range $[10^{-3}, 10]$. As expected, the optimal objective value, shown in Figure \ref{graph of solutions for different node scales and sample sizes-Objective function values}, increases monotonically with $\theta$. The curve exhibits a smooth ascent, dominated by the quadratic penalty term $\theta \|\mathbf{x}\|_2^2$ as $\theta$ becomes large. The narrow shaded region (Min-Max range) indicates that the objective value is relatively stable across different random trials. Figure \ref{graph of solutions for different node scales and sample sizes-Total maximum clique weight} illustrates that the total support weight remains constant for small $\theta$ ($< 0.1$) but undergoes a sharp logistic-like increase in the interval $\theta \in [1, 4]$ before plateauing. This transition implies that as the regularization penalty increases, the optimal solution vector $\mathbf{x}^*$ spreads its mass over a larger set of nodes or adjusts the weighting scheme to minimize the penalized objective, leading to a higher aggregate weight in the solution support. Figure \ref{graph of solutions for different node scales and sample sizes-Graph density of final solution graphs} shows the graph density of the final solution, which is defined as follows:
\begin{equation*}
    \rho(G^{*}) := \frac{|E^{*}|}{
        \begin{pmatrix}
        |V^{*}| \\
        2
        \end{pmatrix}
    } 
\end{equation*}
This metric corroborates the visual findings from Figure \ref{graph of solutions to decision-independent case under different beta and beta}. The graph density remains at $\rho(G) \approx 1.0$ (indicating a clique) for $\theta < 0.2$. A phase transition occurs between $\theta \approx 0.2$ and $\theta \approx 2.0$, where the density drops drastically, settling at a lower steady state ($\approx 0.3$). This confirms that the \textit{clique} property of the solution is only maintained when the radius $\theta$ is below a critical threshold. Beyond this threshold, the sparsity-inducing effects of the regularization (or the shift in eigen-structure) break the clique topology. The running time analysis, shown in Figure \ref{graph of solutions to decision-independent case under for different node scales and sample sizes-Running time of final solution graphs}, reveals a critical computational phenomenon. The solver is highly efficient for both small $\theta$ (stable clique regime) and very large $\theta$ (stable sparse regime). However, a significant spike in solving time is observed in the transition region $\theta \in [1, 3]$. This peak suggests that the optimization landscape becomes most complex near the phase transition boundary, where the solver struggles to distinguish between competing local optima or structural configurations.
\end{example}

\begin{figure}[htbp]
    \centering
    \begin{subfigure}[b]{0.40\textwidth}
        \centering        \includegraphics[width=\textwidth]{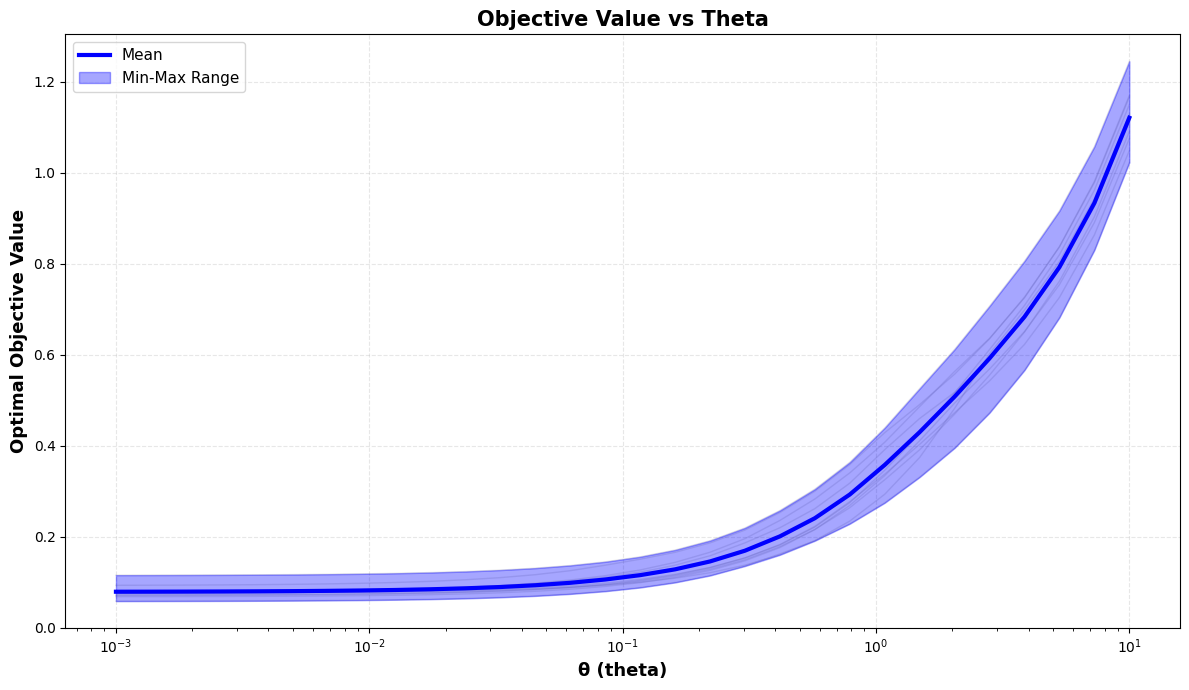}
        \caption{Objective function values.}
        \label{graph of solutions for different node scales and sample sizes-Objective function values}
    \end{subfigure}
    \begin{subfigure}[b]{0.40\textwidth}
        \centering        \includegraphics[width=\textwidth]{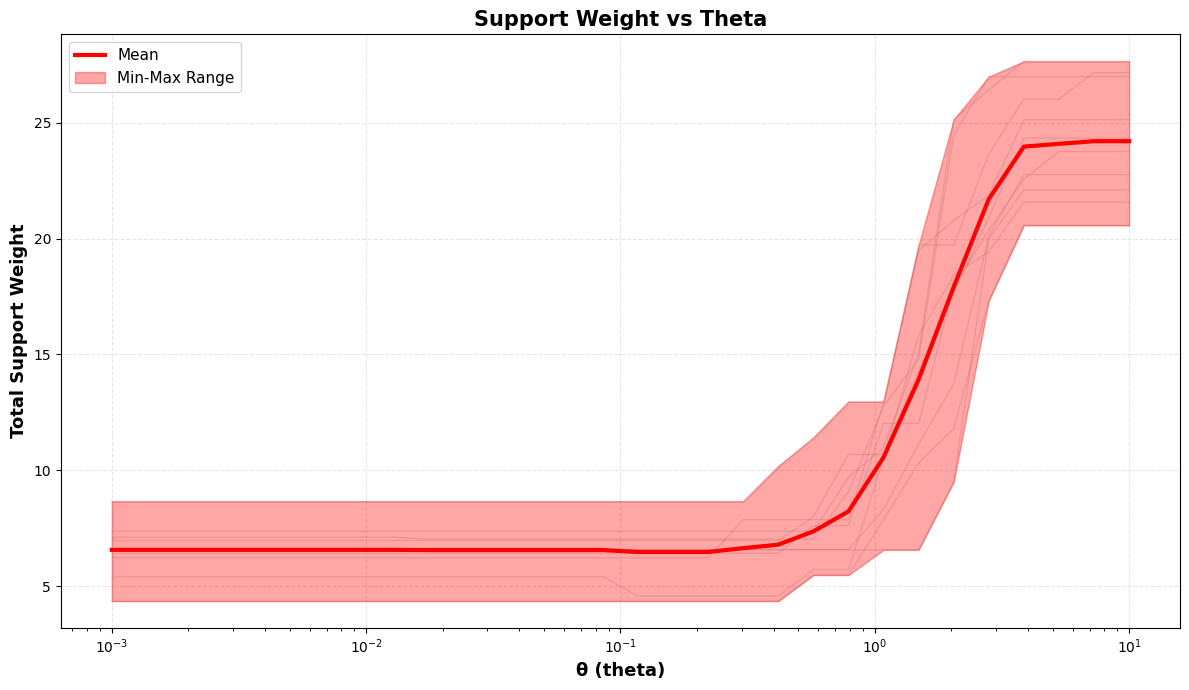}
        \caption{Total maximum clique weight.}
        \label{graph of solutions for different node scales and sample sizes-Total maximum clique weight}
    \end{subfigure}
    \begin{subfigure}[b]{0.40\textwidth}
        \centering        \includegraphics[width=\textwidth]{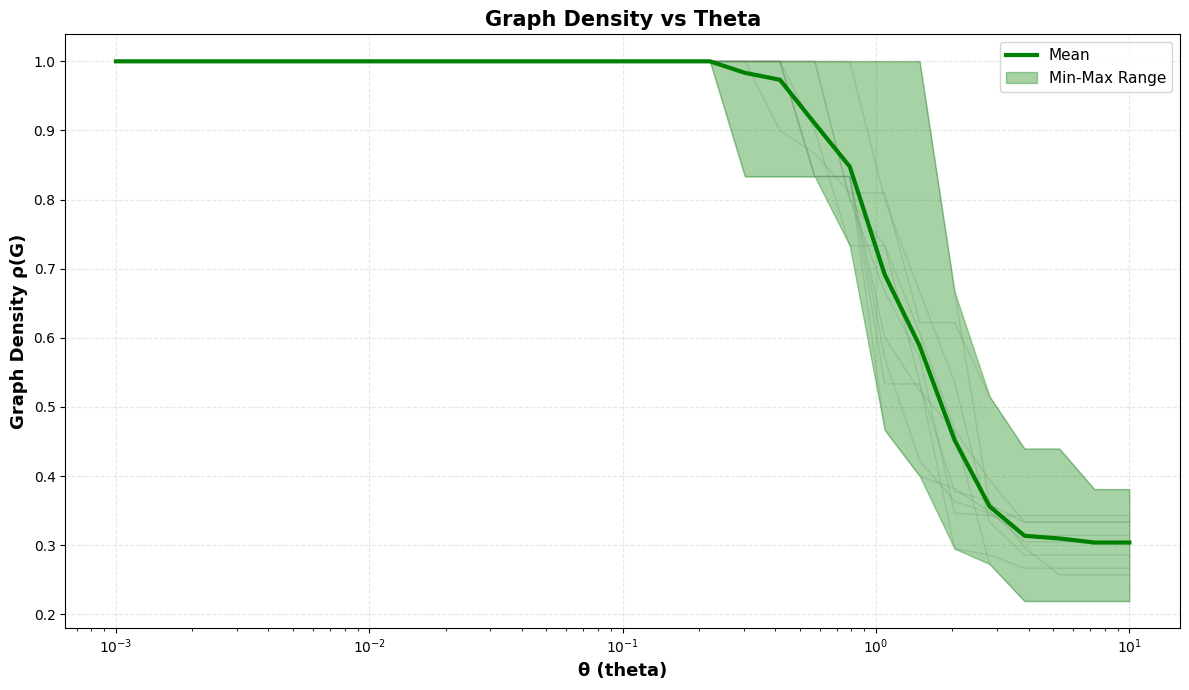}
        \caption{Graph density of final solution graphs.}
        \label{graph of solutions for different node scales and sample sizes-Graph density of final solution graphs}
    \end{subfigure}
    \begin{subfigure}[b]{0.40\textwidth}
        \centering        \includegraphics[width=\textwidth]{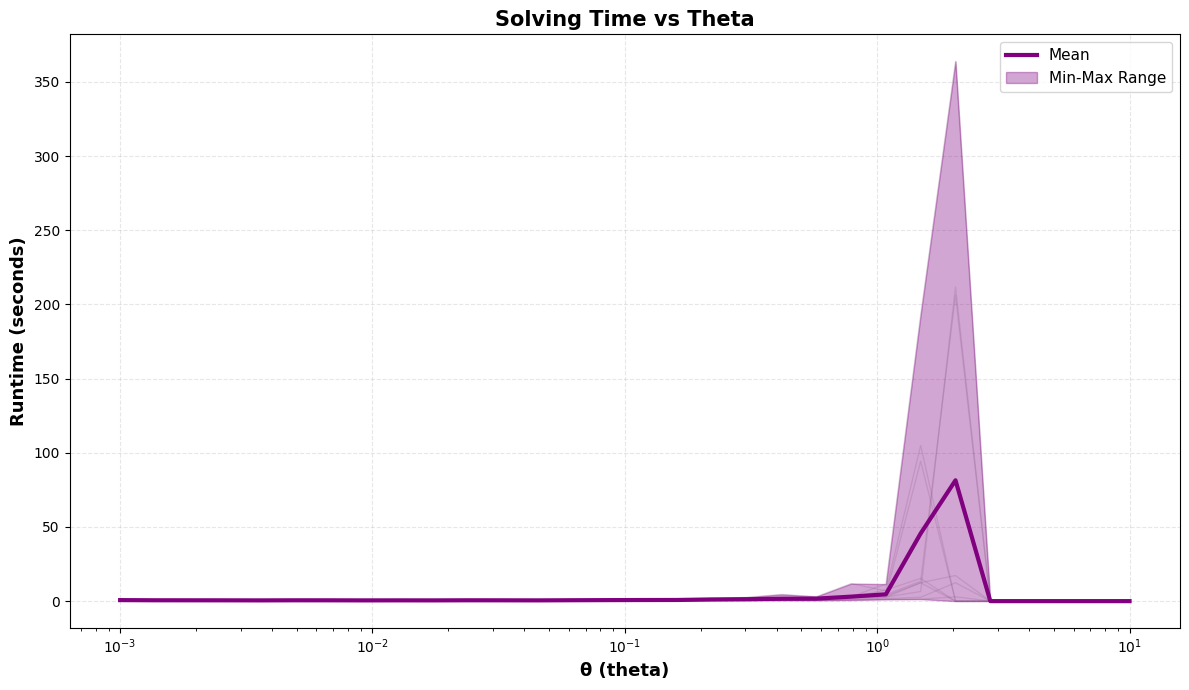}
        \caption{Running time.}
        \label{graph of solutions to decision-independent case under for different node scales and sample sizes-Running time of final solution graphs}
    \end{subfigure}
    \caption{Objective values, maximum clique weights, Graph density and running time for maximum clique solutions under different $\theta$. The trials size is $10$ including $20$ samples within each trial. The node size is $15$. $\beta = 0.001$. The range of $\theta$ is $[10^{-3}, 10]$. (a) Optimal objective value; (b) Maximum clique weight $W(S)$; (c) Graph density $\rho(G)$; (d) Solver runtime in seconds. (Example \ref{Out-of-sample performance to decision-independent case under differential gamma}).}
    \label{graph of solutions to decision-independent case for different node scales and sample sizes-obj and max weights}
\end{figure}

\subsection{
Wasserstein balls with decision-dependent radius}\label{exper}

Illustrating the model covered in Section~\ref{d3}, we will derive a \eqref{D^3RO} version of \eqref{prob:MWCP-re}.
Similar to Section \ref{section numerical simulation with decision-independent radius}, we randomly generate an undirected graph $G=(V,E)$, with $n= |V|$ nodes and $|E|$ edges and a weights vector $\w$ using  \eqref{weight setting}. Then, we compute the weighted adjacency matrix $\Ab$ by~\eqref{weighted adjacency matrix} and form the matrix $\Qb_{\rm nom} := \Eb - \Ab$. 
We sample $\{\widehat \Wb_i\}_{i=1}^N \sim \mathcal W_n(\Ib,n)$, recall Definition~\ref{def:Wishart}, and set 
\begin{equation}\label{sample matrix definition}
\widehat\Qb_{\beta,i}:= \Qb_{\rm nom} + \beta \,\widehat \Wb_i\,.
\end{equation}
for a $\beta > 0$. The sample mean of $\{\widehat\Qb_{\beta,i}\}_{i=1}^N$ is denoted by
\begin{equation*}
\overline\Qb_{\beta}:=\frac{1}{N}\sum_{i=1}^N\widehat\Qb_{\beta,i}\, = \Qb_{\rm nom} +  \frac{\beta}{N} \sum^{N}_{i = 1} \widehat\Wb_{i} =: \Qb_{\rm nom} + \beta \, \Rb.
\end{equation*}

\noindent By construction, 
\begin{equation*}
\Qb_{\rm nom} = \Eb - \Ab =
    \left\{
        \begin{array}{lll}
        &\frac{1}{2 w_{i}},  &\text{ if } i = j \in V,\\
        &0,  &\text{ if }\lbrace i, j \rbrace \in E, \text{ and }\\
        &\frac{1}{2 w_{i}} + \frac{1}{2 w_{j}}, &\text{ otherwise },
        \end{array}
    \right.
\end{equation*}
which implies that
\begin{equation}\label{strict copositive of true expectation}
    \x^{\top} \Qb_{\rm nom} \x = \sum_{i = 1}^n\sum_{j = 1}^n x_{i} x_{j} (\Qb_{\rm nom})_{ij} \geq \sum_{i=1}^n x_i^2 (\Qb_{\rm nom})_{ii} = \frac{1}{2} \sum_{i=1}^n \frac{x_i^2}{w_{i}} >0 \quad\mbox{for all }\x \in \Delta\,,
\end{equation}
and thus $\Qb_{\rm nom}$ is copositive. Since $\overline\Qb_{\beta}$ serves as the sample average approximation of the true expectation of $\widetilde{\Qb}$, it is desirable for $\overline\Qb_{\beta}$ to inherit the structural properties of the true expectation; namely, it should be indefinite and strictly copositive on the standard simplex. Strict copositivity is satisfied by construction: since $\Qb_{\rm nom}$ is strictly copositive (see \eqref{strict copositive of true expectation}) and $\Rb$ is positive semi-definite, their sum remains strictly copositive for any $\beta > 0$:
\begin{equation*}
\x^{\top}\overline\Qb_{\beta} \x = \x^{\top} \Qb_{\rm nom} \x + \beta \x^{\top} \Rb \x > 0, \quad \forall \x \in \Delta.
\end{equation*}
However, the indefiniteness of $\overline\Qb_{\beta}$ is not automatically guaranteed. Since $\Rb$ is positive semi-definite, a large $\beta$ could shift the eigenvalues sufficiently to render $\overline\Qb_{\beta}$ positive definite. To ensure that $\overline\Qb_{\beta}$ remains indefinite, we must limit the range of $\beta$. Using Weyl's inequality, we know that
\begin{equation}\label{weyl inequality 2}
    \lambda_{\min}(\overline\Qb_{\beta}) \leq \lambda_{\min}(\Qb_{\rm nom}) + \lambda_{\max}(\beta \Rb )\, .
\end{equation}
Using \eqref{weyl inequality 2}, we can see that it suffices to take $\beta$ satisfying
\begin{equation*}
    \beta < \frac{-\lambda_{\min}(\Qb_{\rm nom})}{\lambda_{\max}( \Rb )} =: \beta_{\max}\, .
\end{equation*}
Thus, to guarantee that the generated matrix $\overline\Qb_{\beta}$ is indefinite and still strictly copositive, we take $\beta \in (0, \beta_{\max})$ with $\beta_{\max}$ depending on $\Qb_{\rm nom}$ and $\Rb$. If $\beta>\beta_{\max}$, we would generate a possibly convex instance which represents a significantly large (random) deviation from the sample mean, i.e. a larger variability within the sample. It is easy to see that the quadratic form $\x\T\overline\Qb_{\beta}\x$ is convex in $\x$ if and only if
    $$
    \beta\ge \beta_{\rm conv} := -\lambda_{\rm min}(\Rb^{-1/2}\Qb_{\rm nom}\Rb^{-1/2})\, .
    $$

\noindent Given $\gamma \in \mathbb{R}$, we take the radius functions in the following form
$$
\theta_{\beta, \gamma}(\x):=\frac{\gamma}{\x^{\top}\overline\Qb_{\beta}\x}\,.
$$
Using Corollary \ref{Cor rewrite of D3RO under r(x)} again, we can derive the following data-driven problem: 
\begin{equation}\label{D3RO for numerical data-driven}
\underset{\x \in \Delta}{\vphantom{\sup}\inf} \, \underset{\PP \in \mathbb B_{\theta_{\beta, \gamma}(\x),p} (\widehat\PP_N)}{\sup} \,\mathbb E_{\PP}[\x^{\top}\widetilde\Qb\x] =  \inf_{\x \in \Delta} \left [\x^\top \overline\Qb_{\beta} \x + \gamma \, \frac{\x^\top \x}{\x^\top \overline\Qb_{\beta}\x}\right ] \,.
\end{equation}
We solve \eqref{D3RO for numerical data-driven} via the following more suitable equivalent reformulation
\begin{equation*}
    \begin{aligned}
        \underset{\x,y,t}{\min} &&& t + \gamma y\\
        \operatorname{s.t.} &&&  t = \x^{\top} \overline\Qb_{\beta}\x\\
        &&& y t \geq \x^\top\x\\
        &&&\x \in \Delta, \, y \in \mathbb R_+ , \, t \in \mathbb{R}
    \end{aligned}
\end{equation*}
with the non-convex solver {\tt Gurobi}.

\noindent Next, we discuss the influence of $\gamma$ on the convexity of problem \eqref{D3RO for numerical data-driven}. For brevity, we define $v(\x) := \|\x\|^{2}_{2}$ and $u(\x;\beta) := \x^\top \overline\Qb_{\beta}\x$. Then the objective function of \eqref{D3RO for numerical data-driven} can be written as:
\begin{equation}\label{temporary obj of D3RO}
    f(\x; \beta, \gamma) = u(\x; \beta) + \gamma \frac{v(\x)}{u(\x; \beta)}\, .
\end{equation}

\begin{proposition}\label{convexity of the final objective}Assume that $\overline\Qb_{\beta}\in {\mathcal S}^{n}$ is positive definite with eigenvalues $0 < \lambda_{\min}(\overline\Qb_{\beta}) \leq \lambda_{\max}(\overline\Qb_{\beta})$. Define the constant $C(\overline{\Qb}_{\beta}, n)$ as:
\begin{equation*}
    C(\overline\Qb_{\beta}, n) := n \left( \frac{2}{\lambda_{\min}(\overline\Qb_{\beta})} + \frac{10 \lambda_{\max}(\overline\Qb_{\beta})}{\lambda_{\min}(\overline\Qb_{\beta})^2} + \frac{8 \lambda_{\max}(\overline\Qb_{\beta})^2}{\lambda_{\min}(\overline\Qb_{\beta})^3} \right)\, .
\end{equation*}
Then the objective function $f(\x;\beta,\gamma)$ is convex in $\x$ across $\Delta$ if
\begin{equation*}
    \gamma \leq \gamma_{\rm conv} := \frac{2 \lambda_{\min}(\overline\Qb_{\beta})}{C(\overline\Qb_{\beta}, n)}\, .
\end{equation*}
\end{proposition}
\begin{proof} 
Straightforward calculation yields the gradient of $f(\x; \beta, \gamma)$ as:
\begin{equation*}
    \nabla_{\x} f(\x; \beta, \gamma) = 2 \overline\Qb_{\beta} \x + 2 \gamma \left[ \frac{u(\x; \beta) \x - v(\x) \overline\Qb_{\beta} \x}{u(\x; \beta)^2} \right].
\end{equation*}
The Hessian matrix of $f(\x; \beta, \gamma)$ is given by:
\begin{equation}\label{Hessian matrix for obj}
    \nabla^{2}_{\x\x} f(\x; \beta, \gamma) = 2 \overline\Qb_{\beta}+ \frac{2\gamma}{u(\x; \beta)^3} \Ab(\x; \beta),
\end{equation}
where the symmetric matrix $\Ab(\x; \beta)$ is derived as follows:
\begin{equation}\label{matrix_A_corrected}
    \begin{split}
        \Ab(\x; \beta) &:= u(\x; \beta)^2 \Ib - u(\x; \beta) \left[ 2 \left( \overline\Qb_{\beta} \x \x^{\top} + \x \x^{\top} \overline\Qb_{\beta} \right) + v(\x) \overline{\Qb}^{\beta} \right] + 4 v(\x) \overline\Qb_{\beta} \x \x^{\top} \overline\Qb_{\beta} \\
        &= u(\x; \beta)^2 \Ib - 2 u(\x; \beta) \left( \overline\Qb_{\beta} \x \x^{\top} + \x \x^{\top} \overline\Qb_{\beta} \right) - v(\x) u(\x; \beta) \overline\Qb_{\beta} + 4 v(\x) \overline\Qb_{\beta} \x \x^{\top} \overline\Qb_{\beta}\, .
           \end{split}
\end{equation}
For any $\x \in \Delta$, we have the bounds:
\begin{equation*}
    \frac{1}{n} \leq \|\x\|^2_2 \leq 1, \quad \lambda_{\min}(\overline\Qb_{\beta})\|\x\|^2_2 \leq u(\x; \beta) \leq \lambda_{\max}(\overline\Qb_{\beta})\|\x\|^2_2, \quad \|\overline\Qb_{\beta}\x\|_2 \leq \lambda_{\max}(\overline\Qb_{\beta})\|\x\|_2.
\end{equation*}
To ensure convexity, we require $\nabla^2 f(\x) \succeq 0$. Since $2\overline\Qb_{\beta} \succeq 2\lambda_{\min}(\overline\Qb_{\beta}) \Ib$, it suffices to control the spectral norm of the perturbation term caused by $\gamma$.

\noindent By applying the triangle inequality to the components of $\nabla^2_{\x\x}(v(\x)/u(\x; \beta))$ derived from \eqref{matrix_A_corrected}, we obtain the following upper bound:
\begin{equation*}
    \left\| \frac{2}{u(\x; \beta)^3} \Ab(\x; \beta) \right\|_2 \leq \frac{1}{\|\x\|^2_2} \left( \frac{2}{\lambda_{\min}(\overline\Qb_{\beta})} + \frac{10 \lambda_{\max}(\overline{\Qb}_{\beta})}{\lambda_{\min}(\overline\Qb_{\beta})^2} + \frac{8 \lambda_{\max}(\overline\Qb_{\beta})^2}{\lambda_{\min}(\overline\Qb_{\beta})^3} \right)\, .
\end{equation*}
Since $1/\|\x\|^2_2 \leq n$ for $\x \in \Delta$, the condition ensures that $2\lambda_{\min}(\overline\Qb_{\beta}) - \gamma C(\overline\Qb_{\beta}, n) \geq 0$, guaranteeing the minimum eigenvalue of the total Hessian is non-negative. \end{proof}
\begin{remark}
When $\overline\Qb_{\beta}$ is indefinite but strictly copositive (i.e., $u(\x; \beta) > 0$ for $\x \in \Delta$, but $\lambda_{\min}(\overline\Qb_{\beta}) < 0$), the sufficient condition for convexity derived for the positive definite case fails. Specifically, the condition derived for the positive definite case,
\begin{equation*}
    \lambda_{\min}(\nabla^{2}_{\x\x} f(\x; \beta, \gamma)) \geq 2\lambda_{\min}(\overline\Qb_{\beta}) - \gamma \left\| \frac{2}{u(\x; \beta)^3} \Ab(\x; \beta) \right\|_2 \geq 0,
\end{equation*}
becomes impossible to satisfy because $\lambda_{\min}(\overline\Qb_{\beta}) < 0$ and $\gamma > 0$.

\noindent We focus our spectral analysis strictly on the asymptotic behavior near the zero-level set $\mathcal{Z} = \{\x : u(\x; \beta) = 0\}$. The rationale is twofold:
\begin{enumerate}[label=(\roman*)]
    \item Role of Regularization: The regularization term $\gamma/u(\x;\beta)$ acts as a barrier function. Its influence is negligible in the interior where $u(\x; \beta)$ is large (where $\nabla^2_{\x\x} f(\x; \beta,\gamma) \approx 2\overline\Qb_{\beta}$, which is essentially indefinite).
    \item Singularity Dominance: As $u(\x;\beta) \to 0$, the regularization derivatives grow unbounded, potentially overriding the negative eigenvalues of $\overline{\Qb}_{\beta}$.
\end{enumerate}

\noindent To analyze this, we reformulate the Hessian \eqref{Hessian matrix for obj} by organizing terms by powers of $u$:
\begin{equation*} 
    \nabla^{2}_{\x\x} f(\x;\beta, \gamma) = \underbrace{2 \overline{\Qb}_{\beta}}_{\text{Constant}} + \underbrace{\frac{2\gamma}{u} \Ib}_{\mathcal O(u^{-1})} - \underbrace{\frac{2\gamma}{u^2} \mathbf{K}(\x)}_{\mathcal O(u^{-2})} + \underbrace{\frac{8\gamma v(\x)}{u^3} (\overline{\Qb}_{\beta}\x)(\overline{\Qb}_{\beta}\x)^\top}_{\mathcal O(u^{-3})},
\end{equation*}
where $u$ denotes $u(\x; \beta)$, $\mathbf{K}(\x) = v(\x)\overline{\Qb}_{\beta} + 2(\overline{\Qb}_{\beta}\x\x^\top + \x\x^\top\overline{\Qb}_{\beta})$, and $\gamma > 0$. Consider a sequence $\x^{(k)} \to \bar{\x} \in \mathcal{Z}$. We assume the non-degenerate case where the gradient $\overline{\Qb}_{\beta}\x^{(k)} \nrightarrow \mathbf{0}$. Furthermore, since $\x$ resides on the standard simplex, $v(\x^{(k)}) = \|\x^{(k)}\|_2^2$ is strictly bounded away from zero ($v(\x^{(k)}) \asymp 1$). To determine the definiteness, we perform a spectral decomposition relative to the gradient direction. Let $\mathbf{p} = \overline{\Qb}_{\beta}\x$. We decompose an arbitrary direction $\mathbf{d}$ into a component along $\mathbf{p}$ and an orthogonal component $\mathbf{z}$:
$$ \mathbf{d} = \alpha \frac{\mathbf{p}}{\|\mathbf{p}\|_2} + \mathbf{z}, \quad \text{where } \mathbf{z}^\top \mathbf{p} = 0. $$
The quadratic form $q_H(\mathbf{d}) = \mathbf{d}^\top \nabla^2_{\x\x} f(\x;\beta,\gamma) \mathbf{d}$ behaves as follows:

\begin{itemize}
    \item \textbf{Along the Gradient ($\alpha \neq 0$):} The curvature is dominated by the rank-1 term:
    $$ q_H(\mathbf{z}) \approx \alpha^2 \frac{8\gamma v(\x) \|\mathbf{p}\|^2_2}{u^3} > 0. $$
    This creates a strong convex barrier repelling the path orthogonally from the zero-level set.

    \item \textbf{In the Orthogonal Subspace ($\alpha = 0, \mathbf{d} = \mathbf{z}$):} The dominating $\mathcal O(u^{-3})$ term vanishes strictly since $\mathbf{z} \perp \mathbf{p}$. The local convexity is determined by the competition between the stabilizing isotropic term and the indefinite cross-term:
    $$ q_H(\mathbf{z}) \approx \frac{2\gamma}{u} \|\mathbf{z}\|^2_2 - \frac{2\gamma}{u^2} \mathbf{z}^\top \mathbf{K}(\x) \mathbf{z} = \frac{2\gamma}{u} \|\mathbf{z}\|^2_2 - \frac{2\gamma}{u^2}v(\x) \mathbf{z}^\top \overline{\Qb}_{\beta}\mathbf{z}. $$
\end{itemize}
The expansion in the orthogonal subspace reveals a structural limitation. If there exists a direction $\mathbf{z} \perp \overline{\Qb}_{\beta}\x$ such that $\mathbf{z}^\top \overline{\Qb}_{\beta}\mathbf{z} > 0$, the negative curvature term scales as $\mathcal O(u^{-2})$, which asymptotically overwhelms the stabilizing $\mathcal O(u^{-1})$ term regardless of the value of $\gamma$.
Consequently, no finite $\gamma > 0$ can guarantee global positive definiteness arbitrarily close to the zero-level set in the indefinite case. The regularization effectively stabilizes the gradient direction but may leave tangent directions with negative curvature, potentially creating saddle points near the boundary rather than a strictly convex area.

\end{remark}

\begin{example}[Solutions under different $\beta$ and $\gamma$]\label{solutions under different beta and gamma}
In this numerical experiment, we investigate the sensitivity of the solutions to problem \eqref{D3RO for numerical data-driven} with respect to the ambiguity parameters $\beta$ and $\gamma$. We generate an undirected graph $G = (V, E)$ with $n= 12$ nodes, $|E| = 17$ edges and sample size $N=50$. As a benchmark, we first solve the deterministic problem (setting $\beta = 0$ and $\gamma = 0$) in \eqref{D3RO for numerical data-driven}, i.e.,
\begin{equation*}
    \underset{\x \in \Delta}{\min}\,\x^{\top}\Qb_{\rm nom}\x\,,
\end{equation*}
yielding the baseline solution shown in Figure \ref{graph of the original basic graph and its true solution}.

\noindent Figure \ref{graph of solutions under different beta and gamma} illustrates the evolution of the solution topology under varying robustness levels. We define a parameter grid with $\beta \in \{0.005, 0.08, 0.2\}$ and $\gamma \in \{0.005, 0.05, 0.3\}$. The results demonstrate that the solution structure is sensitive to both regularization terms. In the regime of low uncertainty (e.g., $\beta=0.005, \gamma=0.005$), the robust solution preserves the sparse structure of the baseline, identifying a specific subset of critical vertices. However, as either $\beta$ or $\gamma$ increases, the solution becomes progressively more conservative, selecting a significantly denser set of edges and vertices. Simultaneously, the objective value increases monotonically, reflecting the higher cost associated with robust protection. When the ambiguity parameters are sufficiently large (e.g., $\beta = 0.2$ or $\gamma = 0.3$), the solution saturates to a trivial state, encompassing almost the entire graph to mitigate the worst-case scenario. Notably, in this saturated regime, the solution topology becomes invariant to the other parameter, indicating that the robust constraints have fully dominated the optimization landscape. Furthermore, the selected noise levels effectively span the structural regimes defined by the spectral thresholds $\beta_{\max} \approx 0.127$ and $\beta_{\text{conv}} \approx 0.136$. The values $\beta \in \{0.005, 0.08\}$ satisfy $\beta < \beta_{\max}$, maintaining the indefiniteness of $\overline{\Qb}_{\beta}$ and favoring sparse, boundary-localized solutions typical of non-convex StQPs. Conversely, the case $\beta = 0.2$ exceeds $\beta_{\rm conv}$, rendering the quadratic form positive definite. This convexity forces the optimal solution away from the vertices toward the interior of the simplex, resulting in a loss of sparsity that shows as the saturated, high-density topology observed in Figure \ref{graph of solutions under different beta and gamma}.

\noindent Table~\ref{tab:eigen_constants} quantifies the spectral properties governing the convexity of the objective function defined in \eqref{temporary obj of D3RO}. For $\beta \in \{0.005, 0.08\}$, the negative values of $\lambda_{\min}(\overline{\Qb}_{\beta})$ confirm the indefiniteness of the unregularized problem. In the positive definite regime ($\beta=0.2$), the theoretical conservative convexity bound is calculated as $\gamma_{\rm conv} \approx 1.02 \times 10^{-4}$ according to Proposition \ref{convexity of the final objective}. This value is orders of magnitude smaller than the operational radii used in our experiments (e.g., $\gamma=0.3$), indicating that even when the nominal quadratic form is convex, the rational regularization term induces significant non-convexity for non-trivial ambiguity radii.

\end{example}

\begin{figure}[H]
    \centering
    \includegraphics[scale = 0.25]{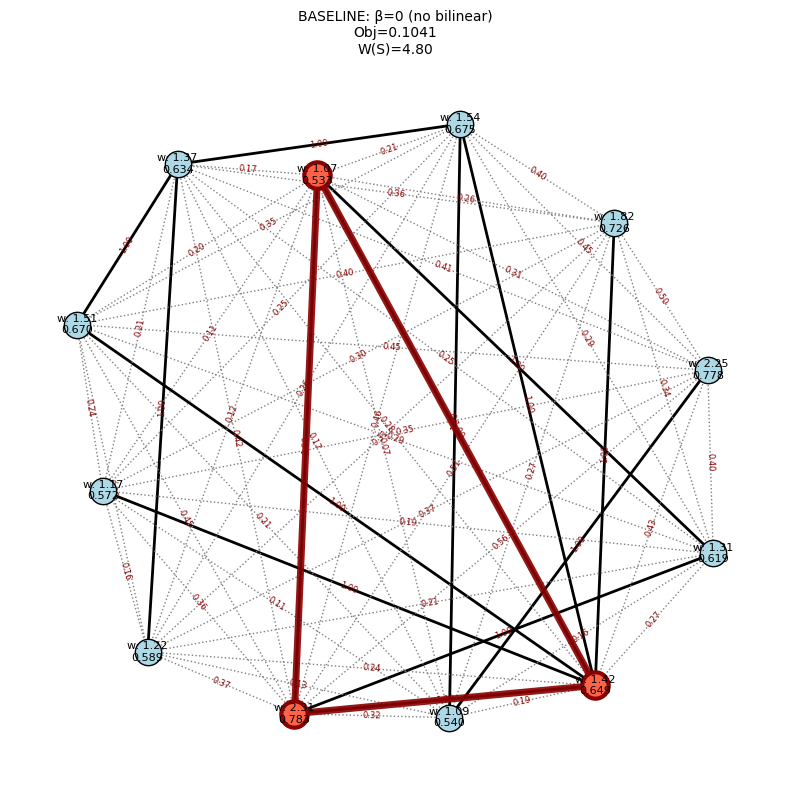}
    \caption{The baseline deterministic solution ($\beta=0, \gamma=0$) with an objective value of $0.1041$. The sample size is $50$ and nodes size is $12$. $\beta_{\max} = 0.127235$, $\beta_{\rm conv}=0.135836$. (Example \ref{solutions under different beta and gamma}).}
    \label{graph of the original basic graph and its true solution}
\end{figure}
\begin{figure}[H]
\centering
\includegraphics[scale = 0.23]{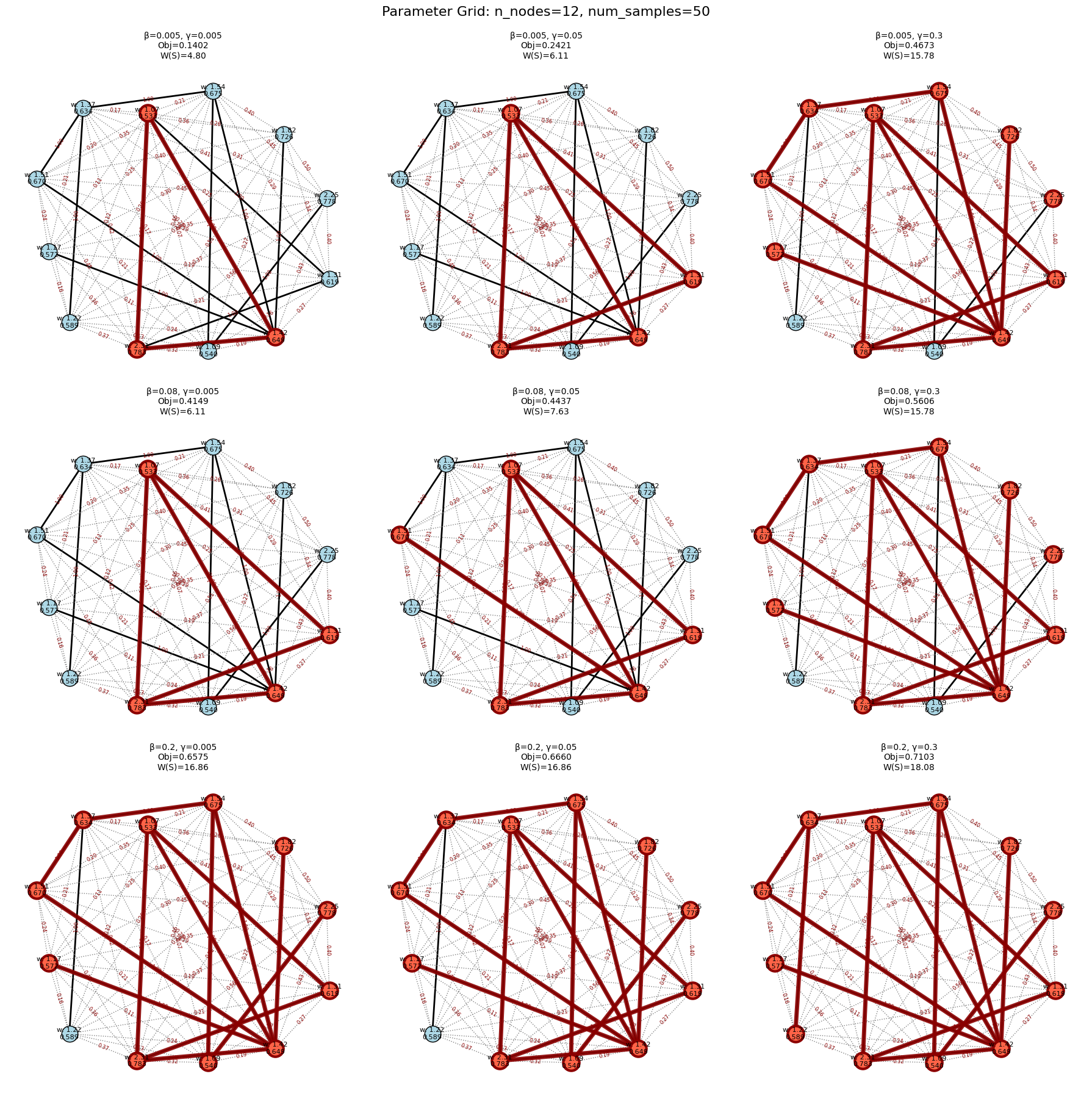}
\caption{Solutions under varying $\beta$ (rows) and $\gamma$ (columns). The parameter $\beta \in \{ 0.005, 0.08, 0.2 \}$ increases from top to bottom, and $\gamma \in \{ 0.005, 0.05, 0.3 \}$ increases from left to right. Solution subgraphs are highlighted in red. $\beta_{\max} = 0.127235$, $\beta_{\text{conv}}=0.135836$. (Example \ref{solutions under different beta and gamma}).}
\label{graph of solutions under different beta and gamma}
\end{figure}
\begin{table}[H]
\centering
\begin{tabular}{@{}lcccc@{}}
\toprule
\textbf{Parameter $\beta$} & $\boldsymbol{\lambda_{\min}}(\overline{\Qb}_{\beta})$ & $\boldsymbol{\lambda_{\max}}(\overline{\Qb}_{\beta})$ & $\boldsymbol{C}(\overline{\Qb}_{\beta}\boldsymbol{, n)}$ & $\boldsymbol{\gamma_{\rm conv}}$ \\
\midrule
0.005 & -1.7883 & 6.7471 & - & - \\
0.08  & -0.7535 & 7.6342 & - & - \\
0.2   & \phantom{-}0.8188 & 9.0588 & \phantom{-}$1.59995 \times 10^{4}$ & $1.0236 \times 10^{-4}$ \\
\bottomrule
\end{tabular}
\caption{Eigenvalues of $\overline{\Qb}_{\beta}$ and calculated constants $C(\overline{\Qb}_{\beta}, n)$ and $\gamma_{\rm conv}$ for different $\beta$ values ($n=12$). $C(\overline{\Qb}_{\beta}, n)$ and $\gamma_{\rm conv}$ are defined in Proposition \ref{convexity of the final objective}. $\lambda_{\min}(\overline{\Qb}_{\beta})$ and $\lambda_{\max}(\overline{\Qb}_{\beta})$ represent the minimum and maximum eigenvalues of $\overline{\Qb}_{\beta}$, respectively.}
\label{tab:eigen_constants}
\end{table}

\begin{example}[Finite sample performance under different $\gamma$]\label{Out-of-sample performance under differential gamma}
In this experiment, we evaluate the out-of-sample performance and computational characteristics of the proposed method under varying ambiguity scaling parameters $\gamma$. We conduct $20$ independent trials based on a same basic matrix $\Qb_{\rm nom}$ in \eqref{sample matrix definition}, each generating a graph with $|V|=10$ nodes and $N=50$ samples, with $\beta$ fixed at $0.01$.

\noindent Figure \ref{graph of solutions for different node scales and sample sizes-graph Objective function values} shows that the objective value increases exponentially as $\gamma$ grows, reflecting the higher cost of robustness against larger distributional ambiguity. The influence of $\gamma$ on the solution topology is evident in Figures \ref{graph of solutions for different node scales and sample sizes-graph clique weight} and \ref{graph of solutions for different node scales and sample sizes-graph density}. In the low-ambiguity regime ($\gamma \leq 10^{-2}$), the mean maximum clique weight remains stable, and the graph density is approximately $1.0$. This indicates that for small $\gamma$, the solver successfully identifies a complete subgraph (a clique). As $\gamma$ increases from $0.01$ to $0.5$, the solution expands to satisfy robustness requirements: the total weight increases monotonically (Figure \ref{graph of solutions for different node scales and sample sizes-graph clique weight}), while the density drops sharply to approximately $0.3$ (Figure \ref{graph of solutions for different node scales and sample sizes-graph density}). This density decrease implies that the solution set $S$ is expanding to include non-connected vertices, eventually encompassing the entire graph when $\gamma \geq 0.5$.

\noindent Finally, Figure \ref{graph of solutions for different node scales and sample sizes-running time} highlights the computational complexity relative to $\gamma$. The solving time remains negligible for $\gamma < 10^{-1}$ but exhibits a sharp peak in the interval $\gamma \in [0.2, 0.9]$. This \textit{phase transition} suggests a competitive interaction between the quadratic term and the rational term in the objective function \eqref{D3RO for numerical data-driven}. When $\gamma$ is either very small or very large, one term dominates, simplifying the optimization landscape for the solver. However, in the intermediate regime ($\gamma \approx 0.4$), the two non-convex terms exert comparable influence, significantly increasing the computational effort required for the global solver to close the optimality gap.

\end{example}

\begin{figure}[H]
    \centering
    \begin{subfigure}[b]{0.40\textwidth}
        \centering        \includegraphics[width=\textwidth]{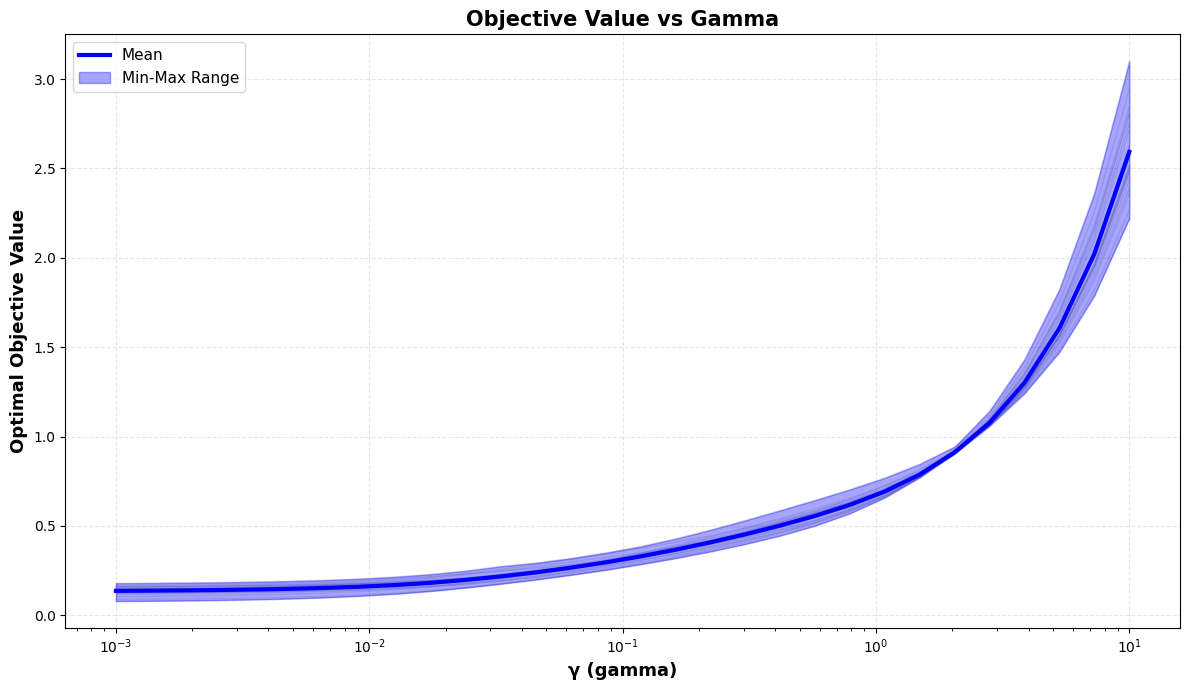}
        \caption{Objective function values.}
        \label{graph of solutions for different node scales and sample sizes-graph Objective function values}
    \end{subfigure}
    \begin{subfigure}[b]{0.40\textwidth}
        \centering        \includegraphics[width=\textwidth]{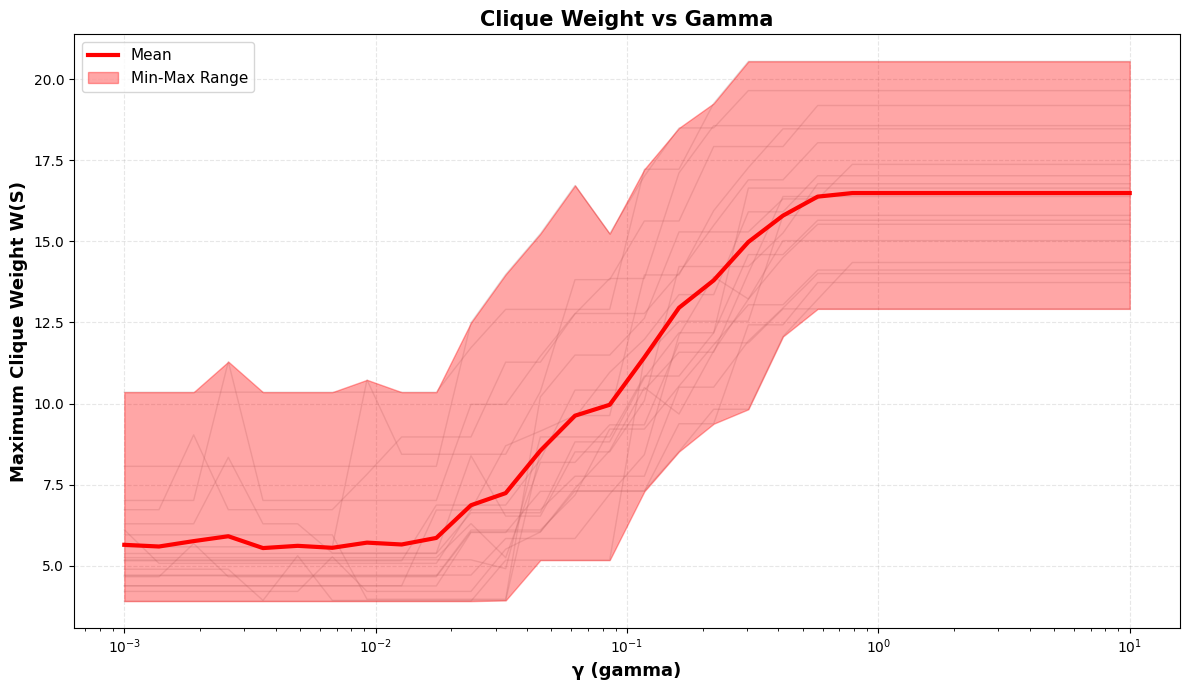}
        \caption{Total maximum clique weight.}
        \label{graph of solutions for different node scales and sample sizes-graph clique weight}
    \end{subfigure}
    \begin{subfigure}[b]{0.40\textwidth}
        \centering        \includegraphics[width=\textwidth]{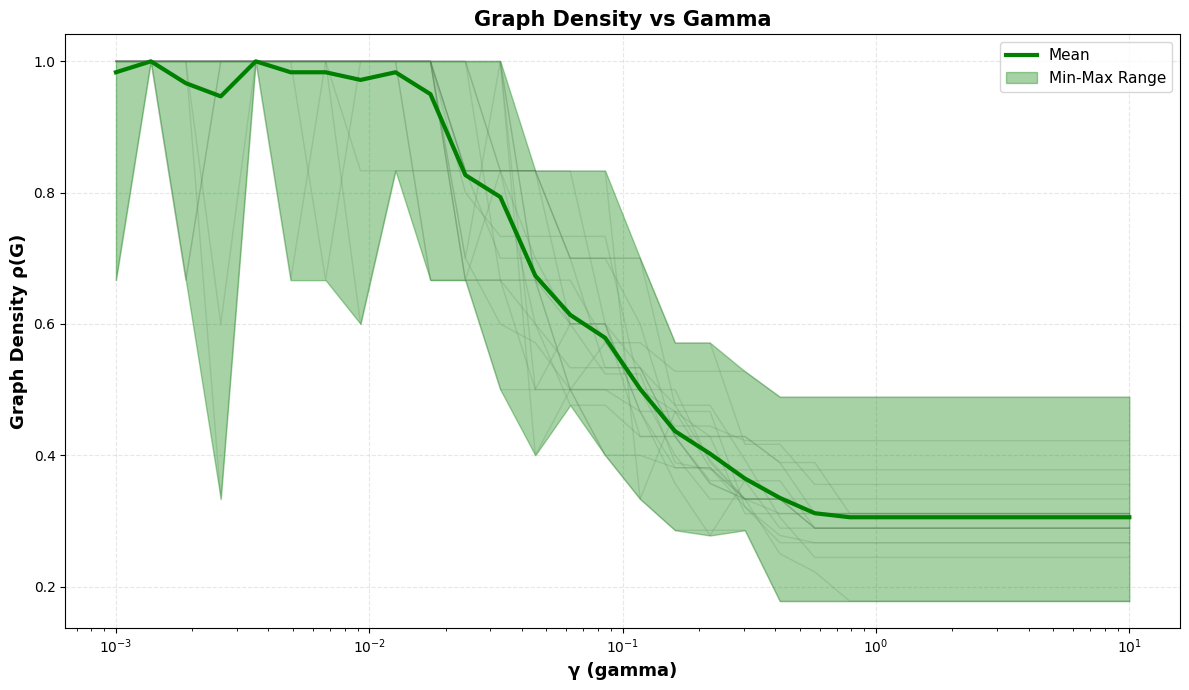}
        \caption{Graph density of final solution graphs.}
        \label{graph of solutions for different node scales and sample sizes-graph density}
    \end{subfigure}
    \begin{subfigure}[b]{0.40\textwidth}
        \centering        \includegraphics[width=\textwidth]{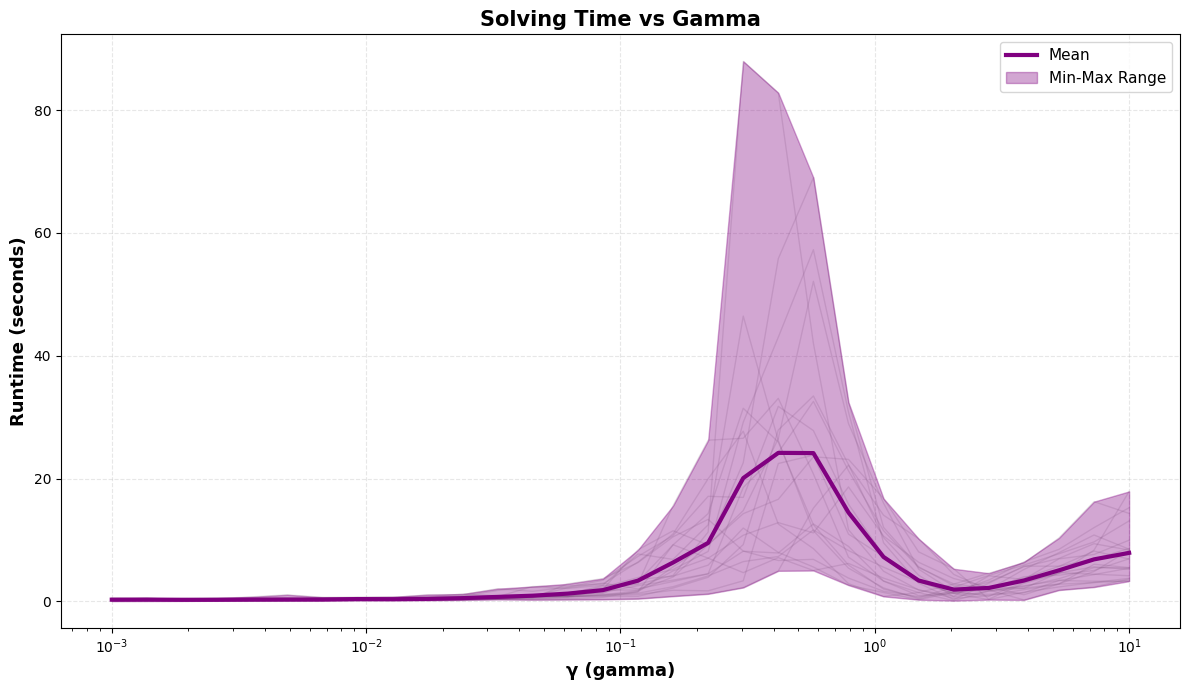}
        \caption{Running time.}
        \label{graph of solutions for different node scales and sample sizes-running time}
    \end{subfigure}
    \caption{Performance metrics under varying $\gamma$ ($\beta = 0.01$, sample size $N=50$). The shaded regions indicate the min-max range across 20 trials, and the solid lines represent the mean. (a) Optimal objective value; (b) Maximum clique weight $W(S)$; (c) Graph density $\rho(G)$; (d) Solver runtime in seconds. (Example \ref{Out-of-sample performance under differential gamma}).}
    \label{graph of solutions for different node scales and sample sizes-obj and max weights}
\end{figure}

\begin{example}[Solutions for different node scales and sample sizes]\label{Solutions for different node scales and sample sizes}
In this example, we examine the sensitivity of the solutions to \eqref{D3RO for numerical data-driven} with respect to the graph dimension (node size $n = |V|$) and data availability (sample size $N$). We fix the ambiguity parameters at $\beta = 0.01$ and $\gamma = 0.01$.

\noindent The first row in Figure~\ref{subfig:varying_nodes} demonstrates the scalability of the method across varying graph sizes. We observe that the solver consistently identifies dense subgraphs (cliques) as the node size $n$ increases. While the optimal objective value decreases with the expansion of the graph dimension, the total weight of the identified clique, $W(S)$, increases significantly from $n=10$ to $n=15$ and subsequently stabilizes around $9.8$ for $n=25$. This indicates that the method effectively localizes the maximum weight clique even within larger search spaces.

\noindent The second row in Figure \ref{subfig:varying_samples} illustrates the impact of sample size on solution quality for a fixed graph of size $n=20$. We observe a distinct convergence behavior: at a low sample size ($N=20$), the solution yields a clique weight of $6.17$. As the sample size increases to $N=50$, the solution converges to a more optimal structure with a weight of $6.64$, which remains stable at $N=100$. This suggests that the proposed formulation is data-efficient, capable of achieving robust, high-quality solutions with a moderate number of samples.
\end{example}

\begin{figure}[H]
    \centering
    \begin{subfigure}{0.87\textwidth}
        \includegraphics[scale = 0.27]{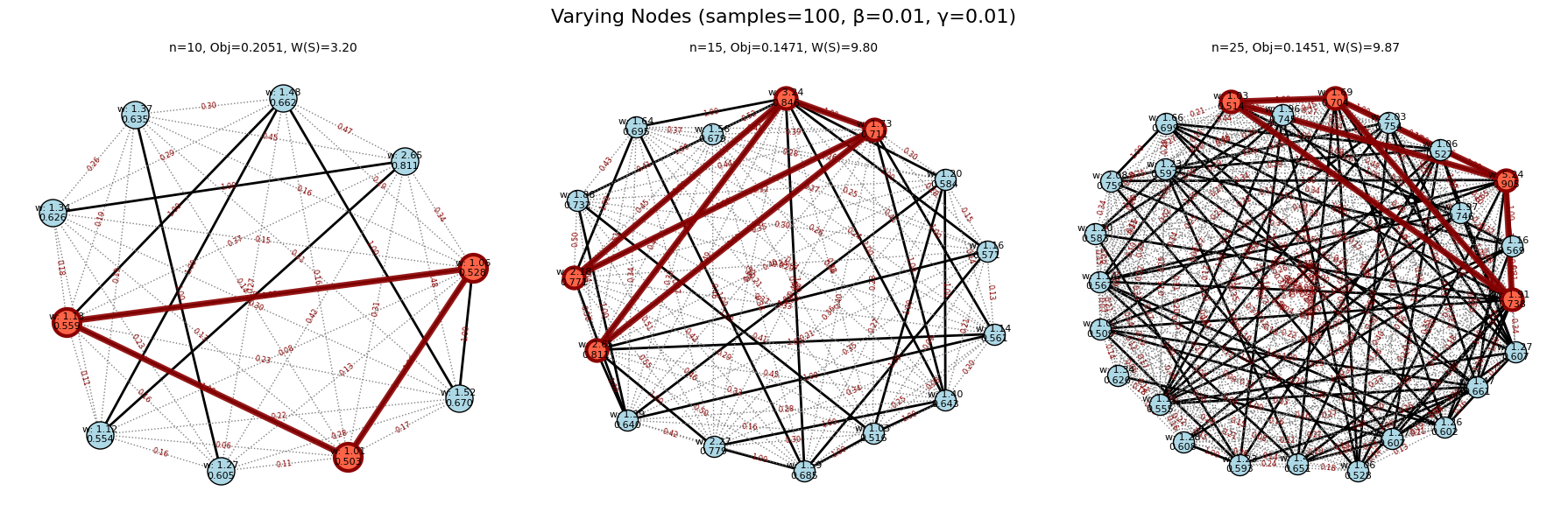}
        \caption{Varying node sizes ($N=100$).}
        \label{subfig:varying_nodes}
    \end{subfigure}
    \\
    \hfill
    \begin{subfigure}{0.93\textwidth}
        \includegraphics[scale = 0.27]{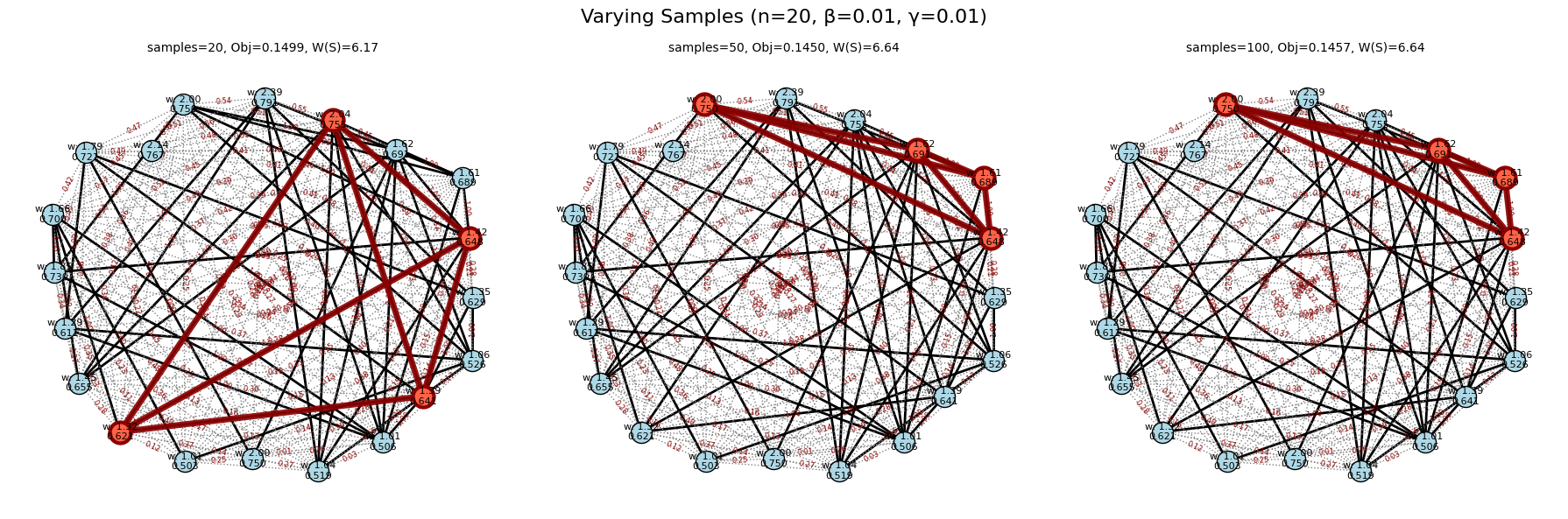}
        \caption{Varying sample sizes ($n=20$).}
        \label{subfig:varying_samples}
    \end{subfigure}
    \caption{Solutions for different node scales and sample sizes ($\gamma = 0.01$ and $\beta = 0.01$). The first row has fixed sample size $N=100$ and varying node sizes $n \in \{10, 15, 25\}$. The second row has fixed node size $n=20$ and varying sample sizes $N \in \{20, 50, 100\}$. (Example \ref{Solutions for different node scales and sample sizes}).}
    \label{graph of solutions for different node scales and sample sizes}
    \end{figure}

\begin{example}[Frequency of solutions through several trials under different $\beta$ and $\gamma$]\label{Frequency of solutions through several trials under different beta and gamma}

Figure~\ref{graph of frequency of solutions through several trials under different beta and gamma} visualizes the stability and topology of the solutions to \eqref{D3RO for numerical data-driven} by plotting the frequency of selected edges and nodes over repeated trials. The grid illustrates the interplay between the noise level $\beta$ (rows) and the ambiguity radius $\gamma$ (columns).

\noindent A clear trend is observed regarding the structural evolution driven by both parameters. In the first column, where the ambiguity radius is small ($\gamma = 0.005$), the solutions at low noise ($\beta = 0.005$) are sparse and exhibit weak connectivity. However, as $\beta$ increases to $0.2$ (moving down the first column), we observe a distinct transition towards a denser topology. Contrary to the expectation that higher noise yields more scattered solutions, the emergence of thicker, darker edges indicates that high noise levels induce a form of implicit regularization, forcing the solver to converge onto a consistent, highly connected subgraph similar to that observed in robust regimes.

\noindent As $\gamma$ increases to $0.05$ and $0.3$ (second and third columns), the solution topology rapidly stabilizes into this dense clique-like configuration. The edges in these columns are significantly thicker and darker, indicating that the same set of edges is consistently selected across almost all trials (high frequency). Notably, in these high-$\gamma$ regimes, the solution structure becomes largely invariant to the noise level $\beta$. Comparing the plots within the third column, the graph topology remains virtually identical regardless of whether $\beta$ is $0.005$ or $0.2$. This demonstrates that a sufficiently large ambiguity radius $\gamma$ dominates the influence of the sample noise $\beta$, locking the solver into a robust configuration that maximizes the worst-case objective.
 




\noindent To quantify the computational performance, we report three standard metrics provided by the {\tt Gurobi} solver. The \textit{Best Objective} corresponds to the objective value of the incumbent solution, serving as an upper bound for the minimization problem. The \textit{Best Bound} represents the best theoretical lower bound derived from the branch-and-bound relaxation. Finally, the \textit{Gap} denotes the relative optimality gap, defined as $\frac{|\text{Best Objective} - \text{Best Bound}|}{|\text{Best Objective}|}$.

\noindent Several key conclusions may be drawn from the solver statistics presented in Figure~\ref{Box plots of best objective, best bound and gap} and Table~\ref{tab:solver_stats_beta_gamma}. Overall, {\tt Gurobi} demonstrated exceptional performance in tackling this non-convex optimization problem. The average optimality gap remained consistently below $1\%$ across all parameter configurations, indicating that the feasible solutions found were of high quality and proximal to the global optimum.

\noindent The influence of parameter $\gamma$ remains clear and monotonic: with $\beta$ fixed, both the best objective function value and the best bound value increase as $\gamma$ rises from 0.005 to 0.3. This trend reinforces the interpretation that $\gamma$ acts primarily as a scaling parameter or a cost coefficient within the model structure.

\noindent The role of parameter $\beta$ reveals more complex structural properties. Unlike previous observations where $\beta$ had a negligible impact, the new data indicates a strong positive correlation: increasing $\beta$ leads to a significant elevation in the average objective value. Furthermore, $\beta$ appears to govern the solution variability. At $\beta=0.005$, the standard deviation of the objective is minimal, but it grows significantly as $\beta$ increases to 0.08 and 0.2. This suggests that larger $\beta$ values introduce higher sensitivity or instance-dependency into the optimization landscape.

\noindent Notably, a distinct interaction effect is observed regarding the stability of the optimality gap. As shown in the figure, the standard deviation of the gap is conspicuously larger (approximately an order of magnitude higher) in the regime where both parameters are relatively small (specifically, combinations of $\beta \in \{0.005, 0.08\}$ and $\gamma \in \{0.005, 0.05\}$). In contrast, scenarios involving larger parameter values (e.g., $\gamma=0.3$ or $\beta=0.2$) exhibit minimal gap variance, despite a slight increase in the average gap magnitude. This phenomenon implies that higher values of $\beta$ and $\gamma$ may exert a stabilizing effect on the convergence behavior —possibly by convexifying the solution space or imposing stronger penalties— thereby enabling the solver to consistently terminate with a uniform level of precision. Conversely, lower parameter settings appear to yield a \textit{looser} landscape, resulting in distinct fluctuations in the final convergence quality.

\noindent In summary, the numerical experiments illustrate the distinct roles of the parameters: $\gamma$ functions as a linear scaling factor, $\beta$ impacts the volatility of the solution search, and their combined magnitude correlates with the consistency of the solver's convergence.

\end{example}

 \begin{figure}[H]
    \centering
    \includegraphics[scale = 0.20]{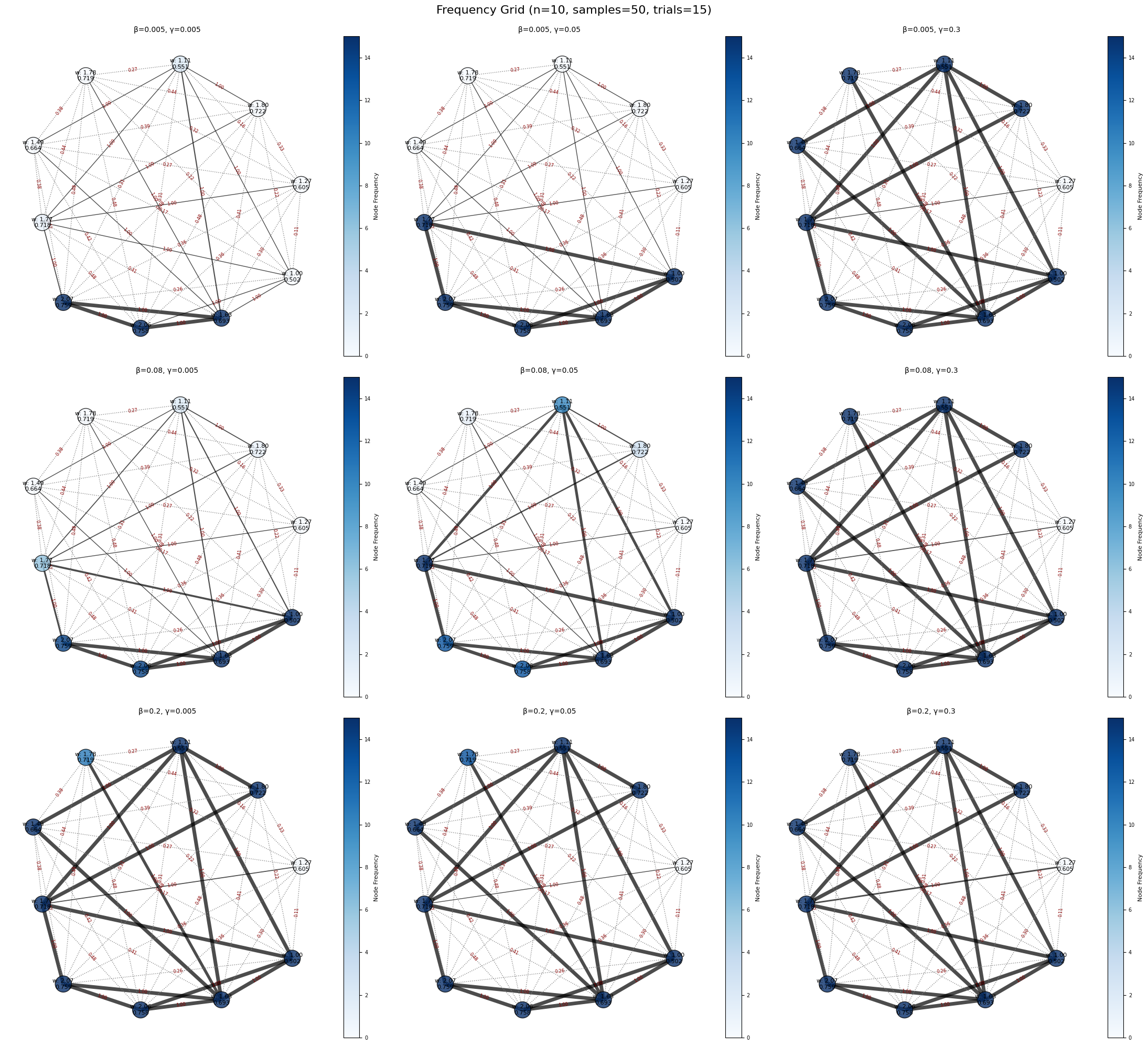}
    \caption{The frequency for solutions graphs under different $\beta$ and $\gamma$. The trials size is $15$ including $50$ samples within each trial. The node size is $10$. The parameter $\beta \in \{ 0.005, 0.08, 0.2 \}$ increases from top to bottom, and $\gamma \in \{ 0.005, 0.05, 0.3 \}$ increases from left to right. (Example~\ref{Frequency of solutions through several trials under different beta and gamma}).}
    \label{graph of frequency of solutions through several trials under different beta and gamma}
    \end{figure}

    \begin{figure}[H]
    \centering
    \includegraphics[scale = 0.32]{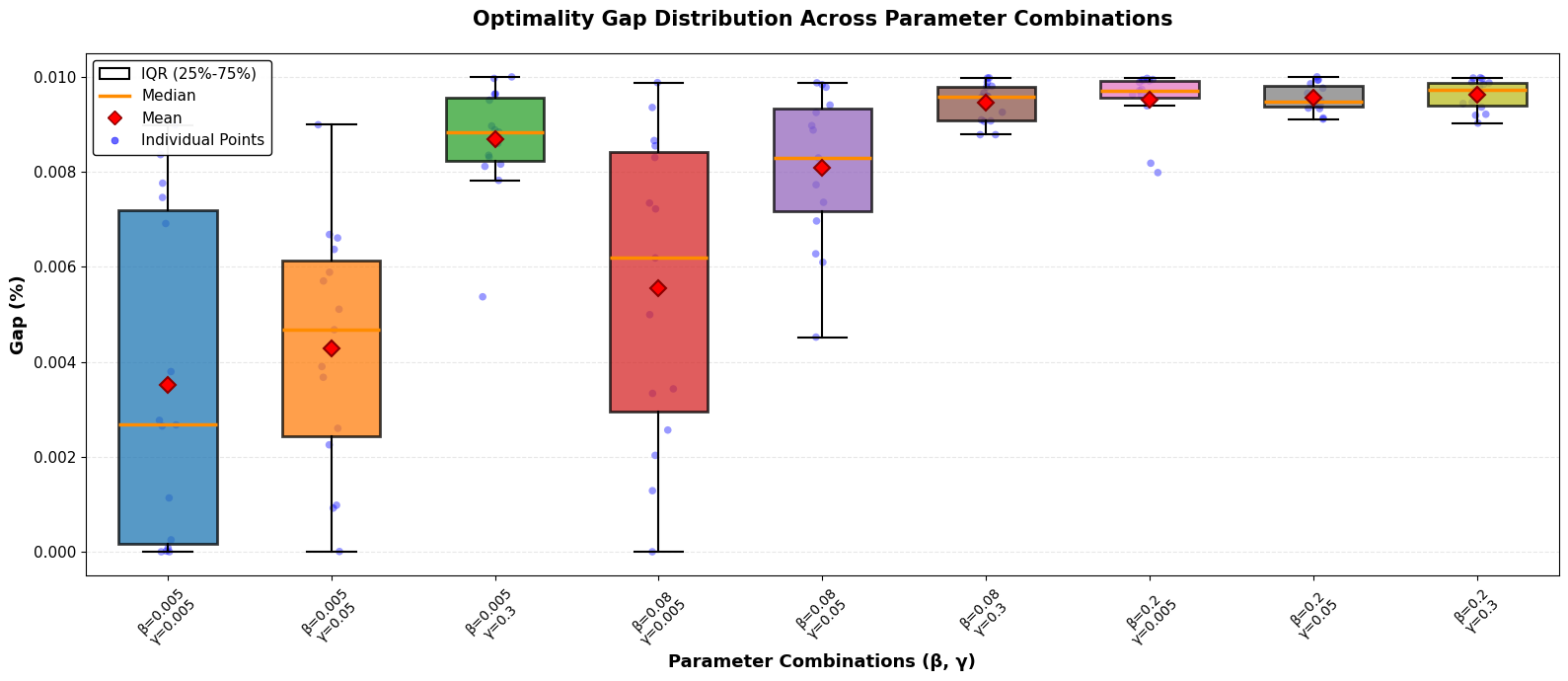}
    \caption{Box plots of gap under different $\beta$ and $\gamma$ (Example \ref{Frequency of solutions through several trials under different beta and gamma}).}
    \label{Box plots of best objective, best bound and gap}
    \end{figure}
\begin{table}[H]
\centering
\begin{tabular}{@{}llccc@{}}
\toprule
& & \multicolumn{3}{c}{\textbf{Parameter $\gamma$}} \\
\cmidrule(l){3-5}
\textbf{Parameter $\beta$} & \textbf{Quantities} & \textbf{0.005} & \textbf{0.05} & \textbf{0.3} \\
\midrule

\multirow{3}{*}{\large\textbf{0.005}}
& Best Objective & 0.1195 ($\pm$0.0007) & 0.2353 ($\pm$0.0003) & 0.4246 ($\pm$0.0002) \\
& Best Bound     & 0.1195 ($\pm$0.0007) & 0.2353 ($\pm$0.0003) & 0.4245 ($\pm$0.0002) \\
& Gap (\%)       & 0.0035 ($\pm$0.0033) & 0.0043 ($\pm$0.0025) & 0.0087 ($\pm$0.0011) \\
\midrule

\multirow{3}{*}{\large\textbf{0.08}}
& Best Objective & 0.3391 ($\pm$0.0103) & 0.3726 ($\pm$0.0088) & 0.4868 ($\pm$0.0055) \\
& Best Bound     & 0.3391 ($\pm$0.0103) & 0.3725 ($\pm$0.0088) & 0.4867 ($\pm$0.0055) \\
& Gap (\%)       & 0.0055 ($\pm$0.0031) & 0.0081 ($\pm$0.0015) & 0.0095 ($\pm$0.0004) \\
\midrule

\multirow{3}{*}{\large\textbf{0.2}}
& Best Objective & 0.5375 ($\pm$0.0186) & 0.5489 ($\pm$0.0178) & 0.6087 ($\pm$0.0142) \\
& Best Bound     & 0.5375 ($\pm$0.0186) & 0.5488 ($\pm$0.0178) & 0.6087 ($\pm$0.0142) \\
& Gap (\%)       & 0.0095 ($\pm$0.0006) & 0.0096 ($\pm$0.0003) & 0.0096 ($\pm$0.0003) \\
\bottomrule
\end{tabular}
\caption{Solver statistics for different parameters $\beta$ and $\gamma$. The format within each cell is mean ($\pm$standard deviation), rounded to four decimal places.}\label{tab:solver_stats_beta_gamma}
\end{table}

\section{Conclusion}

In this paper, we discussed the StQP under Wasserstein ambiguity. We demonstrated that this non-convex problem admits an exact reformulation as a deterministic StQP with a spectral regularization term, a result we further extended to decision-dependent ambiguity sets. Additionally, we derived finite-sample out-of-sample performance guarantees specifically for the random quadratic objective. Finally, we applied this framework to the maximum weighted clique problem, analyzing the spectral conditions for convexity and the resulting solution topology.


\section*{Acknowledgements}
\noindent The authors thank Ralf Werner, Daniel Kuhn and Bernardo Pagnoncelli for stimulating discussions and valuable suggestions. The authors are indebted to PGMO and HCODE for supporting several research visits between Vienna and Paris. The authors are indebted to VGSCO for financial support enabling presentation of this work at various major conferences: ECSO-CSM 2024, PGMODays (2024 and 2025), and ICCOPT 2025. The authors profited from the feedback of delegates to these meetings, which truly helped improving the quality of this paper.

\bibliography{refs}

\appendix
\section{Non-smooth norms may allow for a minimax theorem}\label{sec:AppendixA non-smooth norms}

\subsection{Maximum norm \texorpdfstring{$||\cdot||_\infty$}{||.||infty}}\label{sec:AppendixA1-l-infty-norm}
We first deal with the non-smooth maximum norm ${\norm \cdot}_\infty$. The following result can also be derived from~\cite[Theorems 6.14 and 6.46]{rockafellar1998variational}, but we provide a concise proof in order to be self-contained.

\begin{lemma}\label{normal cone of M set identical to R+}
Let  $\theta > 0$, $\norm{\cdot} = \norm{\cdot}_{\infty}$ and $\w\in \R^m$. Then, for $\z:=\w + \theta \e$ we have $\mathcal{N}_{B_\theta(\w)}(\z) = \mathbb R^m_+$.
\end{lemma}

\begin{proof}
Under $\norm{\cdot} = \norm{\cdot}_{\infty}$, 
\begin{equation}\label{normal cone of M set identical to R+ equ 2}
B_\theta(\w) = \{\y \in \mathbb R^m: w_i - \theta \leq y_i \leq w_i+\theta\,, \text{ for all } i\in [m]\}\,.
\end{equation}
First take any arbitrary $\v \in \mathcal{N}_{B_\theta(\w)}(\z) $; for $i\in[m]$ define $\z_i :=\z - 2\theta \e_i$, where $\e_i\in \R^m$ denotes the $i$-th standard basis vector ($i$-th column of the identity matrix of order $m$).
%
%
Obviously $\z_i \in B_\theta(\w)$ for all $i\in [m]$. Therefore, from the definition of the normal cone \eqref{normal cone of M set}
we obtain
    \begin{equation*}
    0 \leq \v^{\top} (\z - \z_i) =  2 \theta \e_i\T \v \,,
    \end{equation*}
    which implies  $\v \in \mathbb R^m_+$ and so  $\mathcal{N}_{B_\theta(\w)}(\z)\subseteq \mathbb{R}^{m}_{+}$. To establish the reverse inclusion, let $\v \in \mathbb R^m_+$ be arbitrary. From \eqref{normal cone of M set identical to R+ equ 2} it follows that
$$
\v^{\top}\z = \sum_{i=1}^mv_iz_i = \sum_{i=1}^mv_i(w_i + \theta) \geq \sum_{i=1}^mv_iy_i = \v^{\top}\y \quad \text{ for every } \y \in \mathcal{N}_{B_\theta(\w)}(\z)\,,
$$
which implies that $\mathbb R^m_+\subseteq\mathcal{N}_{B_\theta(\w)}(\z)$.
\end{proof}

\begin{corollary}\label{maximal element under infty norm} Under $\norm{\cdot} = \norm{\cdot}_{\infty}$ with $\h(\x) = \operatorname{svec}(\x\x^{\top})$ on the standard simplex $\Delta \subset \mathbb{R}^{n}$ and $\z:=\mathbb E_{\mathbb P^{\prime}}[\boldsymbol{\tilde \xi^{\prime}}] + \theta \e$, we have
    \begin{equation*}
        \lbrace \h(\x): \x \in \Delta \rbrace \subset \mathcal{N}_{\mathcal M_{\theta,p}(\mathbb P^{\prime})}(\z)\,.
    \end{equation*}
\end{corollary}
\begin{proof} Follows immediately from Lemma \ref{normal cone of M set identical to R+} since
$\h(\Delta) \subset \mathbb{R}^{m}_{+}$.
\end{proof}

\noindent This allows for the following reformulation of the \ref{DRStQP} under the maximum norm $\norm{\cdot} = \norm{\cdot}_{\infty}$.

\begin{theorem}\label{thm:DRStQP-maximum-norm}
Let $p\geq 1$, $\theta > 0$, $\norm{\cdot} = \norm{\cdot}_{\infty}$ be the maximum norm on $\mathbb R^m$ and $\mathbb P^{\prime} = \widehat{\mathbb P}_N$ be the empirical measure. Then~\eqref{DRStQP} is equivalent to a deterministic StQP, which is independent of the choice of $p$. In particular:
\begin{equation*}
  \underset{\x \in \Delta}{\vphantom{\sup}\inf} \, \underset{\PP \in \mathbb B_{\theta,p} (\widehat\PP_N)}{\sup} \,\mathbb E_{\PP}[\x^{\top}\widetilde\Qb\x] = \frac{\theta}{\sqrt{2}} + \min_{\x \in \Delta} \, \x^{\top}\left(\overline \Qb + \frac{\theta(\sqrt{2}-1)}{\sqrt{2}}\, \Ib\right)\x \,,
\end{equation*}
where $\overline{\Qb}:=\frac{1}{N}\sum_{i=1}^N\widehat{\Qb}_i$ denotes the sample mean, and for all $\x\in\Delta$, the worst-case distribution is given by the same $\mathbb P^* = T_{\#}\widehat{\mathbb P}_N$, with $T(\Zb) = \Zb + \tfrac{\theta}{\sqrt{2}}[\e\e^{\top} + (\sqrt{2}-1)\Ib]$. Therefore the minimax theorem holds: the minimax gap is zero.
\end{theorem}
\begin{proof}
  Abbreviate  the sample mean by $\boldsymbol{\overline{\xi}}  = \mathbb E_{\widehat{\mathbb P}_N}[\boldsymbol{\tilde{\xi}^{\prime}}]$. Rewrite
 $$
 \underset{\x \in \Delta}{\vphantom{\sup}\inf} \, \underset{\PP \in \mathbb B_{\theta,p} (\widehat\PP_N)}{\sup} \,\mathbb E_{\PP}[\x^{\top}\widetilde\Qb\x] =  \underset{\x \in \Delta}{\vphantom{\sup}\inf} \, \underset{\PP \in \mathbb B_{\theta,p} (\widehat\PP_N)}{\sup} \,\mathbb E_{\PP}[\h(\x)^{\top}\boldsymbol{\tilde \xi}]\,,
 $$
 with $\h(\x)=\operatorname{svec}(\x\x^{\top})$. By Corollary \ref{maximal element under infty norm}, for $\z:= \boldsymbol{\overline{\xi}} + \theta\e$ we have
\begin{equation*}
        \lbrace \h(\x): \x \in \Delta \rbrace \subset \mathcal{N}_{\mathcal M_{\theta,p}(\mathbb P^{\prime})}(\z)\,.
    \end{equation*}
By Theorem \ref{thm:maximal-element}, this implies that
$$
\underset{\x \in \Delta}{\vphantom{\sup}\inf} \, \underset{\PP \in \mathbb B_{\theta,p} (\widehat\PP_N)}{\sup} \,\mathbb E_{\PP}[\h(\x)^{\top}\boldsymbol{\tilde \xi}] = \inf_{\x \in \Delta} \h(\x)^{\top}\z = \inf_{\x \in \Delta} \h(\x)^{\top}(\boldsymbol{\overline{\xi}} + \theta\e)\,.
$$
Since the symmetric vectorization operator multiplies off-diagonal entries by $\sqrt{2}$, we obtain
$$
\inf_{\x \in \Delta} \h(\x)^{\top}(\boldsymbol{\overline{\xi}} + \theta\e) = \inf_{\x \in \Delta} \x^{\top}\left(\overline{\Qb} + \frac{\theta}{\sqrt{2}}(\e\e^{\top} + (\sqrt{2}-1)\Ib)\right)\x\,.
$$
Finally, the feasibility condition $\x^{\top}\e = 1$ yields the claim.
\end{proof}

\subsection{Manhattan norm \texorpdfstring{$||\cdot||_1$}{||.||1}}\label{sec:AppendixA2-l-1-norm}

\noindent Next, we will see that the result of Corollary \ref{maximal element under infty norm} is invalid for the situation under $\| \cdot \| := \| \cdot \|_{1}$, although $\| \cdot \|_{1}$ is non-smooth. To this end, we first provide a reformulation for the normal cone of the $\ell_1$-ball $B_\theta(\w) = \mathcal M_{\theta,p}(\mathbb P^{\prime})$ around $\w=\mathbb{E}_{\mathbb{P^{\prime}}}[\boldsymbol{\tilde \xi^{\prime}}] \in \R^m$ in the following result.  Recall that a vector $\v \in \mathbb R  ^m$ is called a subgradient of a convex function $f:\mathbb R^m \to \mathbb R$ at a point $\z\in \mathbb R^m$, if 
$$
f(\y) \geq f(\z) + \v^{\top}(\y-\z) \quad \text{ for all } \y \in \mathbb R\,,
$$
and that the set of all subgradients of $f$ at $\z$ is denoted as $\partial^c f(\z)$.

\begin{lemma}\label{reformulation for normal cone of M under L1 norm}
    Let $\| \cdot \| := \| \cdot \|_{1}$, and let the assumptions of Theorem \ref{expectation ball and wasserstein ball} hold. Denote by 
    $f_{i}(\s)=|s_{i}|$ for $\s \in \mathbb{R}^{m}$, $i\in [m]$,  and by $\partial^{c} |s_i|  $ the (convex) subgradient of $|\cdot|$ at $s_i$, calculated by
    \begin{equation*}
      \partial^{c} |s_i| =
        \left\{
            \begin{aligned}
            \{ 1 \}\, , \ \ \ \ & \ \ \text{if} \ \ s_i  > 0\, ,\\
            \{ -1 \}, \ \ & \ \  \text{if} \ \ s_i < 0\, ,\\
            \ [-1, 1]\, , & \ \  \text{if} \ \ s_i  = 0\, .
            \end{aligned}
        \right.
    \end{equation*}
    Then,
    \begin{equation*}
        \mathcal{N}_{B_\theta(\w)}(\z) = \mathrm{cone}(\partial^{c} |z_1-w_1 | \times \cdots \times \partial^{c} | z_m-w_m | ) \ \text{ for } \z \in \partial B_\theta(\w)\, ,
    \end{equation*}
    where $\times$ denotes the Cartesian product.
\end{lemma}
\begin{proof}
  Define $g : \mathbb{R}^{m} \to \mathbb{R}$ by
    \begin{equation*}
        g(\z) := \| \z - \w \|_{1} - \theta = \sum^{m}_{i = 1} f_i(\z-\w) - \theta\, ,\quad \z \in \mathbb{R}^{m}\, .
    \end{equation*}
    Since $ f_{i}$ and thus $g$ are convex  Lipschitz functions, they are regular~\cite[Proposition 2.3.6 (b)]{clarke1990optimization}). Thus, using~\cite[Theorem~2.3.10 and Example~2.1.3]{clarke1990optimization}, 
    in conjunction with~\cite[Proposition 2.4.2, Proposition 2.3.3 and the associated Corollary 3]{clarke1990optimization}, we can reformulate the normal cone of $B_{\theta}(\w)$ as follows:
    \begin{equation*}
  \mathcal{N}_{B_\theta(\w)}(\z) = \text{cl}(\lambda \partial^{c} g(\z): \lambda \geq 0 ) = \text{cone}(\partial^{c} g(\z))= \text{cone}(\partial^{c}|z_1-w_1 | \times   \cdots \times \partial^{c} | z_m-w_m|  )
    \end{equation*}
     for all $\z \in \partial B_\theta(\w)$, which completes the proof.
\end{proof}
\begin{corollary}
    Let $\|\cdot\| := \|\cdot\|_{1}$ and let the assumptions of Theorem \ref{expectation ball and wasserstein ball} hold. Define $\h(\x) = \operatorname{svec}(\x \x^{\top})$ on the standard simplex $\Delta \subset \mathbb{R}^{n}$ and let $\w = \mathbb E_{\mathbb P^{\prime}}[\boldsymbol{\tilde {\xi}^{\prime}}]\in \mathbb R^m$. If $m > 1$, then there is no $\z \in B_{\theta}(\w)$ such that
\begin{equation}\label{inclusion relation of M}
        \lbrace \h(\x): \x \in \Delta \rbrace \subset \mathcal{N}_{B_\theta(\w)}(\z).
    \end{equation}
    However, if $m = 1$, then $z := w + \theta$ satisfies \eqref{inclusion relation of M}.
\end{corollary}

\begin{proof} For $m = 1$,  the norms $\norm\cdot_1 = {\norm\cdot}_\infty$ and their ball coincide  to 
    \begin{equation*}
        B_\theta(w) = \lbrace y \in \mathbb{R}: | y -w | \leq \theta \rbrace = [ \mathbb{E}_{\mathbb{P'}}[\tilde{\xi}^{\prime}] - \theta, \mathbb{E}_{\mathbb{P'}}[\tilde{\xi}^{\prime}] + \theta]\,.
    \end{equation*}
   Thus $z := \mathbb{E}_{\mathbb{P'}}[\tilde{\xi}^{\prime}] + \theta$ satisfies \eqref{inclusion relation of M} by application of Corollary~\ref{maximal element under infty norm}.
   Now assume $m > 1$ and put $\mathcal{I}:=[m]$. Using a proof by contradiction, we assume that there exists $\z \in B_\theta(\w)$ (in fact, $\z \in \partial B_\theta(\w)$, otherwise $\h(\Delta)\nsubseteq  \lbrace {\boldsymbol{0}} \rbrace = \mathcal{N}_{B_\theta(\w)}(\z)$ is already a contradiction) satisfying \eqref{inclusion relation of M}. Define index subsets as follows:
    \begin{equation*}
        \mathcal{I}_{+} := \lbrace k \in \mathcal{I}: z_{k} > w_k \rbrace \, , \ \mathcal{I}_{-} := \lbrace k \in \mathcal{I}: z_{k} < w_k \rbrace\, , \ \mathcal{I}_{0} := \lbrace k \in \mathcal{I}: z_{k} =w_k \rbrace\, .
    \end{equation*}
    Using Lemma \ref{reformulation for normal cone of M under L1 norm}, we can obtain
\begin{equation}\label{normal cone of M for n > 1}
  \mathcal{N}_{B_\theta(\w)}(\z) = \text{cone} (G)
    \end{equation}
    with $$
    G:=\Biggl\{ \sum_{k\in  \mathcal{I}_{+} }\e_k - \sum_{k\in \mathcal{I}_{-}}\e_k + \sum_{k\in \mathcal{I}_{0}}  \eta_k\e_k \in \mathbb{R}^{m}: \eta_k \in [-1, 1]\mbox{ for all }k\in \mathcal{I}_{0}\Biggr\} .
    $$
Define $\u:=\tfrac{1}{n^2}(1,\dots,1, \sqrt{2},\dots,\sqrt{2})^{\top} \in \mathbb R^{m}$  and $\u_i:=\e_i \in \mathbb R^m$, for $i=1,\dots,n$ and observe that $\u, \u_1,\dots,\u_n$ correspond to feasible solutions for the StQP(``the equal weight" and the ``$n$ basis vectors", respectively):
$$
\u, \u_1,\dots,\u_n \in \{\h(\x):\x \in \Delta\}\,.
$$
We claim that $\mathcal{I}_{-} = \emptyset$. For the sake of contradiction, suppose that there exists $j \in \mathcal{I}_{-}$. Using~\eqref{normal cone of M for n > 1}, we can see that for any $\y\in \mathcal{N}_{B_\theta(\w)}(\z)$, there exists $\lambda \geq 0$ and $\v \in G$ such that $\y = \lambda \v$. In particular, this means
    \begin{equation*}
        y_{j} = \lambda v_{j} = -\lambda \leq 0\,,
    \end{equation*}
i.e. every $\y\in \mathcal{N}_{B_\theta(\w)}(\z)$ has a nonnegative entry. However, $\u \in  \{\h(\x):\x \in \Delta\}$ has only strictly positive entries. This means that $\u \notin \mathcal{N}_{B_\theta(\w)}(\z)$ which contradicts~\eqref{inclusion relation of M}.\\
Next, we claim that $\mathcal{I}_{+} = \emptyset$. For that, assume that there exists $j \in \mathcal{I}_{+}$. Since 
$$
\{\u_i\}_{i=1}^n \subset \{\h(\x):\x \in \Delta\}
$$
and by assumption $\{\h(\x):\x \in \Delta\} \subset \mathcal{N}_{B_\theta(\w)}(\z)$, we have $\{\u_i\}_{i=1}^n\subset \mathcal{N}_{B_\theta(\w)}(\z)$. Thus, using~\eqref{normal cone of M for n > 1}, we can see that there exists $\lambda_{i} \geq 0$ and $\v_{i} \in G$ such that
    \begin{equation}\label{counter example equ 1}
        \u_{i} = \lambda_{i} \v_{i} \ \text{ and } \ u_{i,j} = \lambda_{i} v_{i,j} = \lambda_{i} \geq 0.
    \end{equation}
    Since $u_{i,j} = 0$, using \eqref{counter example equ 1}, we can see $\lambda_{i} = 0$ and $\u_{i} = \lambda_{i} \v_{i} = \boldsymbol{0}$, which contradicts the definition of $\u_{i}$.\\
    We have proved $\mathcal{I}_{-} = \mathcal{I}_{+} = \emptyset$, which implies that $\mathcal{I}_{0} = \mathcal{I}$. Based on the definition of $\mathcal{I}_{0}$, we can see that
    \begin{equation*}
        \|\z - \w\|_{1} = \| \boldsymbol{0}\|_{1} = 0 < \theta\, .
    \end{equation*}
    But then $\z$ is an interior point of $B_\theta(\w)$, which contradicts $\z \in \partial B_\theta(\w)$. Thus, we have finished the proof.
\end{proof}

\section{Uniform versions of performance guarantee}
\label{sec:AppendixB about proof of section out-of-sample performance guarantees}

In this section we will 
continue the discussion of Section~\ref{sec:4 out-of-sample performance guarantees}, now providing uniform out-of-sample performance guarantees for sub-Gaussian and sub-exponential distributions. In Proposition~\ref{lem:centered Gaussian is also sub-Gaussian} we showed that if $\widetilde \Gb \sim \operatorname{GOE}(n)$ is a random symmetric matrix from the Gaussian Orthogonal Ensemble, then $\boldsymbol{\tilde \xi} = \operatorname{svec}(\widetilde \Gb)\sim \mathcal N_m(\boldsymbol{0},2\Ib)$ and it is in particular sub-Gaussian. Recall that every sub-Gaussian random vector is automatically also sub-exponential. 

\noindent {In Proposition~\ref{prop:Wishart are sub-exponential but not sub-Gaussian} we showed that if $\widetilde \Wb \sim \mathcal W_n(\Ib,n)$ is a random symmetric matrix from the Wishart Ensemble with identity covariance matrix and $n$ degrees of freedom, then $\boldsymbol{\tilde \xi} = \operatorname{svec}(\widetilde \Wb)$ is sub-exponential but not sub-Gaussian.}

\begin{theorem}[Hoeffding's inequality using a covering number based method for sub-Gaussian random vectors]\label{Hoeffding inequality for sub-gaussian vectors}
    Let $\lbrace \boldsymbol{\tilde \xi}_{i} \rbrace^{N}_{i=1} \subset \mathbb{R}^{m}$ be independent sub-Gaussian random vectors. 
    Then, for any $\beta \in (0, 1)$, with probability at least $1 - \beta$,
    \begin{equation*}
        \left\| \frac{1}{N} \sum^{N}_{i = 1} \boldsymbol{\tilde \xi}_{i} \right\|_{2} \leq C K \left( \sqrt{\frac{m}{N}} + \sqrt{\frac{\log(2 / \beta)}{N}} \right),
    \end{equation*}
    where $C > 0$ is an absolute constant and $K := \max_{i} \| \boldsymbol{\tilde \xi}_{i}\|_{\psi_{2}}$.
\end{theorem}
\begin{proof}
    Fix $\y \in \mathbb{S}^{m-1}$ arbitrarily. We can see that $\y^{\top}\boldsymbol{\tilde \xi}_{i}$ are independent random variables with mean zero satisfying 
    \begin{equation*}
        \| \y^{\top}\boldsymbol{\tilde \xi}_{i}\|_{\psi_{2}} \leq \| \boldsymbol{\tilde \xi}_{i}\|_{\psi_{2}} \leq K.
    \end{equation*}
    Using the general Hoeffding's inequality \cite[Theorem 2.6.3]{vershynin2018high}, we obtain \begin{equation}\label{Hoeffding's inequality for sub-gaussian equ 1}
        \mathbb{P}_{\rm true}^N\left[ \left| \frac{1}{N} \sum^{N}_{i = 1} \y^{\top}\boldsymbol{\tilde \xi}_{i} \right| \geq r \right] = \mathbb{P}_{\rm true}^N\left[ \left| \sum^{N}_{i = 1} \y^{\top}\boldsymbol{\tilde \xi}_{i} \right| \geq N r \right] \leq 2 \exp \left( - \frac{c N r^{2}}{K^{2}}\right)
    \end{equation}
    with an absolute constant $c>0$. Let $\mathcal N$ be an $\frac{1}{2}$-net (see \cite[Definition 4.2.1]{vershynin2018high}) of the sphere ${\mathbb S}^{m-1}$. By definition, we have $\mathcal N \subset \mathbb S^{m-1}$. By~\cite[Corollary 4.2.13]{vershynin2018high}, we know that $2^{m} \leq |\mathcal{N}| \leq 5^{m}$. Using~\cite[Exercise 4.4.2]{vershynin2018high}, we can get
\begin{equation}\label{Hoeffding's inequality for sub-gaussian equ 2}
        \left\| \frac{1}{N} \sum^{N}_{i = 1} \boldsymbol{\tilde \xi}_i\right\|_{2} \leq \frac{2}{N} \sup_{\v \in \mathcal{N}} \sum^{N}_{i = 1}\v^{\top}\boldsymbol{\tilde \xi}_i\,.
    \end{equation}
    Thus, for any $r > 0$, using \eqref{Hoeffding's inequality for sub-gaussian equ 1} and \eqref{Hoeffding's inequality for sub-gaussian equ 2}, we can obtain
\begin{equation}\label{Hoeffding's inequality for sub-gaussian equ 3}
        \begin{aligned}
            \mathbb{P}_{\rm true}^N&\left[ \left\| \frac{1}{N} \sum^{N}_{i = 1} \boldsymbol{\tilde \xi}_i \right\|_{2} \geq r \right]\leq \mathbb{P}_{\rm true}^N \left[ \sup_{\v \in \mathcal{N}} \sum^{N}_{i = 1}\v^{\top}\boldsymbol{\tilde \xi}_i \geq \frac{Nr}{2} \right] = \mathbb{P}_{\rm true}^N \left[ \bigcup_{\v \in \mathcal{N}}\Biggl\{\sum^{N}_{i = 1}\v^{\top}\boldsymbol{\tilde \xi}_i \geq \frac{Nr}{2} \Biggr\}\right] \\
            &\leq \sum_{\v \in \mathcal{N}} \mathbb{P}_{\rm true}^N\left[ \sum^{N}_{i = 1}\v^{\top}\boldsymbol{\tilde \xi}_i \geq \frac{Nr}{2} \right] \leq 5^{m} \cdot 2 \exp\left( - \frac{c' N r^{2}}{K^{2}}\right),
        \end{aligned}
    \end{equation}
    where $c' = c / 4$ is an absolute constant. Take $\beta \in (0, 1)$ arbitrarily. Letting the right hand side of \eqref{Hoeffding's inequality for sub-gaussian equ 3} equal to $\beta$, after simplification, we can yield
    \begin{equation*}
        r = \frac{K}{\sqrt{c'}} \sqrt{\frac{m \log 5 + \log(2/\beta)}{N}}.
    \end{equation*}
    Thus, we can obtain
    \begin{equation*}
        \mathbb{P}_{\rm true}^N\left[ \left\| \frac{1}{N} \sum^{N}_{i = 1} \boldsymbol{\tilde \xi}_i\right\|_{2} \leq C K \left( \sqrt{\frac{m}{N}} + \sqrt{\frac{\log(2 / \beta)}{N}} \right) \right] \geq 1 - \beta
    \end{equation*}
    with an absolute constant $C$.
\end{proof}


\noindent In the following, we use two further methods to derive Bernstein-type inequalities for sub-exponential random vectors. Recall that the experiments in Section~\ref{sec:5 numerical simulations} satisfy the sub-exponentiality assumption by~Propositions~\ref{lem:centered Gaussian is also sub-Gaussian} and~\ref{prop:Wishart are sub-exponential but not sub-Gaussian}.

\begin{theorem}[Bernstein's inequality using a covering number based method for sub-exponential random vectors]\label{bernstein inequality under sub-exponential based on covering number method}
    Let $\lbrace \boldsymbol{\tilde \xi}_i  \rbrace^{N}_{i=1} \subset \mathbb{R}^{m}$ be independent sub-exponential random vectors. 
    Then, for any $s \geq 0$, with probability at least $1 - 2 \exp(-s)$,
    \begin{equation*}
        \left\| \frac{1}{N} \sum^{N}_{i = 1} \boldsymbol{\tilde \xi}_i \right\|_{2} \leq C K \left( \sqrt{\frac{m + s}{N}} + \frac{m + s}{N} \right)
    \end{equation*}
    where $C > 0$ is an absolute constant and $K := \max_{i} \| \boldsymbol{\tilde \xi}_i\|_{\psi_{1}}$.
\end{theorem}
\begin{proof}
    Using \cite[ Corollary 2.8.3]{vershynin2018high} and a similar proof to Theorem \ref{Hoeffding inequality for sub-gaussian vectors}, we obtain  \begin{equation}\label{Bernstein's inequality for sub-gaussian equ 1}
\mathbb{P}_{\rm true}^N\left[\left\| \frac{1}{N} \sum^{N}_{i = 1} \boldsymbol{\tilde \xi}_i \right\|_{2} \geq r \right] \leq 5^{m} \cdot 2 \exp\left( - c N \min \biggl\{\frac{(r/2)^{2}}{K^{2}}, \frac{(r/2)}{K} \biggr\}\right)
    \end{equation}
    with an absolute constant $c > 0$. Take any arbitrary $s > 0$. Letting the right hand side of \eqref{Bernstein's inequality for sub-gaussian equ 1} be less or equal than $2 \exp(-s)$, we get
    \begin{equation*}
        \min \biggl\{\frac{(r/2)^{2}}{K^{2}}, \frac{(r/2)}{K} \biggr\}\geq \frac{m \log 5 + s}{c N }.
    \end{equation*}
We can see that there exists an absolute constant $c' > 0$ satisfying $c'(m + s) \geq (m \log 5 + s)/c$. Take
    \begin{equation*}
        r := 2\max\lbrace \sqrt{c'}, c' \rbrace \cdot K \left( \sqrt{\frac{m + s}{N}} + \frac{m + s}{N} \right).
    \end{equation*}
    Finally, it can be verified that
    \begin{equation*}
        \min \biggl\{\frac{(r/2)^{2}}{K^{2}}, \frac{(r/2)}{K} \biggr\}\geq \frac{c'(m + s)}{N}\,.
    \end{equation*}
\end{proof}

\begin{theorem}[Bernstein's inequality using a martingale based method for sub-exponential random vectors]\label{bernstein inequality under sub-exponential based on martingale method}
    Let $\lbrace \boldsymbol{\tilde \xi}_i  \rbrace^{N}_{i=1} \subset \mathbb{R}^{m}$ be independent sub-exponential random vectors. 
    Assume that there exists $R > 0$ such that
    \begin{equation*}
        \| \| \boldsymbol{\tilde \xi}_i \|_{2} \|_{\psi_{1}} \leq R\, \quad \text{for all } i \in [N]\,.
    \end{equation*}
Then, for any $r > 0$,   \begin{equation}\label{bernstein's inequality upper bound}
  \mathbb{P}_{\rm true}^N\left[\left\| \frac{1}{N} \sum^{N}_{i = 1} \boldsymbol{\tilde \xi}_i \right\|_{2} \geq r \right] \leq 2 \exp\left( - \frac{N r^{2}}{8R^{2} + R r} \right).
    \end{equation}
Let $\beta \in (0, 1)$, then with probability at least $1 - \beta$,
    \begin{equation*}
        \left\| \frac{1}{N} \sum^{N}_{i = 1} \boldsymbol{\tilde \xi}_i \right\|_{2} \leq 2\sqrt{2}R \sqrt{\frac{2 \log(2/\beta)}{N}} + \frac{R\log(2/\beta)}{N}.
    \end{equation*}
\end{theorem}
\begin{proof} First recall that if $\boldsymbol{\tilde \xi}\in \mathbb R^m$ is a sub-exponential random vector with parameter $\lambda$, this means that $\y^{\top}\boldsymbol{\tilde \xi}$ is a sub-exponential random variable with parameter $\lambda$, for every $\y \in \mathbb S^{m-1}$. This means that by definition of the Euclidean norm (whose dual is also the Euclidean norm!),
$$
\|\boldsymbol{\tilde \xi}\|_2=\sup_{\y \in \mathbb S^{m-1}} \y^{\top}\boldsymbol{\tilde \xi}
$$
is also a sub-exponential random variable with parameter $\lambda$. Using Definition \ref{definition of sub-exponential}, for any $k \in \mathbb{N}$ and $i \in [N]$, we know that \begin{equation*}\mathbb{E}_{\mathbb P_{\rm true}^N}\left[  \frac{\left(\frac{\|\boldsymbol{\tilde \xi}_i \|_{2}}{\| \| \boldsymbol{\tilde \xi}_i \|_{2} \|_{\psi_{1}}}\right)^{k}}{k!} \right] \leq \mathbb{E}_{\mathbb P_{\rm true}^N}\left[\sum_{k=0}^{\infty}\frac{\left(\frac{\|\boldsymbol{\tilde \xi}_i \|_{2}}{\| \| \boldsymbol{\tilde \xi}_i \|_{2} \|_{\psi_{1}}}\right)^{k}}{k!} \right] = \mathbb{E}_{\mathbb P_{\rm true}^N}\left[ \exp\left( \frac{\| \boldsymbol{\tilde \xi}_i \|_{2}}{\| \| \boldsymbol{\tilde \xi}_i \|_{2} \|_{\psi_{1}}} \right) \right] \leq 2\,.
    \end{equation*}
The last step can be explained in the following way. The function $t\mapsto \mathbb E_{\mathbb P_{\rm true}}[\exp(|\tilde \xi|/t)]$ is decreasing in $t$, which means
$$
\mathbb E_{\mathbb P_{\rm true}}[\exp(|\tilde \xi|/R)] \leq \mathbb E_{\mathbb P_{\rm true}}[\exp(|\tilde \xi|/t)]\,, \quad \text{for all } t \leq R\,.
$$
\noindent Now if $\|\tilde \xi\|_{\Psi_1} \leq R$, then there exists a $t^*\in (0,R]$ such that
$$
\mathbb E_{\mathbb P_{\rm true}}[\exp(|\tilde \xi|/R)] \leq \mathbb E_{\mathbb P_{\rm true}}[\exp(|\tilde \xi|/t^*)]\leq 2
$$
hence in particular, $\|\tilde \xi\|_{\Psi_1} \leq R$ implies $
\mathbb E_{\mathbb P_{\rm true}}[\exp(|\tilde \xi|/R)]\leq 2$. Therefore, since 
\begin{equation*}\mathbb{E}_{\mathbb P_{\rm true}}\left[  \frac{\left(\frac{\|\boldsymbol{\tilde \xi}_i \|_{2}}{\| \| \boldsymbol{\tilde \xi}_i \|_{2} \|_{\psi_{1}}}\right)^{k}}{k!} \right] \leq 2\,,
    \end{equation*}
by rearranging and building the sum over all $i\in [N]$ we obtain
\begin{equation*}
\sum_{i=1}^N\mathbb{E}_{\mathbb P_{\rm true}^N}[\|\boldsymbol{\tilde \xi}_i \|_{2}^{k}] \leq 2k!\sum_{i=1}^{N}\mathbb{E}_{\mathbb P_{\rm true}^N}[\| \| \boldsymbol{\tilde \xi}_i \|_{2} \|_{\psi_{1}}^{k}] \leq 2k!\cdot N \cdot R^k\,.
\end{equation*}

\noindent Letting $C := R$, $B := 2 \sqrt{N} R$, and using \cite[Theorem 3.3]{pinelis1994optimum}, we obtain for any $s > 0$,  
\begin{equation*}
\mathbb{P}_{\rm true}^N\left[\max_{j \in [N]}\left\| \sum_{i = 1}^j \boldsymbol{\tilde \xi}_i \right\|_{2} \geq s \right]\leq 2 \exp\left( - \frac{s^{2}}{B^{2} + B \sqrt{B^{2} + 2 C s}} \right)\,.
    \end{equation*}
Since 
$$
\left\| \sum_{i = 1}^N\boldsymbol{\tilde \xi}_i \right\|_{2}  \leq \max_{j \in [N]}\left\| \sum_{i = 1}^j \boldsymbol{\tilde \xi}_i \right\|_{2}\,, 
$$
we obtain by the monotonicity of $\mathbb P_{\rm true}^N$
\begin{equation*}
\mathbb{P}_{\rm true}^N\left[\left\| \sum_{i = 1}^N \boldsymbol{\tilde \xi}_i \right\|_{2} \geq s \right]\leq 2 \exp\left( - \frac{s^{2}}{B^{2} + B \sqrt{B^{2} + 2 C s}} \right)\,.
    \end{equation*}
Now we can use the fact that 
 \begin{equation*}
        \sqrt{1 + x} \leq 1 + \frac{x}{2}\, \quad \text{for all } x \geq 0,
    \end{equation*}
which implies 
 \begin{equation*}
        B \sqrt{B^{2} + 2C s} = B^{2} \sqrt{1 + \frac{2 C s}{B^{2}}} \leq B^{2} \left( 1 + \frac{C s}{B^{2}} \right) = B^{2} + C s\,,
    \end{equation*}
therefore 
\begin{equation}\label{bernstein inequality for sub-gaussian equ 5}
2 \exp\left( - \frac{s^{2}}{B^{2} + B \sqrt{B^{2} + 2 C s}} \right)\leq 2 \exp\left( - \frac{s^{2}}{2B^{2} + C s} \right) = 2 \exp\left( - \frac{s^{2}}{8NR^{2} + R s} \right)\,.
    \end{equation}
Substituting $s := N r$ with $r > 0$ into \eqref{bernstein inequality for sub-gaussian equ 5}, we can obtain \eqref{bernstein's inequality upper bound}. Letting the right hand side of \eqref{bernstein's inequality upper bound} equal to $\beta \in (0, 1)$, we get
\begin{equation}\label{bernstein inequality for sub-gaussian equ 6}
        \frac{N r^{2}}{8R^{2} + R r} = \log(2/\beta)\,.
    \end{equation}
By the quadratic formula, this implies that
\begin{equation*}
        r = \frac{R \log(2/\beta) + \sqrt{R^{2}(\log(2/\beta))^{2} + 32 N R^{2} \log(2/\beta)}}{2N} \leq \frac{R}{N} \log(2/\beta) + 2 \sqrt{2} R \sqrt{\frac{\log(2/\beta)}{N}}\,,
    \end{equation*}
where in the last step we have used the fact that
$$
\sqrt{x + y} \leq \sqrt{x} + \sqrt{y}\, \quad \text{for all } x,y \geq 0\,.
$$
The claimed result follows.\end{proof}

\noindent In Theorem~\ref{thm:sub-exponential-random-vectors-pointwise} we claim that for any $\y \in \mathbb{S}^{n-1}$ and $\beta \in (0, 1)$, with probability at least $1 - \beta$, we can guarantee
\begin{equation}\label{aim out-of-sample guarantee 1}
    \y^{\top} \left( \mathbb{E}_{\widehat{\mathbb{P}}_{N}}[\widetilde\Qb] - \mathbb{E}_{\mathbb P_{\rm true}}[\widetilde\Qb] \right) \y \leq \theta_N(\beta),
\end{equation}
with some $\theta_N(\beta) > 0$. In fact, we will focus on a stronger form to guarantee \eqref{aim out-of-sample guarantee 1}, namely
\begin{equation*}
    \sup_{\y \in \mathbb{S}^{n-1}} \biggl\{\y^{\top} \left( \mathbb{E}_{\widehat{\mathbb{P}}_{N}}[\widetilde\Qb] - \mathbb{E}_{\mathbb P_{\rm true}}[\widetilde\Qb] \right) \y \biggr\} \leq \theta_N(\beta)
\end{equation*}
with probability at least $1-\beta$.

\begin{lemma}\label{transition for high probability bound}
Let  $ \widehat{\boldsymbol{\xi}}_i = \operatorname{svec}(\widehat \Qb_i)$ be i.i.d. sub-exponential random vectors generated by observations of random $\widehat \Qb_i\in {\mathcal S}^n$, $i\in [N]$. Let $\overline{\boldsymbol{\xi}}:=\frac{1}{N}\sum_{i=1}^N\widehat{\boldsymbol{\xi}}_i$ denote the sample mean, let $\boldsymbol{\mu}_{\rm true} := \mathbb E_{\mathbb P_{\rm true}}[\boldsymbol{\tilde \xi}]$ denote the true mean, and set $\widehat{\boldsymbol{\zeta}}_i = \widehat{\boldsymbol{\xi}}_i - \boldsymbol{\mu}_{\rm true}$, so that $\overline{\boldsymbol{\xi}}- \boldsymbol{\mu}_{\rm true}= \frac{1}{N} \sum^{N}_{i = 1} \widehat{\boldsymbol{\zeta}}_i $. Then,
    \begin{equation*}
        \sup_{\y \in \mathbb{S}^{n-1}}  \biggl\{ \y^{\top} \left( \mathbb{E}_{\widehat{\mathbb{P}}_{N}}[\widetilde\Qb] - \mathbb{E}_{\mathbb P_{\rm true}}[\widetilde\Qb] \right) \y  \biggr\}\leq \| \overline{\boldsymbol{\xi}}- \boldsymbol{\mu}_{\rm true}\|_{2} = \left\| \frac{1}{N} \sum^{N}_{i = 1} \widehat{\boldsymbol{\zeta}}_i \right\|_{2}.
    \end{equation*}
\end{lemma}
\begin{proof}
    For any $\y \in \mathbb{S}^{n - 1}$ and $\Ab = \mathbb{E}_{\widehat{\mathbb{P}}_{N}}[\widetilde\Qb] - \mathbb{E}_{\mathbb P_{\rm true}}[\widetilde\Qb] \in \mathcal S^{n}$, we know by Cauchy-Schwarz in ${\mathcal S}^n$ that
   $$
    \y^{\top} \Ab \y = \langle \y \y^{\top} ,   \Ab \rangle_F \leq {\norm{\y\y\T}}_F
    {\norm{\Ab}}_F\, .
   $$
   But $\y \in \mathbb{S}^{n-1}$ implies ${\norm{\y\y\T}}_F =\norm{\y}^2_2 = 1$.
  Thus 
 \begin{equation*}
        \sup_{\y \in \mathbb{S}^{n-1}}  \biggl\{ \y^{\top} \left( \mathbb{E}_{\widehat{\mathbb{P}}_{N}}[\widetilde\Qb] - \mathbb{E}_{\mathbb P_{\rm true}}[\widetilde\Qb] \right) \y  \biggr\} 
        \leq {\norm{\Ab}}_F = 
       {\norm{ \operatorname{svec}(\Ab)}}_2 = \| \overline{\boldsymbol{\xi}}- \boldsymbol{\mu}_{\rm true}\|_{2}\, ,
    \end{equation*}
  by the isometry property of the $\operatorname{svec}(\cdot) $ operator.
\end{proof}
\begin{theorem}[Uniform result for sub-exponential random instances]\label{thm:sub-exponential-random-vectors-uniform}
Let $ \widehat{\boldsymbol{\xi}}_i = \operatorname{svec}(\widehat \Qb_i)$ be i.i.d. sub-exponential random vectors generated by observations of random $\widehat \Qb_i\in {\mathcal S}^n$, $i\in [N]$. Then, for any $\beta \in (0, 1)$,
  \begin{equation*}  
        \mathbb P_{\rm true}^N\left[  \sup_{\y \in \mathbb{S}^{n-1}}  \biggl\{\y^{\top}\left( \mathbb{E}_{\widehat{\mathbb{P}}_{N}}[\widetilde\Qb] - \mathbb{E}_{\mathbb P_{\rm true}}[\widetilde\Qb] \right)\y \biggr\} \leq C K \left( \sqrt{\frac{m + \log(2/\beta)}{N}} + \frac{m + \log(2/\beta)}{N}  \right) \right] \geq 1 - \beta,
    \end{equation*}
    where $C > 0$ is an absolute constant, $K := \max_{i} \| \widehat{\boldsymbol{\xi}}_i \|_{\psi_{1}}$, and $\| \cdot \|_{\psi_{1}}$ represents the sub-exponential norm defined in Definition \ref{definition of sub-exponential}. Furthermore, assume that there exists an $R > 0$ such that
    \begin{equation*}
        \| \|\widehat{\boldsymbol{\xi}}_i\|_{2} \|_{\psi_{1}} \leq R\, \quad \text{for all } i \in [N]\,.
    \end{equation*}
    Then,
    \begin{equation*}
        \mathbb P_{\rm true}^N\left[  \sup_{\y \in \mathbb{S}^{n-1}}  \biggl\{\y^{\top}\left( \mathbb{E}_{\widehat{\mathbb{P}}_{N}}[\widetilde\Qb] - \mathbb{E}_{\mathbb P_{\rm true}}[\widetilde\Qb] \right)\y  \biggr\} \leq 2\sqrt{2}R \sqrt{\frac{2 \log(2/\beta)}{N}} + \frac{R\log(2/\beta)}{N} \right] \geq 1 - \beta.
    \end{equation*}
    Moreover, if $\{\widehat{\boldsymbol{\xi}}_i\}_{i=1}^N\subset \mathbb R^m$ are i.i.d. sub-Gaussian random vectors, we have
    \begin{equation*}
        \mathbb P_{\rm true}^N\left[  \sup_{\y \in \mathbb{S}^{n-1}} \biggl\{ \y^{\top}\left( \mathbb{E}_{\widehat{\mathbb{P}}_{N}}[\widetilde\Qb] - \mathbb{E}_{\mathbb P_{\rm true}}[\widetilde\Qb] \right)\y \biggr\} \leq C K \left( \sqrt{\frac{m}{N}} + \sqrt{\frac{\log(2 / \beta)}{N}} \right) \right] \geq 1 - \beta,
    \end{equation*}
    where $C > 0$ is an absolute constant, $K := \max_{i} \| \widehat{\boldsymbol{\xi}}_i \|_{\psi_{2}}$, and $\| \cdot \|_{\psi_{2}}$ represents the sub-Gaussian norm defined in Definition \ref{definition of sub-gaussian}.
\end{theorem}
\begin{proof}
Using Lemma~\ref{transition for high probability bound}, we know that
    \begin{equation*}
        \sup_{\y \in \mathbb{S}^{n-1}} \biggl\{\y^{\top} \left( \mathbb{E}_{\widehat{\mathbb{P}}_{N}}[\widetilde\Qb] - \mathbb{E}_{\mathbb P_{\rm true}}[\widetilde\Qb] \right) \y \biggr\} \leq  \left\| \frac{1}{N} \sum^{N}_{i = 1} \widehat{\boldsymbol{\zeta}}_i\right\|_{2}.
    \end{equation*}
The first and the second claims follow immediately from Theorem \ref{bernstein inequality under sub-exponential based on covering number method} and Theorem \ref{bernstein inequality under sub-exponential based on martingale method}, respectively. Based on an analogue of the above proof, we immediately obtain the third claim using Lemma \ref{transition for high probability bound} and Theorem \ref{Hoeffding inequality for sub-gaussian vectors}.
\end{proof}

\end{document}